\documentclass{article}
\usepackage{url}
\usepackage{graphicx}
\graphicspath{{./figures/}}
\usepackage[tbtags]{amsmath}
\usepackage{amsfonts}
\usepackage{bbold}
\usepackage{mathtools}
\mathtoolsset{showonlyrefs=fasle,showmanualtags}
\usepackage{amsthm} 
\theoremstyle{plain}
\newtheorem{theorem}{Theorem}[section]
\newtheorem{lem}[theorem]{Lemma}

\theoremstyle{definition}

\newtheorem{assu}{Assumption}
\theoremstyle{remark}
\newtheorem*{remark}{Remark}

\usepackage{amssymb}
\usepackage{algorithm}
\usepackage{algorithmic}
\usepackage{float}

\usepackage[english]{babel}
\usepackage{enumitem}
\usepackage{color,booktabs,colortbl}
\usepackage{latexsym,bm}
\usepackage{multicol}
\usepackage{comment}
\usepackage{geometry}
\usepackage{fancyhdr}
\newcommand{\ra}[1]{\renewcommand{\arraystretch}{#1}}
\usepackage{arydshln}
\usepackage[colorlinks,hyperindex,bookmarks=true]{hyperref}
\definecolor{dark-red}{rgb}{0.4,0.15,0.15}
\definecolor{dark-blue}{rgb}{0.15,0.15,0.4}
\definecolor{medium-blue}{rgb}{0,0,0.5}
\hypersetup{%
  linkcolor = dark-red,%
  citecolor = dark-blue,%
  urlcolor = medium-blue,
  bookmarksnumbered=true, bookmarksopen=true, bookmarksopenlevel=2,
  pdftitle=paper for phase retrieval,
  pdfauthor=MathLee,
  pdfsubject= Mathematic,
  pdfkeywords={}}
\usepackage{tikz}
\usetikzlibrary{calc,intersections}
\usetikzlibrary{datavisualization,plotmarks}
\usetikzlibrary{external}
\tikzdatavisualizationset {
general axis/.style={
scientific axes={inner ticks}, 
all axes = {padding=2mm}, 
x axis = {ticks={many,minor steps between steps=1}},
y axis = {grid,grid={style=white,major={style={black!30,dashed}}},ticks={some,minor steps between steps=2}},
style sheet = vary hue, 
}
}
\usetikzlibrary{angles,quotes}
\usepackage{caption}
\usepackage{subcaption}
\newcommand{\norm}[1]{\left\lVert#1\right\rVert}

\DeclareMathOperator*{\argmin}{arg\,min}

\DeclareMathOperator*{\re}{Re}
\DeclareMathOperator*{\im}{Im}
\DeclareMathOperator*{\diag}{diag}

\DeclareMathOperator*{\dist}{dist}

\DeclareMathOperator*{\prox}{prox}
\usepackage{listings}
\usepackage{authblk}
 \usepackage{cite}

\begin{document} \title{Solving Phase Retrieval via Graph Projection Splitting} 
\author[*]{Ji Li\thanks{email: keelee@csrc.ac.cn, matliji@nus.edu.sg}}
\author[**]{Hongkai Zhao\thanks{email: zhao@math.uci.edu}}
\affil[*]{Beijing Computational Science Research Center, Beijing, China}
\affil[**]{Department of Mathematics, University of California, Irvine, CA, USA}
\date{\today}
\maketitle
\begin{abstract}
Phase retrieval with prior information can be cast as a nonsmooth and nonconvex optimization problem. To decouple the signal and measurement variables, we introduce an auxiliary variable and reformulate it as an optimization with an equality constraint.  We then solve the reformulated problem by graph projection splitting (GPS), where the two proximity subproblems and the graph projection step can be solved efficiently. With slight modification, we also propose a robust graph projection splitting (RGPS) method to stabilize the iteration for noisy measurements. Contrary to intuition, RGPS outperforms GPS with fewer iterations to locate a satisfying solution even for noiseless case. Based on the connection between GPS and Douglas-Rachford iteration, under mild conditions on the sampling vectors, we analyze the fixed point sets and provide the local convergence of GPS and RGPS applied to noiseless phase retrieval without prior information. For noisy case, we provide the error bound of the reconstruction. Compared to other existing methods, thanks for the splitting approach, GPS and RGPS can efficiently solve phase retrieval with prior information regularization for general sampling vectors which are not necessarily isometric. For Gaussian phase retrieval, compared to existing gradient flow approaches, numerical results show that GPS and RGPS are much less sensitive to the initialization. Thus they markedly improve the phase transition in noiseless case and reconstruction in the presence of noise respectively. GPS shows sharpest phase transition among existing methods including RGPS, while it needs more iterations than RGPS when the number of measurement is large enough. RGPS outperforms GPS  in terms of stability for noisy measurements. When applying RGPS to more general non-Gaussian measurements with prior information, such as support, sparsity and TV minimization, RGPS either outperforms state-of-the-art solvers or can be combined with state-of-the-art solvers to improve their reconstruction quality. 
\end{abstract}

\section{Introduction}

We consider the phase retrieval problem with prior information expressed as:
\begin{equation}
  \label{eq:prob}
 \begin{aligned}
  \min \quad& g(\bm{x})\\
  \text{s.t.}\quad &\lvert\bm{a}_i^*\bm{x}\rvert = b_i+\epsilon_i^{\text{noise}},i=1,\ldots,m,
\end{aligned}
\end{equation}
where the objective function $g(\bm{x})$ corresponds to the prior information, such as $\ell_1$ norm for sparsity, total variation for piecewise constant, or the indicator function of solution set $\mathcal{X}\subset \mathbb{C}^n$, $\{\bm{a}_i\}_{i=1}^m\in\mathbb{C}^n$ are the sampling vectors, the nonnegative $b_i$'s are amplitude measurements and $\epsilon_i^{\text{noise}}$'s are the corruption noise. Without loss of generality, we assume that $b_{\min}=\min\{b_1,\ldots,b_m\}>0$, as we can put the equalities with $b_i=0$ as constraints included in $g(\bm{x})$. Finding $\bm{x}$ amounts to solving a system of quadratic equations, which is generally an NP-hard problem. The main difficulty of solving~\eqref{eq:prob} stems from the lack of the phase information and the nonconvexity of the amplitude measurement constraints. When the sampling vectors are drawn from Fourier transform basis, problem~\eqref{eq:prob} is the so-called \emph{Fourier phase retrieval}, which has a wide range of imaging applications in science and engineering, such as X-ray crystallography~\cite{Millane1990},
electron microscopy~\cite{Misell1973}, X-ray diffraction
imaging~\cite{Shechtman2015}, optics~\cite{Kuznetsova1988} and
astronomy~\cite{Fienup1982}, just to name a few. When $\bm{a}_i$'s are drawn from (complex) Gaussian distribution, problem~\eqref{eq:prob} is called \emph{Gaussian phase retrieval}, which is the model problem in recent research of phase retrieval, due to its nice statistical properties that lead to provable theoretical results. 

Much efforts have been devoted to developing provable algorithms for Gaussian phase retrieval without prior information, i.e., $g(\bm{x})=0$. For Gaussian phase retrieval, the solution is unique up to a global phase offset in noiseless case ($\epsilon_i^{\text{noise}}=0$) and it is also stable in noisy case when the number of measurements $m=\mathcal{O}(n)$~\cite{Candes2012,Chen2015b}. There are many provable algorithms to locate the solution to~\eqref{eq:prob} from a good initial guess for Gaussian measurements. These algorithms include Wirtinger flow~\cite{Candes2015}, truncated Wirtinger flow~\cite{Chen2015b}, amplitude truncated flow (ATF)~\cite{Wang16}, reweighted amplitude flow (RAF)~\cite{Wang2017Solving} and alternating minimization~\cite{Netrapalli2013}. Most of them are based on gradient flows for different loss functions starting from a specific initialization. Thus the choice of step size is important to ensure convergence and achieve fast convergence rate. When ratio $m/n$ is large enough, a good initial guess can be generated by various initialization schemes, such as spectral method~\cite{Candes2015} and reweighted maximal correlation method~\cite{Wang2017Solving}. However, finding a good initial guess stably is not a simple task in general, for examples, when the number of Gaussian measurements is not large enough or the measurements are not Gaussian. 
The crucial dependence of a good initial guess  can be avoided by the lifting technique for quadratic programming. Convex semidefinite relaxation (SDR) approaches, such as PhaseLift~\cite{Candes2013} and PhaseCut~\cite{Waldspurger2015}, have been proposed to solve~\eqref{eq:prob}. However, the extended dimension of SDR is prohibitive for high dimensional phase retrieval applications. Another convexification algorithm PhaseLin is introduced with the help of anchor vector~\cite{Dhifallah2017Phase}, which plays the similar role as initialization in gradient flow algorithms. To reduce the high computation cost of PhaseLift, matrix-factorization based approach, IncrePR~\cite{Li2018},  with improved phase transition is developed. The main issue of these solvers is their inapplicability to tackle practical phase retrieval problems with non-Gaussian measurements, where a good initial guess can not be easily obtained for gradient flow solvers and the tightness of SDR does not hold for convex solvers. Although these methods can be adapted to the inclusion of a regularization term $g(\bm{x})$ in the objective, they usually fail to find a satisfactory solution due to the stagnation of nonconvex optimization without a good initial guess. 

For different types of sampling vectors, the uniqueness and stability of the solution to~\eqref{eq:prob} may differ. For Fourier phase retrieval, the uniqueness is almost guaranteed up to three trivial ambiguities -- translation, mirror shift, global phase offset -- and their combinations. To mitigate these difficulties and improve the algorithmic efficiency, some additional constraints on the solution are imposed, such as support set, real-valuedness, nonnegativity, and sparsity. With the prior information, the popular solver HIO (hybrid input-output) and RAAR are widely used to solve Fourier phase retrieval~\cite{Luke2004,Luke2003}. However, they do not apply in the case of non-isometric measurements with more general prior constraints, e.g., total variation regularization. Recently, RAAR has been adapted to nonisometric measurements~\cite{Li162}, but it does not support prior information. Thus these existing algorithms for non-Gaussian measurements have different kinds of restrictions.  In this paper, we propose an unified algorithmic framework to solve phase retrieval that supports general measurements and prior information simultaneously.

For noiseless phase retrieval~\eqref{eq:prob}, we stack the sampling vectors into a matrix $\bm{A}=[\bm{a}_1,\ldots,\bm{a}_m]\in\mathbb{C}^{n\times m}$, then the measurements can be written as $\lvert \bm{A}^*\bm{x}\rvert=\bm{b}$. Throughout this paper, we assume $\bm{A}^*$ is full column rank. With the indicator function $f(\bm{y}) = \mathbb{I}_{\{y:\lvert y\rvert=b\}}(\bm{y})$, problem~\eqref{eq:prob} can be reformulated as an unconstrained optimization to minimize $f(\bm{A}^*\bm{x})+g(\bm{x})$. To tackle the difficulty of the involvement of $\bm{A}^*$ in $f$, we consider the following splitting form:
\begin{equation}
\label{eq:split}
\begin{aligned}
  \min \quad & f(\bm{y}) + g(\bm{x})\\
  \text{s.t.}\quad & \bm{A}^*\bm{x}=\bm{y}.
\end{aligned}
\end{equation}
Equation~\eqref{eq:split} can be solved by the following standard ADMM:
\begin{equation}
  \label{eq:admm1}
  \begin{aligned}
    \bm{x}^{k+\frac{1}{2}}&=\argmin_{\bm{x}}\quad g(\bm{x})+\frac{\rho}{2}\norm{\bm{A}^*\bm{x}-(\bm{y}^{k}-\bm{\lambda}^k)}^2\\
    \bm{y}^{k+1} &=\prox_f\left(\bm{A}^*\bm{x}^{k+\frac{1}{2}}+\bm{\lambda}^k\right) \\
    \bm{\lambda}^{k+1}&=\bm{\lambda}^k+\bm{A}^*\bm{x}^{k+\frac{1}{2}}-\bm{y}^{k+1},
  \end{aligned}
\end{equation}
where $k$ is the iteration number. The proximity operator is defined as
\begin{equation*}
  \prox_{\phi}(\bm{y})=\argmin_{\bm{z}}\left(\phi(\bm{z})+\frac{\rho}{2}\norm{\bm{z}-\bm{y}}_2^2\right),
\end{equation*}
where we suppress the parameter $\rho$ in our notation.
Operator $\prox_{f}(\bm{z})$ is just the projection of $\bm{z}$ onto the set $\{\bm{y}\in\mathbb{C}^m\mid \lvert \bm{y}\rvert = \bm{b}\}$ with the expression
\begin{equation*}
  \prox_{f}(\bm{z}) =\bm{b}\circ \frac{\bm{z}}{\lvert \bm{z}\rvert}.
\end{equation*}
Note that the division is elementwise. When $\lvert z_i\rvert=0$, we set $z_i/\lvert z_i\rvert=0$.

Although solving the first subproblem is not straightforward in general, it can be solved easily in some special situations. For example, if $m=n$ and matrix $\bm{A}^*$ satisfies $\bm{A}\bm{A}^*=\bm{A}^*\bm{A}=\bm{I}$, then the first subproblem is equivalent to 
\begin{equation*}
  \bm{x}^{k+\frac{1}{2}}=\arg\min_{\bm{x}}\quad g(\bm{x})+\frac{\rho}{2}\norm{\bm{x}-\bm{A}(\bm{y}^{k}-\bm{\lambda}^k)}^2.
\end{equation*}
Furthermore, if $g(\bm{x})$ is the indicator function of a set, such as nonnegativeness, real-valuedness, the ADMM method is the same as the case considered in~\cite{Wen2012Alternating}. 

If $g(\bm{x})=0$ and $m\geq n$, the first subproblem is just a least-squares problem, we have $\bm{x}^{k+\frac{1}{2}}=(\bm{A}\bm{A}^*)^{-1}\bm{A}(\bm{y}^{k}-\bm{\lambda}^k)$. Let $\bm{y}_{\text{DR}}^k=\bm{\lambda}^k+\bm{y}^k$ and substitute it to~\eqref{eq:admm1}, we have
\begin{equation}
\label{eq:drg}
  \begin{aligned}
    \bm{x}^{k+\frac{1}{2}}&=(\bm{A}\bm{A}^*)^{-1}\bm{A}(2\bm{y}^{k}-\bm{y}_{\text{DR}}^k)\\
    \bm{y}^{k+1}&= \prox_f(\bm{A}^*\bm{x}^{k+\frac{1}{2}}+\bm{y}_{\text{DR}}^k-\bm{y}^k)\\
    \bm{y}_{\text{DR}}^{k+1}&= \bm{A}^*\bm{x}^{k+\frac{1}{2}}+\bm{y}_{\text{DR}}^k-\bm{y}^k.
  \end{aligned}
\end{equation}
So after one iteration, we have $\bm{y}^{k+1}=\prox_f(\bm{y}_{\text{DR}}^{k+1})$. With the relation, we have the equivalent Douglas-Rachford (DR) iteration
\begin{equation}
\bm{y}_{\text{DR}}^{k+1} = \bm{y}_{\text{DR}}^k+\bm{A}^*(\bm{A}\bm{A}^*)^{-1}\bm{A}\left(2\bm{b}\circ \frac{\bm{y}_{\text{DR}}^k}{\lvert \bm{y}_{\text{DR}}^k\rvert}-\bm{y}_{\text{DR}}^k\right)-\bm{b}\circ \frac{\bm{y}_{\text{DR}}^k}{\lvert \bm{y}_{\text{DR}}^k\rvert}.\tag{DR}
\end{equation}
Then we can generate the solution sequence $\bm{x}^k$ by the expression $(\bm{A}\bm{A}^*)^{-1}\bm{A}\prox_f(\bm{y}_{\text{DR}}^k)$. DR can also be reduced from solving a set feasible problem~\cite{Li162}. Although the above DR is simple, it can not deal with the general case when $g(x)\ne 0$. Later we will include this DR algorithm in comparison when $g(x)=0$ in our numerical experiments. 

To develop an efficient algorithm for the general case, we revert to graph projection splitting (GPS) to solve~\eqref{eq:split} instead, where each subproblem can be solved efficiently. To the best of our knowledge, it is the first time GPS is used to solve  phase retrieval problem~\eqref{eq:prob}. For phase retrieval problem with noise, we propose a robust GPS (RGPS) method. 
We would like to point out several advantages of using graph projection splitting in $(\bm{x},\bm{y})\in \mathbb{C}^{n+m}$ graph space 
\begin{itemize}
\item
Prior information for both signal and measurement can be easily incorporated.
\item
No difficult parameters, e.g., time step for gradient based method, to tune.
\item
Graph projection finds a pair $(\bm{x},\bm{y})$ that satisfies the exact relation in GPS.
\item
Graph projection updates $\bm{x},\bm{y}$ simultaneously without bias in an optimal way.
\item
$(\bm{I}+\bm{A}\bm{A}^*)^{-1}$ is better conditioned (than $(\bm{A}\bm{A}^*)^{-1}$).
\item
Using the distance to the graph in $(\bm{x},\bm{y})$ space in RGPS is more intrinsic and unbiased to measure the violation of the constraint.
\end{itemize}

Details of both algorithms are presented in Section 2. Actually RGPS can also be used for noiseless case and performs better than GPS near the solution. Using the equivalence relation between GPS and Douglas-Rachford, we show local convergence of GPS and RGPS by formulating them in single-variable-updating forms in Section 3 and defer some technical details in Section 4. Numerical experiments are conducted to demonstrate the performance of GPS and RGPS in Section 5. Section 6 provides the conclusion. 

\paragraph{Notation}

We use bold fonts for vectors and matrices. We denote the inner product in $\mathbb{C}^n$ by $\langle \bm{u},\bm{v}\rangle=\bm{u}^*\bm{v}$. $\bm{u}$ is said orthogonal to $\bm{v}$ if $\re\langle \bm{u},\bm{v}\rangle = 0$. The indicator function for a set $C$ is denoted as $\mathbb{I}_{C}$, which is defined by $\mathbb{I}_{C}(\bm{x})=0$ for $\bm{x}\in C$ and $\mathbb{I}_{C}(\bm{x})=+\infty$ otherwise. Notation $\bm{u}\circ\bm{v}$ and $\bm{u}/\bm{v}$ is the elementwise multiplication and division between vectors.

\paragraph{Reproducible research}

The accompany code for this paper can be found online at: \url{https://github.com/Chilie/GPS}.

\section{Algorithms}

\subsection{Graph Projection Splitting}

Instead of recognizing $\bm{y}$ as the auxiliary variable, we view splitting form~\eqref{eq:split} as an optimization with a stacked variable $(\bm{x},\bm{y})\in\mathbb{C}^{n+m}$. Then we apply graph projection splitting (GPS) to solve~\eqref{eq:split}, where the equality constraint indicts the stacked variable is in the graph set $C=\{(\bm{x},\bm{y})\in\mathbb{C}^{n+m}\mid \bm{A}^*\bm{x}=\bm{y}\}$. Given initial primal variables $\bm{x}^0$, $\bm{y}^0$ and dual variables $\bm{\lambda}^0,\bm{\nu}^0$, the GPS iteration is given by
\begin{subequations}
\label{eq:gps}
\begin{align}
 (\bm{x}^{k+\frac{1}{2}},\bm{y}^{k+\frac{1}{2}})&:=\Pi_{\bm{A}^*}(\bm{x}^{k}-\bm{\lambda}^k,\bm{y}^{k}-\bm{\nu}^k)\label{eq:gr}\\
  (\bm{x}^{k+1},\bm{y}^{k+1})&:=\left(\prox_{g}(\bm{x}^{k+\frac{1}{2}}+\bm{\lambda}^k),\prox_{f}(\bm{y}^{k+\frac{1}{2}}+\bm{\nu}^k)\right)\label{eq:px}\\
  (\bm{\lambda}^{k+1},\bm{\nu}^{k+1})&:=\left(\bm{\lambda}^k+\bm{x}^{k+\frac{1}{2}}-\bm{x}^{k+1},\bm{\nu}^k+\bm{y}^{k+\frac{1}{2}}-\bm{y}^{k+1}\right),
\end{align}
\end{subequations}
where operator $\Pi_{\bm{A}^*}$ denotes projection onto the graph set $C$ and implies the dependence on $\bm{A}^*$. The graph projection step is the most dominant computation, as the two proximal steps~\eqref{eq:px} are easily obtained for $f$ and $g$ in our problems. Typically, we initialize $\bm{\lambda}^0,\bm{\nu}^0$ with zero vectors, $\bm{x}^0$ with a random vector and $\bm{y}^0=\bm{A}^*\bm{x}^0$. The most appealing advantage of GPS is the \emph{parameter-free} iteration. 

The graph projection splitting names after the graph projection step in the algorithm. Actually, GPS can be interpreted as a specific ADMM in the stack variable $\bm{z}=(\bm{x},\bm{y})\in\mathbb{C}^{n+m}$. By setting $\phi(\bm{z})=f(\bm{y})+g(\bm{x})$, we are to minimize $\phi(\bm{z})$ with the constraint $\bm{z}\in C=\{\bm{z}=(\bm{x},\bm{y})|\bm{A}^*\bm{x}=\bm{y}\}$. The ADMM iteration for it is as follows
\begin{equation}
  \label{eq:admm}
    \begin{aligned}
    \bm{z}^{k+\frac{1}{2}}&:=\Pi_{\bm{A}^*}(\bm{z}^k-\bm{\zeta}^k)\\
    \bm{z}^{k+1}&:=\prox_{\phi}(\bm{z}^{k+\frac{1}{2}}+\bm{\zeta}^k)\\
    \bm{\zeta}^{k+1}&:=\bm{\zeta}^k+\bm{z}^{k+\frac{1}{2}}-\bm{z}^{k+1},
  \end{aligned}\tag{ADMM}
\end{equation}
where $\bm{\zeta}=(\bm{\lambda},\bm{\nu})\in\mathbb{C}^{n+m}$ is the dual variable. With the splitting of our problem and the separation of function $\phi(\bm{z})$, substituting the stacked variable $\bm{z}$ into ADMM iteration, we get GPS iteration~\eqref{eq:gps}. Note that for convex problem, the ADMM iteration ensures global convergence, and so does GPS. The ADMM/GPS can be directly applied to phase retrieval without modification, although the global convergence is not guaranteed. We will show its local convergence near the solution. Note that although GPS and the following RGPS are still based on nonconvex optimization, their applications to Gaussian phase retrieval exhibit global convergence numerically. 

\subsection{Robust GPS}

Note that the graph projection step~\eqref{eq:gr} is associated with the proximal operator for the indicator function $\mathbb{I}_{C}(\cdot)$. When we consider noisy measurements, we should change the function $f(\bm{y})$. One way is to replace the indicator function by least-squares function $f(\bm{y})=\frac{1}{2}\|\lvert\bm{y}\rvert-\bm{b}\|^2$, which shows instability in our experiments. Instead, we introduce a relaxed projection by replacing $\mathbb{I}_{C}(\cdot)$ with $\dist_{C}^2(\bm{s})$, which is square of the distance from $\bm{s}$ to its projection $\Pi_{C}(\bm{s})$. 

Hence, the graph projection step~\eqref{eq:gr} given by
\begin{equation*}
  (\bm{x}^{k+\frac{1}{2}},\bm{y}^{k+\frac{1}{2}}) = \argmin_{\bm{x},\bm{y}}\quad \mathbb{I}_{C}(\bm{x},\bm{y}) + \frac{\rho}{2}\norm{\bm{x}-\bm{c}^k}^2+\frac{\rho}{2}\norm{\bm{y}-\bm{d}^k}^2,
\end{equation*}
where $\bm{c}^k=\bm{x}^k-\bm{\lambda}^k,\bm{d}^k=\bm{y}^k-\bm{\nu}^k$,
is replaced by
\begin{equation*}
  (\bm{x}^{k+\frac{1}{2}},\bm{y}^{k+\frac{1}{2}}) = \argmin_{\bm{x},\bm{y}}\quad \frac{1}{2}{\dist}_{C}^2(\bm{x},\bm{y}) + \frac{\rho}{2}\norm{\bm{x}-\bm{c}^k}^2+\frac{\rho}{2}\norm{\bm{y}-\bm{d}^k}^2.
\end{equation*}

\begin{lem}
  Given $C\subset\mathbb{C}^n$ is a nonempty, closed convex set, then the distance function $\frac{1}{2}{\dist}_{C}^2(\bm{s})$ is differentiable, and its gradient is
  \begin{equation*}
    \nabla (\frac{1}{2}{\dist}_{C}^2(\bm{s})) = (\bm{I}-\Pi_{C})\bm{s}.
  \end{equation*}
\end{lem}
\begin{proof}
  The proof is straightforward. Since 
  \begin{equation*}
    \frac{1}{2}{\dist}_{C}^2(\bm{s})=\frac{1}{2}\norm{(\bm{I}-\Pi_{C})\bm{s}}_2^2,
  \end{equation*}
the gradient is given by $(\bm{I}-\Pi_C)s$.
\end{proof}
Moreover, since $C=\{(\bm{x},\bm{y})\in\mathbb{C}^{n+m}\mid \bm{A}^*\bm{x}=\bm{y}\}$ is a linear space, we have the following properties.
\begin{lem}
\label{lem:2}
  The inverse of the operator $(\bm{I}-t\Pi_{C})$ for $t\in\mathbb{R},t\neq 1$ is $\bm{I}+\frac{t}{1-t}\Pi_{C}$.
\end{lem}
\begin{proof}
  It is easy to check that
  \begin{equation*}
    (\bm{I}-t\Pi_{C})(\bm{I}+\frac{t}{1-t}\Pi_{C})=\bm{I}.
  \end{equation*}
\end{proof}
\begin{theorem}
  The proximity operator for function $\frac{1}{2}{\dist}_{C}^2(\bm{s})$ is given by
  \begin{equation*}
    \prox_{\frac{1}{2}\dist_{C}^2}(\bm{s}) = \left(\frac{\rho}{1+\rho}\bm{I}+\frac{1}{1+\rho}\Pi_{C}\right)(\bm{s}).
  \end{equation*}
\end{theorem}
\begin{proof}
 The proximity point $\bm{x}$ satisfies 
\begin{equation*}
  \bm{x}=\prox_{\frac{1}{2}\dist_{C}^2}(\bm{s})=\argmin_{\bm{y}}\quad \frac{1}{2}{\dist}_{C}^2(\bm{y}) + \frac{\rho}{2}\norm{\bm{y}-\bm{s}}^2. 
\end{equation*}
According to the first-order optimal condition, $\bm{x}$ satisfies 
  \begin{equation*}
    (\bm{I}-\Pi_{C})\bm{x}+\rho(\bm{x}-\bm{s})=0.
  \end{equation*}
The optimal solution $\bm{x}$ is
\[
  \bm{x}=\frac{\rho}{1+\rho}\left(\bm{I}-\frac{1}{1+\rho}\Pi_{C}\right)^{-1}\bm{s}
   = \frac{\rho}{1+\rho}\left(\bm{I}+\frac{1}{\rho}\Pi_{C}\right)\bm{s}
  = \left(\frac{\rho}{1+\rho}\bm{I}+\frac{1}{1+\rho}\Pi_{C}\right)(\bm{s}),
\]
where Lemma~\ref{lem:2} is used.
\end{proof}

We obtain the robust graph projection, 
\begin{equation}
  \label{eq:smooth-gr}
  (\bm{x}^{k+\frac{1}{2}},\bm{y}^{k+\frac{1}{2}})=(t\bm{I}+(1-t)\Pi_{\bm{A}^*})(\bm{x}^{k}-\bm{\lambda}^k,\bm{y}^{k}-\bm{\nu}^k). 
\end{equation}
When it replaces~\eqref{eq:gr} in GPS, we call the resulting algorithm robust GPS (RGPS). To ensure the local convergence, the allowable range of values of $t$ falls into $(0,t_{\max})$, where the upper $t_{\max}$ depends on the singular values of $\bm{A}^*$, see Section 3.

\begin{remark}
  We propose the robust GPS motivated by the infeasibility of graph set in noisy case. Actually, as we will see, besides noisy case, RGPS can also be used for noiseless case and outperforms GPS around the solution. This phenomenon can be somewhat explained by the local convergence behaviors of the two methods in Section 3.
\end{remark}

Multiplying with $\bm{A}^*$ on both sides of the first step of~\eqref{eq:drg} and viewing $\bm{A}^*\bm{x}^{k+\frac{1}{2}}$ as a whole variable, we recognize the resulting step as the projection $2\bm{y}^k-\bm{y}_{\text{DR}}^k$ onto the range of $\bm{A}^*$, by the same argument, robust Douglas-Rachford (RDR) can be also proposed. The iteration reads as follows
\begin{equation}
\bm{y}_{\text{DR}}^{k+1} = \bm{y}_{\text{DR}}^k+(t\bm{I}+(1-t)\bm{A}^*(\bm{A}\bm{A}^*)^{-1}\bm{A})\left(2\bm{b}\circ \frac{\bm{y}_{\text{DR}}^k}{\lvert \bm{y}_{\text{DR}}^k\rvert}-\bm{y}_{\text{DR}}^k\right)-\bm{b}\circ \frac{\bm{y}_{\text{DR}}^k}{\lvert \bm{y}_{\text{DR}}^k\rvert}.\tag{RDR}
\end{equation}

\subsection{Graph Projection Step}

For $\bm{c}\in \mathbb{C}^n, \bm{d}\in \mathbb{C}^m$, the projection of $(\bm{c},\bm{d})$ onto the set graph set 
 $C=\{(\bm{x},\bm{y})|\bm{A}^*\bm{x}=\bm{y},\bm{x}\in\mathbb{C}^n,\bm{y}\in\mathbb{C}^m\}$ can be computed explicitly as follows. 
\begin{theorem}
  The projection $\Pi_{\bm{A}^*}(\bm{c},\bm{d})$ is given by the solution to the linear system
  \begin{equation}
    \label{eq:proj-eq}
    \begin{pmatrix}
      \bm{I} & \bm{A}\\ \bm{A}^* & -\bm{I}
    \end{pmatrix}
    \begin{pmatrix}
      \bm{x}\\\bm{y}
    \end{pmatrix}=
    \begin{pmatrix}
      \bm{c}+\bm{A}\bm{d}\\\bm{0}
    \end{pmatrix}.
  \end{equation}
\end{theorem}

\begin{proof}
  The projection is to solve the following optimization problem with linear constraint
  \begin{align*}
    \min_{\bm{x},\bm{y}}\quad & \frac{1}{2}\norm{\bm{x}-\bm{c}}^2+\frac{1}{2}\norm{\bm{y}-\bm{d}}^2\\
    \text{s.t.}\quad & \bm{A}^*\bm{x}=\bm{y}.
  \end{align*}
According to the KKT condition, there exists $\bm{\lambda}$ such that
\begin{align*}
  \bm{x}-\bm{c}+\bm{A}\bm{\lambda}&=0\\
  \bm{y}-\bm{d}-\bm{\lambda} & =0\\
  \bm{A}^*\bm{x}&=\bm{y}.
\end{align*}
After some substitutions, we reach the conclusion.
\end{proof}

In most applications $m\geq n$, it is more efficient to compute the projection by the following
\begin{equation}
\label{eq:gpstep1}
\begin{aligned}
  \bm{x} &= (\bm{I}+\bm{A}\bm{A}^*)^{-1}(\bm{c}+\bm{A}\bm{d})\\
  \bm{y} &= \bm{A}^*\bm{x}.
\end{aligned}
\end{equation}
Using the Cholesky decomposition $\bm{L}\bm{L}^*$ of $\bm{I}+\bm{A}\bm{A}^*$, where $\bm{L}\in\mathbb{C}^{n\times n}$ is lower triangular matrix, we compute $\bm{x}$ by the following forward and backward substitutions,
\begin{subequations}
\begin{align*}
  \bm{L}\tilde{\bm{x}}&=\bm{c}+\bm{A}\bm{d}\tag{forward substitution}\\
  \bm{L}^*\bm{x} &= \tilde{\bm{x}}\tag{backward substitution}.
\end{align*}
\end{subequations}

\begin{remark}
  If we are to solve a phase retrieval problem with sparsity prior, $m$ may be less than $n$. In this case, we compute the projection as follows
\begin{equation}
\label{eq:gpstep2}
  \begin{aligned}
    \bm{y} &= (\bm{I}+\bm{A}^*\bm{A})^{-1}\bm{A}^*(\bm{c}+\bm{A}\bm{d})\\
    \bm{x} & = \bm{c} + \bm{A}(\bm{d}-\bm{y}).
  \end{aligned}
\end{equation}
\end{remark}

\begin{remark}
  If the sampling matrix $\bm{A}^*$ satisfies $\bm{A}\bm{A}^*=l\bm{I}$, such as coded diffraction pattern (CDP)~\cite{Candes2014}, in which $l$ is the number of Fourier measurements with phase mask. In this case, we have 
  \begin{equation*}
    (\bm{I}+\bm{A}\bm{A}^*)^{-1}=\frac{1}{l+1}\bm{I}.
  \end{equation*}
\end{remark}

Once the Cholesky decomposition is done once and the factor $\bm{L}$ is stored, the only computations involved are matrix-vector multiplication in the computation. The computation of $\bm{I}+\bm{A}\bm{A}^*$ needs $\mathcal{O}(mn^2)$ flops and the computation of Cholesky decomposition needs $\mathcal{O}(n^3)$ flops. At each iteration, the forward and backward substitution take $\mathcal{O}(n^2)$ flops. Each graph projection costs $\mathcal{O}(mn)$. Compared to other state-of-the-art nonconvex solvers, GPS  and RGPS have the same order of computation cost in each step excluding the additional precomputation of Cholesky decomposition which is done once for all. In addition to the parameter-free advantage, another main advantage of GPS and RGPS is its flexibility to handle general measurements with prior information for $\bm{x}$ due to the introduction of $\bm{y}$ as an independent variable. 

\section{Local Convergence Analysis}
\label{sec:loc}

In the section, we provide the local convergence of GPS and RGPS without regularization (for prior information), i.e. $g(\bm{x})=0$. Although one may argue, in this case, DR involving only $\bm{y}_{\text{DR}}^k$ should be used instead of GPS since GPS needs storage of both $\bm{x}^k$ and $\bm{\lambda}^k$. However, GPS seems to have the more tendency of escaping stagnation point than DR, which will be illustrated by numerical tests. Thus GPS (RGPS) achieves sharper phase transition than DR (RDR). Since this specific case is the basis of our splitting formulation, its convergence study will shed insight for our method in more general situations and we defer local convergence study with general regularization for future research. Hereafter, in the following convergence analysis, $g(\bm{x})=0$ is always assumed. 

For our analysis, we need assumptions of the given sampling vectors to ensure the uniqueness of solution to phase retrieval and related magnitude retrieval up to some trivial ambiguities.
\begin{assu}
\label{assu:1}
  Throughout the paper, the following requirements of sampling vectors hold:
  \begin{enumerate}
  \item measurement matrix $\bm{A}^*$ has full column rank.
\item the solution to \emph{phase retrieval} problem~\eqref{eq:prob} is unique up to the trivial ambiguities, including translation shift, mirror flipping and global phase shift.\footnote{This is the general case for Fourier phase retrieval, the global phase shift is the only ambiguity for Gaussian phase retrieval.} The solution up to these ambiguities is denoted by $\bm{x}^{\natural}$ in our paper.
  \item the solution to \emph{magnitude retrieval} problem
    \begin{equation*}
      \begin{aligned}
        \text{find}\quad & \bm{x}\in\mathbb{C}^n\\
        \text{s.t.}\quad & \frac{\bm{a}_i^*\bm{x}}{\lvert\bm{a}_i^*\bm{x}\rvert}=\pm\frac{\bm{a}_i^*\bm{x}^{\natural}}{\lvert\bm{a}_i^*\bm{x}^{\natural}\rvert}, \quad i=1,2, \ldots, m
      \end{aligned}
    \end{equation*}
is unique up to a constant magnitude difference, where the $\pm$ sign is element-by-element, i.e., its solution is $c\bm{x}^{\natural}, c\in \mathbb{R}\setminus\{0\}$.
  \end{enumerate}
\end{assu}

The assumptions \emph{almost} hold for \emph{oversampling} Fourier phase retrieval, see~\cite{Hayes1982}. And we believe they are satisfied for Gaussian phase retrieval, while the uniqueness of magnitude retrieval needs proof.
\subsection{Equivalence of GPS and Douglas-Rachford}

The connection of~\eqref{eq:admm} and Douglas-Rachford (DR) is well-known. As DR involves only one sequence, it is easier to study than~\eqref{eq:admm}. To see this connection, we take $ \bm{z}_{\text{DR}}^k=\bm{z}^k+\bm{\zeta}^k$ and substitute it into~\eqref{eq:admm}, we have
  \begin{align*}
    \bm{z}^{k+\frac{1}{2}}&:=\Pi_{\bm{A}^*}(2\bm{z}^k-\bm{z}_{\text{DR}}^k)\\
    \bm{z}^{k+1}&:=\prox_{\phi}(\bm{z}^{k+\frac{1}{2}}+\bm{z}_{\text{DR}}^k-\bm{z}^k)\\
    \bm{z}_{\text{DR}}^{k+1}&:=\bm{z}^{k+\frac{1}{2}}+\bm{z}_{\text{DR}}^k-\bm{z}^{k}.
  \end{align*}
After one iteration, we have the relation $\bm{z}^{k+1}=\prox_{\phi}(\bm{z}_{\text{DR}}^{k+1})$. So the iteration becomes:
  \begin{align*}
    \bm{z}^{k+\frac{1}{2}}&:=\Pi_{\bm{A}^*}(2\prox_{\phi}(\bm{z}_{\text{DR}}^k)-\bm{z}_{\text{DR}}^k)\\
    \bm{z}_{\text{DR}}^{k+1}&:=\bm{z}_{\text{DR}}^k+\bm{z}^{k+\frac{1}{2}}-\prox_{\phi}(\bm{z}_{\text{DR}}^k).
  \end{align*}
It is easy to see the iteration falls into DR iteration. Thus sequence $\{\bm{z}^k\}$ in~\eqref{eq:admm} can be generated by the following scheme
\begin{subequations}
\begin{align}
  \bm{z}_{\text{DR}}^{k+1}&=\bm{z}_{\text{DR}}^k + \Pi_{\bm{A}^*}\left(2\prox_{\phi}(\bm{z}_{\text{DR}}^k)-\bm{z}_{\text{DR}}^k\right)-\prox_{\phi}(\bm{z}_{\text{DR}}^k)\label{eq:dr}\\
  \bm{z}^{k+1}&=\prox_{\phi}(\bm{z}_{\text{DR}}^{k+1}).\label{eq:drp}
\end{align}
\end{subequations}
The obvious benefit is that we deal with a single-variable in DR iteration which is helpful for the local convergence study of GPS and RGPS for phase retrieval problem~\eqref{eq:prob}.

We let $\bm{z}_{\text{DR}}^k=
\begin{pmatrix}
  \bm{x}_{\text{DR}}^k\\\bm{y}_{\text{DR}}^k
\end{pmatrix}
$ be the iterative sequence of DR, where $\bm{x}_{\text{DR}}^k\in\mathbb{C}^n$ and $\bm{y}_{\text{DR}}^k\in\mathbb{C}^m$. By equality~\eqref{eq:drp}, we have $\bm{x}^k = \bm{x}_{\text{DR}}^k$. The DR iteration of GPS~\eqref{eq:gps} is 
\begin{equation}
\label{eq:iter}
  \begin{pmatrix}
    \bm{x}_{\text{DR}}^{k+1}\\\bm{y}_{\text{DR}}^{k+1}
  \end{pmatrix}=  \begin{pmatrix}
    \bm{x}_{\text{DR}}^{k}\\\bm{y}_{\text{DR}}^{k}
  \end{pmatrix}+\bm{M}\left(2\prox_{\phi}\begin{pmatrix}
    \bm{x}_{\text{DR}}^{k}\\\bm{y}_{\text{DR}}^{k}
  \end{pmatrix}-\begin{pmatrix}
    \bm{x}_{\text{DR}}^{k}\\\bm{y}_{\text{DR}}^{k}
  \end{pmatrix}\right)-\prox_{\phi}\begin{pmatrix}
    \bm{x}_{\text{DR}}^{k}\\\bm{y}_{\text{DR}}^{k}
  \end{pmatrix},
\end{equation}
where 
\begin{equation*}
  \bm{M} =
  \begin{pmatrix}
    \bm{I} & 0\\0 & \bm{A}^*
  \end{pmatrix}
  \begin{pmatrix}
    (\bm{I}+\bm{A}\bm{A}^*)^{-1} & (\bm{I}+\bm{A}\bm{A}^*)^{-1}\\(\bm{I}+\bm{A}\bm{A}^*)^{-1} & (\bm{I}+\bm{A}\bm{A}^*)^{-1}
  \end{pmatrix}\begin{pmatrix}
    \bm{I} & 0\\0 & \bm{A}
  \end{pmatrix}
\end{equation*}
is the corresponding matrix for the graph projection operator for the case $m\geq n$. Hereafter, without specific note, we assume $m\geq n$. 

From the Cholesky decomposition $\bm{I}+\bm{A}\bm{A}^*=\bm{L}\bm{L}^*$, we have $(\bm{I}+\bm{A}\bm{A}^*)^{-1}=(\bm{L}^{-1})^*\bm{L}^{-1}$. The matrix $\bm{M}$ can be expressed as
\begin{equation*}
  \bm{M}=
  \begin{pmatrix}
    (\bm{L}^{-1})^*\\\bm{A}^*(\bm{L}^{-1})^*
  \end{pmatrix}
  \begin{pmatrix}
    \bm{L}^{-1} & \bm{L}^{-1}\bm{A}
  \end{pmatrix}:=\bm{C}^*\bm{C}.
\end{equation*}
It can be easily verified that matrix $\bm{C}^*$ is isometric, i.e., $\bm{C}\bm{C}^*=\bm{I}$.

\subsection{Fixed Point Set}
Now we derive the fixed point set of~\eqref{eq:iter}, and investigate the local convergence around the fixed point. This analysis can be adapted to investigate the local convergence of DR which we omit here.

Denote the solution to~\eqref{eq:prob} by $\bm{x}^{\natural}$, we define the set 
  \begin{equation*}
    \mathcal{F} = \left\{
    \begin{pmatrix}
      \alpha\bm{x}^{\natural}\\
      \alpha(\bm{b}+\bm{\theta})\circ\frac{\bm{y}^{\natural}}{\lvert\bm{y}^{\natural}\rvert}
    \end{pmatrix}\Bigm| \lvert \alpha\rvert = 1, \bm{\theta}\in\mathbb{R}^m, \bm{A}(\bm{\theta}\circ\frac{\bm{y}^{\natural}}{\lvert\bm{y}^{\natural}\rvert})=0 \text{ and } \bm{b}+\bm{\theta}\geq 0
\right\},
\end{equation*}
where $\bm{y}^{\natural}=\bm{A}^*\bm{x}^{\natural}$ with $\lvert \bm{y}^{\natural}\rvert=\bm{b}> 0$.

\begin{lem}
  Upon Assumption~\ref{assu:1}, the fixed point set of iteration scheme~\eqref{eq:iter} is equal to $\mathcal{F}$.
\end{lem}
\begin{proof}
 Obviously, every element in $ \mathcal{F}$ is a fixed point.
 
 On the other hand, the fixed point $\bm{z}_{\text{DR}}^{\infty}=
  \begin{pmatrix}
    \bm{x}_{\text{DR}}^{\infty}\\
    \bm{y}_{\text{DR}}^{\infty}
  \end{pmatrix}
$ for iteration~\eqref{eq:iter} satisfies
  \begin{equation*}
    \bm{M}\left(2\prox_{\phi}\begin{pmatrix}
    \bm{x}_{\text{DR}}^{\infty}\\\bm{y}_{\text{DR}}^{\infty}
  \end{pmatrix}-\begin{pmatrix}
    \bm{x}_{\text{DR}}^{\infty}\\\bm{y}_{\text{DR}}^{\infty}
  \end{pmatrix}\right)=\prox_{\phi}\begin{pmatrix}
    \bm{x}_{\text{DR}}^{\infty}\\\bm{y}_{\text{DR}}^{\infty}
  \end{pmatrix}.
\end{equation*}
By $\bm{M}^2=\bm{M}$, we have
\begin{equation*}
  \bm{M}\left(\prox_{\phi}\begin{pmatrix}
    \bm{x}_{\text{DR}}^{\infty}\\\bm{y}_{\text{DR}}^{\infty}
  \end{pmatrix}\right)=\bm{M}\begin{pmatrix}
    \bm{x}_{\text{DR}}^{\infty}\\\bm{y}_{\text{DR}}^{\infty}
  \end{pmatrix}=\prox_{\phi}\begin{pmatrix}
    \bm{x}_{\text{DR}}^{\infty}\\\bm{y}_{\text{DR}}^{\infty}
  \end{pmatrix}.
\end{equation*}
By the second equality and the uniqueness of problem~\eqref{eq:prob}, we have $\begin{pmatrix}
    \bm{x}_{\text{DR}}^{\infty}\\\bm{y}_{\text{DR}}^{\infty}
  \end{pmatrix} \in \mathcal{F}$.
\end{proof}

Given a point $\bm{z}_{\text{DR}}^k$, we define its projection onto the fixed point set $\mathcal{F}$
\begin{equation*}
  P_{\mathcal{F}}\bm{z}_{\text{DR}}^k= \argmin_{\bm{z}\in\mathcal{F}}\norm{\bm{z}-\bm{z}_{\text{DR}}^k}_2^2.
\end{equation*}
Denote $\bm{z}_*^k:=P_{\mathcal{F}}\bm{z}_{\text{DR}}^k =
\begin{pmatrix}
  \alpha_k\bm{x}^{\natural}\\\alpha_k\bm{y}_*^k
\end{pmatrix}=\begin{pmatrix}
  \alpha_k\bm{x}^{\natural}\\\alpha_k(\bm{b}+\bm{\theta}_k)\circ\frac{\bm{y}^{\natural}}{\lvert\bm{y}^{\natural}\rvert}
\end{pmatrix}
$, where $\bm{\theta}_k\in\mathcal{C}=\left\{\bm{\theta}\in\mathbb{R}^m\Bigm |\bm{A}(\bm{\theta}\circ\frac{\bm{y}^{\natural}}{\lvert\bm{y}^{\natural}\rvert})=0 \text{ and } \bm{b}+\bm{\theta}\geq 0\right\}$.
When we study the local convergence around solution $\bm{z}^{\natural}=\begin{pmatrix}
  \bm{x}^{\natural}\\\bm{y}^{\natural}
\end{pmatrix}$, for each iteration number $k$, we will consider an open ball $\mathcal{V}_{k}\in\mathbb{C}^{n+m}$ of radius $b_{\min}/2$ centered at $\alpha_k\bm{z}^{\natural}$, where $\alpha_k$ is determined by $P_{\mathcal{F}}\bm{z}_{\text{DR}}^k$. 

\begin{lem}
  \label{lem:orth2}
  If $b_{\min}>0$ and $\bm{z}_{\text{DR}}^k\in \mathcal{V}_{k}$, then we have that
  \begin{equation*}
    \left\langle \frac{\overline{\alpha_k\bm{y}_*^k}}{\lvert \bm{y}_*^k\rvert}\circ (\bm{y}_{\text{DR}}^k-\alpha_k\bm{y}_*^k), \bm{\theta}\right\rangle=0, \forall \bm{\theta}\in \mathcal{C}',
  \end{equation*}
  where $\mathcal{C}'=\left\{\bm{\theta}\in\mathbb{R}^m\Bigm |\bm{A}(\bm{\theta}\circ\frac{\bm{y}^{\natural}}{\lvert\bm{y}^{\natural}\rvert})=0\right\}$. Therefore $\re\left(\frac{\overline{\alpha_k\bm{y}_*^k}}{\lvert \bm{y}_*^k\rvert}\circ (\bm{y}_{\text{DR}}^k-\alpha_k\bm{y}_*^k)\right)\perp \bm{\theta}$ and $\im\left(\frac{\overline{\alpha_k\bm{y}_*^k}}{\lvert \bm{y}_*^k\rvert}\circ (\bm{y}_{\text{DR}}^k-\alpha_k\bm{y}_*^k)\right)\perp \bm{\theta}$.
\end{lem}

\begin{proof}
 We consider set $\mathcal{F}'$ with the definition
\begin{equation*}
  \mathcal{F}' = \left\{
    \begin{pmatrix}
      \alpha_k\bm{x}^{\natural}\\
      \alpha_k(\bm{b}+\bm{\theta})\circ\frac{\bm{y}^{\natural}}{\lvert\bm{y}^{\natural}\rvert}
    \end{pmatrix}\Bigm| \bm{\theta}\in\mathbb{R}^m, \bm{A}(\bm{\theta}\circ\frac{\bm{y}^{\natural}}{\lvert\bm{y}^{\natural}\rvert})=0
\right\}.
\end{equation*}

From $\bm{z}_{\text{DR}}^k\in \mathcal{V}_k$, i.e., $
  \norm{\bm{z}_{\text{DR}}^k-\alpha_k\bm{z}^{\natural}}\leq \frac{b_{\min}}{2}$,
by the projection property, it implies that 
\begin{equation*}
  \norm{\bm{z}_{\text{DR}}^k-P_{\mathcal{F}'}\bm{z}_{\text{DR}}^k}\leq \norm{\bm{z}_{\text{DR}}^k-\alpha_k\bm{z}^{\natural}}\leq \frac{b_{\min}}{2}.
\end{equation*}
Hence $\norm{P_{\mathcal{F}'}\bm{z}_{\text{DR}}^k-\alpha_k\bm{z}^{\natural}}\leq b_{\min}$. 
Denote
\begin{equation*}
  P_{\mathcal{F}'}\bm{z}_{\text{DR}}^k = \begin{pmatrix}
  \alpha_k\bm{x}^{\natural}\\\alpha_k(\bm{b}+\bm{\theta}_k)\circ\frac{\bm{y}^{\natural}}{\lvert\bm{y}^{\natural}\rvert}
\end{pmatrix},
\end{equation*}
we have $\bm{b}+\bm{\theta}_k\geq 0$, which implies $P_{\mathcal{F}'}\bm{z}_{\text{DR}}^k=P_{\mathcal{F}}\bm{z}_{\text{DR}}^k$.

By the definition of projection $P_{\mathcal{F}'}$,
  \begin{equation*}
    \bm{\theta}_k = \argmin_{\bm{\theta}\in\mathcal{C}'}\norm{\bm{y}_{\text{DR}}^k-\alpha_k(\bm{b}+\bm{\theta})\frac{\bm{y}^{\natural}}{\lvert \bm{y}^{\natural}\rvert}}_2^2,
  \end{equation*}
one can derive the optimality condition
\begin{equation*}
  \left\langle \frac{\overline{\alpha_k\bm{y}_*^k}}{\lvert \bm{y}_*^k\rvert}\circ (\bm{y}_{\text{DR}}^k-\alpha_k\bm{y}_*^k), \bm{\theta}\right\rangle=0.
\end{equation*}
\end{proof}

\subsection{Local Convergence of GPS}
\label{subsec:GPS}
At the $k$-th step, given $\bm{z}_{\text{DR}}^k$, using the fixed point $P_{\mathcal{F}}\bm{z}_{\text{DR}}^k=\begin{pmatrix}
  \alpha_k\bm{x}^{\natural}\\\alpha_k\bm{y}_*^k
\end{pmatrix}$, we have the recursive relation
\begin{equation}
  \label{eq:conc}
\begin{aligned}
  \begin{pmatrix}
    \bm{x}_{\text{DR}}^{k+1}-\alpha_k\bm{x}^{\natural}\\\bm{y}_{\text{DR}}^{k+1}-\alpha_k\bm{y}_*^k
  \end{pmatrix}
  &=\begin{pmatrix}
    \bm{x}_{\text{DR}}^{k}-\alpha_k\bm{x}^{\natural}\\\bm{y}_{\text{DR}}^{k}-\alpha_k\bm{y}_*^k
  \end{pmatrix}+\bm{C}^*\bm{C}\left(2\prox_{\phi}\begin{pmatrix}
    \bm{x}_{\text{DR}}^{k}\\\bm{y}_{\text{DR}}^{k}
  \end{pmatrix}-2\prox_{\phi}\begin{pmatrix}
    \alpha_k\bm{x}^{{\natural}}\\\alpha_k\bm{y}_*^k
  \end{pmatrix}-\begin{pmatrix}
    \bm{x}_{\text{DR}}^{k}-\alpha_k\bm{x}^{\natural}\\\bm{y}_{\text{DR}}^{k}-\alpha_k\bm{y}_*^k
  \end{pmatrix}\right)\\
  &{\phantom{==}}-\left(\prox_{\phi}
  \begin{pmatrix}
    \bm{x}_{\text{DR}}^{k}\\\bm{y}_{\text{DR}}^{k}
  \end{pmatrix}-\prox_{\phi}
  \begin{pmatrix}
    \alpha_k\bm{x}^{\natural}\\\alpha_k\bm{y}_*^k
  \end{pmatrix}
\right).
\end{aligned}
\end{equation}

Note that $\prox_g(\bm{x}_{\text{DR}}^k)=\bm{x}_{\text{DR}}^k$. At the $k$-th step, let $\bm{B}=\bm{C}\bm{\Omega}$,
where $\bm{\Omega}=\begin{pmatrix}
    \bm{\Omega}_x & \\ & \bm{\Omega}_y
  \end{pmatrix}$, and $\bm{\Omega}_x=\diag\left(\frac{\alpha_k\bm{x}^{\natural}}{\lvert \bm{x}^{\natural}\rvert}\right),\bm{\Omega}_y=\diag\left(\frac{\alpha_k\bm{y}_*^k}{\lvert \bm{y}_*^k\rvert}\right)$. Hence the matrix $\bm{B}$ depends on the iteration number $k$, but we do not explicitly show the dependence to simplify the notation. The proximity term in~\eqref{eq:conc} is nonlinear which has the following directional derivative\footnote{Here we exploit the $\mathbb{C}-\mathbb{R}$ calculus of real-valued function in complex-valued variables, please refer to~\cite{Kreutz-Delgado2009}.} for $\bm{z}\in\mathbb{C}^n$ and $\bm{h}\in\mathbb{C}^n$, 
\begin{equation}
  \label{eq:phasediff}
  D\left(\frac{\bm{z}}{\lvert \bm{z}\rvert}\right)(\bm{h}) = \frac{\bm{h}}{\lvert \bm{z}\rvert}-\frac{\bm{z}\re(\overline{\bm{z}}\bm{h})}{\lvert \bm{z}\rvert^3}=\frac{i\bm{z}\im(\overline{\bm{z}}\bm{h})}{\lvert \bm{z}\rvert^3}.
\end{equation}
Hence, the linear approximation to~\eqref{eq:conc} around point $
\begin{pmatrix}
  \alpha_k\bm{x}^{\natural}\\\alpha_k\bm{y}_*^k
\end{pmatrix}
$ reads
\begin{equation*}
   \begin{pmatrix}
    \bm{x}_{\text{DR}}^{k+1}-\alpha_k\bm{x}^{\natural}\\\bm{y}_{\text{DR}}^{k+1}-\alpha_k\bm{y}_*^k
  \end{pmatrix}=\bm{\Omega}(\bm{I}-\bm{B}^*\bm{B})\begin{pmatrix}
    \bm{\Omega}_x^*(\bm{x}_{\text{DR}}^{k}-\alpha_k\bm{x}^{\natural})\\\bm{\Omega}_y^*(\bm{y}_{\text{DR}}^{k}-\alpha_k\bm{y}_*^k)
  \end{pmatrix}+\bm{\Omega}(2\bm{B}^*\bm{B}-\bm{I})
  \begin{pmatrix}
    \bm{\Omega}_x^*(\bm{x}_{\text{DR}}^{k}-\alpha_k\bm{x}^{\natural})\\ \diag\left(\frac{\bm{b}}{\lvert \bm{y}_*^k\rvert}\right)i\im\left(\bm{\Omega}_y^*(\bm{y}_{\text{DR}}^{k}-\alpha_k\bm{y}_*^k)\right)
  \end{pmatrix}+\text{h.o.t.}
\end{equation*}
Denote $\bm{w}^{k+1}=\begin{pmatrix}
  \bm{\Omega}_x^*(\bm{x}_{\text{DR}}^{k+1}-\alpha_k\bm{x}^{\natural})\\ \bm{\Omega}_y^*(\bm{y}_{\text{DR}}^{k+1}-\alpha_k\bm{y}_*^k)
\end{pmatrix}$ and $\bm{v}^{k}=
\begin{pmatrix}
  \bm{\Omega}_x^*(\bm{x}_{\text{DR}}^{k}-\alpha_k\bm{x}^{\natural})\\ \bm{\Omega}_y^*(\bm{y}_{\text{DR}}^{k}-\alpha_k\bm{y}_*^k)
\end{pmatrix}
$, then the iteration scheme becomes
\begin{align*}
\bm{w}^{k+1} &=(\bm{I}-\bm{B}^*\bm{B})\begin{pmatrix}
  \bm{\Omega}_x^*(\bm{x}_{\text{DR}}^{k}-\alpha_k\bm{x}^{\natural})\\ \bm{\Omega}_y^*(\bm{y}_{\text{DR}}^{k}-\alpha_k\bm{y}_*^k)
\end{pmatrix}
+(2\bm{B}^*\bm{B}-\bm{I})
  \begin{pmatrix}
    \bm{\Omega}_x^*(\bm{x}_{\text{DR}}^k-\alpha_k\bm{x}^{\natural})\\ \diag\left(\frac{\bm{b}}{\bm{b}+\bm{\theta}_k}\right)i\im\left(\bm{\Omega}_y^*(\bm{y}_{\text{DR}}^{k}-\alpha_k\bm{y}_*^k)\right)
  \end{pmatrix}+\text{h.o.t.}\\
&:=J_k(\bm{v}^k)+\text{h.o.t.}.
\end{align*}
Let
$  J_0(\bm{v}^k) = (\bm{I}-\bm{B}^*\bm{B})\begin{pmatrix}
  \bm{\Omega}_x^*(\bm{x}_{\text{DR}}^{k}-\alpha_k\bm{x}^{\natural})\\ \bm{\Omega}_y^*(\bm{y}_{\text{DR}}^{k}-\alpha_k\bm{y}_*^k)
\end{pmatrix}+(2\bm{B}^*\bm{B}-\bm{I})
  \begin{pmatrix}
    \bm{\Omega}_x^*(\bm{x}_{\text{DR}}^k-\alpha_k\bm{x}^{\natural})\\ i\im\left(\bm{\Omega}_y^*(\bm{y}_{\text{DR}}^{k}-\alpha_k\bm{y}_*^k)\right)
  \end{pmatrix},
  $
we infer from the continuity of the operator $J_k$ around $J_0$ that for any $\bm{\theta}\in\mathcal{C}$ such that $\norm{\bm{\theta}}\leq \epsilon_2$,
\begin{equation*}
  \norm{J_k-J_0}\leq \epsilon_1.
\end{equation*}
Given point $\bm{z}_{\text{DR}}^1$ and denote $\bm{z}^{\natural}=
\begin{pmatrix}
  \bm{x}^{\natural}\\\bm{y}^{\natural}
\end{pmatrix}
$, we assume $\norm{\bm{z}_{\text{DR}}^1-\bm{z}^{\natural}}<\epsilon_3<\frac{\epsilon_2}{2}<\frac{b_{\min}}{6}$, then $\bm{z}_*^1=P_{\mathcal{F}}\bm{z}_{\text{DR}}^1=(\alpha_1\bm{x}^{\natural},\alpha_1\bm{y}_*^1)$ satisfies $\norm{\alpha_1\bm{y}_*^1-\bm{y}^{\natural}}<2\norm{\bm{z}_{\text{DR}}^1-\bm{z}^{\natural}}<\epsilon_2<\frac{b_{\min}}{3}$,  and hence $\norm{\lvert\alpha_1\bm{y}_*^1\rvert-\lvert\bm{y}^{\natural}\rvert}<\epsilon_2$, a.k.a. $\norm{\bm{\theta}_1}<\epsilon_2$. The quantity $\epsilon_1,\epsilon_2$ and $\epsilon_3$ will be determined accordingly. Starting from this point, it can be shown that $\norm{\alpha_k\bm{y}_*^k-\bm{y}^{\natural}}\leq \epsilon_2$ for all $k\geq 1$. Therefore
\begin{equation*}
  \norm{\bm{w}^{k+1}}\leq \norm{J_k(\bm{v}^k)}+o(\norm{\bm{v}^k})\leq \norm{J_0(\bm{v}^k)}+(\epsilon_1+o(1))\norm{\bm{v}^k}.
\end{equation*}

Denote $\bm{v}^k=\re(\bm{v}^k)+i\im(\bm{v}^k)=
\begin{pmatrix}
  \bm{e}\\\bm{f}
\end{pmatrix}+i
\begin{pmatrix}
  \bm{g}\\\bm{h}
\end{pmatrix}
$, where $\bm{e},\bm{g}\in\mathbb{R}^{n}$ and $\bm{f},\bm{h}\in\mathbb{R}^{m}$, we have
\begin{align*}
  J_0(\bm{v}^k)&=(\bm{I}-\bm{B}^*\bm{B})\re(\bm{v}^k)+i\bm{B}^*\bm{B}\im(\bm{v}^k) +(2\bm{B}^*\bm{B}-\bm{I})\begin{pmatrix}\re\left(\bm{\Omega}_x^*(\bm{x}_{\text{DR}}^k-\alpha_k\bm{x}^{\natural})\right)\\\bm{0}\end{pmatrix}\\
  &=\bm{B}^*\bm{B}\begin{pmatrix}
  \bm{e}\\\bm{0}
\end{pmatrix}+(\bm{I}-\bm{B}^*\bm{B})\begin{pmatrix}
  \bm{0}\\\bm{f}
\end{pmatrix}+i\bm{B}^*\bm{B}\begin{pmatrix}
  \bm{g}\\\bm{h}
\end{pmatrix}\\
  & = \bm{B}^*\bm{B}\begin{pmatrix}
  \bm{e}+i\bm{g}\\i\bm{h}
\end{pmatrix}+(\bm{I}-\bm{B}^*\bm{B})\begin{pmatrix}
  \bm{0}\\\bm{f}
\end{pmatrix}.
\end{align*}
Since $\bm{B}^*$ is isometric, $\bm{B}^*\bm{B}$ is a projection. It is easy to see that
\begin{equation*}
  \norm{J_0(\bm{v}^k)}^2 = \norm{\bm{B}^*\bm{B}\begin{pmatrix}
  \bm{e}+i\bm{g}\\i\bm{h}
\end{pmatrix}}^2+\norm{(\bm{I}-\bm{B}^*\bm{B})\begin{pmatrix}
  \bm{0}\\\bm{f}
\end{pmatrix}}^2\leq \norm{\bm{e}}^2+\norm{\bm{g}}^2+\norm{\bm{h}}^2+\norm{\bm{f}}^2=\norm{\bm{v}^k}^2.
\end{equation*} 

One can further bound  $\norm{J_0(\bm{v}^k)}$ from the following facts. 
\begin{lem}
  Let $\bm{e}=\re(\bm{\Omega}_x^*(\bm{x}_{\text{DR}}^{k}-\alpha_k\bm{x}^{\natural}))$, $\bm{f}=\re(\bm{\Omega}_y^*(\bm{y}_{\text{DR}}^{k}-\alpha_k\bm{y}_*^k))$ and provided $\bm{z}_{\text{DR}}^k\in\mathcal{V}_k$, there exist two constants $0<\delta_1,\delta_2<1$, such that 
  \begin{equation*}
   \norm{\bm{B}^*\bm{B}
  \begin{pmatrix}
    \bm{e}\\\bm{0}
  \end{pmatrix}}\leq \frac{1}{\sqrt{1+(s_{\min}(\bm{A}))^2}}\norm{\bm{e}}:=\delta_1\norm{\bm{e}},\qquad
\norm{\bm{B}^*\bm{B}\begin{pmatrix}
    \bm{0}\\\bm{f}
  \end{pmatrix}}\geq \frac{s_{\min}(G(\bm{A}\bm{\Omega}_y))}{1+(s_{\max}(\bm{A}))^2}\norm{\bm{f}}:=\delta_2\norm{\bm{f}},
\end{equation*}
where $s_{\max}(\bm{A})$ and $s_{\min}(\bm{A})$ are the largest and smallest non-zero singular values of matrix $\bm{A}$ respectively and $G(\bm{A})=
\begin{pmatrix}
  \re(\bm{A})\\\im(\bm{A})
\end{pmatrix}
$.
\end{lem}

\begin{proof}
  By direct calculation and the isometry of $\bm{C}$, 
  \begin{equation*}
    \norm{\bm{B}^*\bm{B}
  \begin{pmatrix}
    \bm{e}\\\bm{0}
  \end{pmatrix}}=\norm{\bm{\Omega}^*\bm{C}^*\bm{C}\bm{\Omega}\begin{pmatrix}
    \bm{e}\\\bm{0}
  \end{pmatrix}}=\norm{\bm{C}\bm{\Omega}\begin{pmatrix}
    \bm{e}\\\bm{0}
  \end{pmatrix}} = \norm{\bm{L}^{-1}\bm{\Omega}_x\bm{e}},
\end{equation*}
where $\bm{L}$ satisfies $\bm{L}\bm{L}^*=\bm{I}+\bm{A}\bm{A}^*$. By the relation of the eigenvalues of $\bm{I}+\bm{A}\bm{A}^*$ and the singular values of $\bm{A}$, which has full rank, one has the first bound. For the second inequality, we write down its explicit expression
\begin{equation*}
   \bm{B}^*\bm{B}
  \begin{pmatrix}
    \bm{0}\\\bm{f}
  \end{pmatrix}=
  \begin{pmatrix}
    \bm{\Omega}_x^*(\bm{I}+\bm{A}\bm{A}^*)^{-1}\bm{A}\bm{\Omega}_y\bm{f}\\
    \bm{\Omega}_y^*\bm{A}^*(\bm{I}+\bm{A}\bm{A}^*)^{-1}\bm{A}\bm{\Omega}_y\bm{f}
  \end{pmatrix}.
\end{equation*}
In principle, when the vector $\bm{f}$ is in the null space of $\bm{A}\bm{\Omega}_y$, the lower bound of the norm can vanish. By the positive definiteness of matrix $(\bm{I}+\bm{A}\bm{A}^*)^{-1}$, so
\begin{equation*}
  \norm{\bm{\Omega}_x^*(\bm{I}+\bm{A}\bm{A}^*)^{-1}\bm{A}\bm{\Omega}_y\bm{f}}\geq \frac{1}{1+(s_{\max}(\bm{A}))^2}\norm{\bm{A}\bm{\Omega}_y\bm{f}}.
\end{equation*}

By Lemma~\ref{lem:orth2}, we have that $\bm{f}\perp \bm{\theta},\forall \bm{\theta}\in\mathcal{C}'$, so
\begin{equation*}
  \norm{\bm{A}\bm{\Omega}_y\bm{f}}=\norm{G(\bm{A}\bm{\Omega}_y)\bm{f}}\geq s_{\min}(G(\bm{A}\bm{\Omega}_y))\norm{\bm{f}},
\end{equation*}
where $s_{\min}(G(\bm{A}\bm{\Omega}_y))$ denotes the smallest nonzero singular value of $G(\bm{A}\bm{\Omega}_y)$. Note that $s_{\min}(G(\bm{A}\bm{\Omega}_y))$ is the same for different iteration number $k$. Therefore,
\begin{equation*}
  \norm{\bm{B}^*\bm{B}\begin{pmatrix}
    \bm{0}\\\bm{f}
  \end{pmatrix}}\geq \frac{s_{\min}(G(\bm{A}\bm{\Omega}_y))}{1+(s_{\max}(\bm{A}))^2}\norm{\bm{f}}:=\delta_2\norm{\bm{f}},
\end{equation*}
It is also obvious that $0<\delta_1,\delta_2<1$ if $\bm{A}$ is column full rank.
\end{proof}
As a consequence, we have
\begin{equation*}
  \norm{(\bm{I}-\bm{B}^*\bm{B})
    \begin{pmatrix}
      \bm{0}\\\bm{f}
    \end{pmatrix}
}^2= \norm{\bm{f}}^2-\norm{\bm{B}^*\bm{B}\begin{pmatrix}
      \bm{0}\\\bm{f}
    \end{pmatrix}}^2\leq (1-\delta_2^2)\norm{\bm{f}}^2.
\end{equation*}
To bound the first norm, we will prove the following fact in next section.
\begin{lem}
  \label{lem:part2}
  With the specific expression of $\bm{v}^k=
  \begin{pmatrix}
    \bm{e}\\\bm{f}
  \end{pmatrix}+i
  \begin{pmatrix}
    \bm{g}\\\bm{h}
  \end{pmatrix}
  $, we have
  \begin{equation*}
  \norm{\bm{B}^*\bm{B}\begin{pmatrix}
  \bm{e}+i\bm{g}\\i\bm{h}
\end{pmatrix}}^2 \leq (\sigma_2^2+(1-\sigma_2^2)\delta_1^2)\norm{\bm{e}}^2+\sigma_2^2(\norm{\bm{g}}^2+\norm{\bm{h}}^2).
\end{equation*}
for some $0<\sigma_2<1$.
\end{lem}
Combining the above two Lemmas, we have 
\begin{equation*}
  \norm{J_0(\bm{v}^k)}^2\leq (\sigma_2^2+(1-\sigma_2^2)\delta_1^2)\norm{\bm{e}}^2+\sigma_2^2(\norm{\bm{g}}^2+\norm{\bm{h}}^2)+ (1-\delta_2^2)\norm{\bm{f}}^2.
\end{equation*}
Therefore, there exists $0<\beta<1$, such that
\begin{equation*}
  \norm{\bm{w}^{k+1}}\leq (\beta+\epsilon_1)\norm{\bm{v}^k}.
\end{equation*}
And we can choose an appropriate $\epsilon_1$ such that $\beta+\epsilon_1<1$. By the definition of projection onto $\mathcal{F}$, we have
\begin{equation*}
  \norm{\bm{v}^{k+1}}\leq\norm{\bm{w}^{k+1}}\leq (\beta+\epsilon_1)\norm{\bm{v}^k}.
\end{equation*}

The remaining thing needs to be shown is that when the iteration gets close enough to the true solution $\bm{z}^{\natural}$ (up to a global phase), then it will stay close so that $  \norm{J_k-J_0}\leq \epsilon_1$ is true for later iterations. Suppose $\norm{\bm{v}^1}\le \norm{\bm{z}_{\text{DR}}^1-\bm{z}^{\natural}}<\epsilon_3<\frac{\epsilon_2}{2}$, then $\norm{\alpha_1\bm{z}_*^1-\bm{z}^{\natural}}\leq 2\epsilon_3<\epsilon_2$, and for all $\alpha_k\bm{z}_*^k$ for $k=2,\ldots$, we show $\norm{\alpha_k\bm{z}_*^k-\bm{z}^{\natural}}\leq \epsilon_2$.  According to the projection property, it implies
\begin{align*}
  \norm{\alpha_{k+1}\bm{z}_*^{k+1}-\alpha_k\bm{z}_*^k}&=\norm{(\alpha_{k+1}\bm{z}_*^{k+1}-\bm{z}_{\text{DR}}^{k+1})+(\bm{z}_{\text{DR}}^{k+1}-\alpha_k\bm{z}_*^k)}\\
                                                      & \leq \norm{\bm{z}_{\text{DR}}^{k+1}-\alpha_{k+1}\bm{z}_*^{k+1}}+\norm{\bm{z}_{\text{DR}}^{k+1}-\alpha_k\bm{z}_*^k}\\
                                                      &\leq 2\norm{\bm{z}_{\text{DR}}^{k+1}-\alpha_k\bm{z}_*^k}\leq 2(\beta+\epsilon_1) \norm{\bm{v}^k}.
\end{align*}
Therefore,
\begin{align*}
  \norm{\alpha_{k+1}\bm{z}_*^{k+1}-\bm{z}^{\natural}}&=\norm{\alpha_{k+1}\bm{z}_*^{k+1}-\alpha_{k}\bm{z}_*^{k}+\alpha_{k}\bm{z}_*^{k}-\bm{z}^{\natural}}\\
  & \leq \norm{\alpha_{k+1}\bm{z}_*^{k+1}-\alpha_{k}\bm{z}_*^{k}}+\norm{\alpha_{k}\bm{z}_*^{k}-\bm{z}^{\natural}}.
\end{align*}
Iterating the above inequality backward, we have
\begin{align*}
  \norm{\alpha_{k+1}\bm{z}_*^{k+1}-\bm{z}^{\natural}}&\leq \norm{\alpha_{1}\bm{z}_*^{1}-\bm{z}^{\natural}}+\sum_{j=1}^k\norm{\alpha_{j+1}\bm{z}_*^{j+1}-\alpha_{k}\bm{z}_*^{j}}\\
                                                       & \leq \norm{\alpha_{1}\bm{z}_*^{1}-\bm{z}_{DR}^1}+\norm{\bm{z}_{DR}^1-\bm{z}^{\natural}}+2(\beta+\epsilon_1) \sum_{j=1}^k\norm{\bm{v}^j}\\
                                                     &< 2\epsilon_3+2\norm{\bm{v}^1}\sum_{j=1}^k(\beta+\epsilon_1)^j\\
                                                       &< \frac{2\epsilon_3}{1-(\beta+\epsilon_1)}.
\end{align*}
Hence, one can choose $\epsilon_3=\frac{1-(\beta+\epsilon_1)}{2}\epsilon_2$ small enough such that $\alpha_k\bm{y}_*^k$ is uniformly close to $\bm{y}^{\natural}$ such that $\norm{J_k-J_0}\leq \epsilon_1$.

Note that when $\norm{\bm{z}_{\text{DR}}^1-\bm{z}^{\natural}}$ is small enough, both $\norm{\bm{v}^k}$ and $\norm{\alpha_k\bm{z}_*^k-\bm{z}^{\natural}}$ can be controlled. Therefore, the high order term along with the linear approximation of $J_k(\bm{v}^k)$ at each step can also be uniformly bounded. From the above contraction property of the DR iteration~\eqref{eq:iter}, we have the following linear convergence bound for $\dist(\bm{x}^{k},\bm{x}^{\natural}):=\min_{\lvert\alpha\rvert=1}\norm{\bm{x}^{k}-\alpha\bm{x}^{\natural}}$
\begin{equation*}
  \dist(\bm{x}^{k},\bm{x}^{\natural})\leq \norm{\bm{x}^{k}-\alpha_{k}\bm{x}^{\natural}}=\norm{\bm{x}_{\text{DR}}^{k}-\alpha_{k}\bm{x}^{\natural}}\leq\norm{\bm{v}^k}\leq (\beta+\epsilon_1)^{k-1}\norm{\bm{v}^1}.
\end{equation*}
Thus the convergence of $\bm{x}^k$ to $\bm{x}^{\natural}$ (up to a global phase) can be ensured.

Another thing needs to verify is that $\bm{x}_{\text{DR}}^k\in\mathcal{V}_k$, i.e., $\norm{\bm{x}_{\text{DR}}^k-\alpha_k\bm{z}^{\natural}}\leq\frac{b_{\min}}{2}$ for $k\geq 1$, where $\alpha_k$ is determined by $P_{\mathcal{F}}\bm{z}_{\text{DR}}^k$. Note that $\norm{\bm{x}_{\text{DR}}^k-\alpha_k\bm{z}^{\natural}}\leq \norm{\bm{x}_{\text{DR}}^k-\alpha_k\bm{z}_*^k}+\norm{\bm{\theta}_k}\leq \frac{b_{\min}}{6}+\frac{b_{\min}}{3}\leq\frac{b_{\min}}{2}$. 
Here is the main result for the noiseless case of phase retrieval problem. 
\begin{theorem}
\label{thm:main}
For problem~\eqref{eq:prob} in noiseless case ($\epsilon_i^{\text{noise}}=0,i=1,\ldots,m$) with the sampling vectors satisfying the assumption~\ref{assu:1}, if $\norm{\bm{z}_{\text{DR}}^1-\bm{z}^{\natural}}\leq \epsilon_3<\frac{b_{\min}}{6}$, where $\epsilon_3$ depending on $\epsilon_1$ (such that $\norm{J_k-J_0}\leq \epsilon_1$) is small enough. We have that, for GPS~\eqref{eq:gps}, there exists a constant $0<\gamma<1$ such that
\begin{equation*}
  \dist(\bm{x}^k,\bm{x}^{\natural})\leq \gamma^{k-1}\norm{\bm{z}_{\text{DR}}^1-\bm{z}_*^1}\leq \gamma^{k-1}\norm{\bm{z}_{\text{DR}}^1-\bm{z}^{\natural}},
\end{equation*}
where
\begin{equation*}
  \bm{z}_{\text{DR}}^k =
  \begin{pmatrix}
    \bm{\lambda}^k+\bm{x}^k\\
    \bm{\nu}^k+\bm{y}^k
  \end{pmatrix}
\quad \text{ and } \bm{z}_*^k = P_{\mathcal{F}}\bm{z}_{\text{DR}}^k.
\end{equation*}
\end{theorem}

\begin{remark}
 In the basin of local convergence, $\dist(\bm{x}^k,\bm{x}^{\natural})$ is generally not monotonically decreasing, it may oscillate. However, from Theorem~\ref{thm:main}, it goes to zero in a controlled way that it is bounded by a linear convergence.
\end{remark}

\subsection{Error Bound in Noisy Case}

In noisy case, the measurement $\tilde{\bm{b}}=\bm{b}+\bm{\epsilon}^{\text{noise}}$, then we have the following result. 

\begin{theorem}[Noisy Case]
  Assume that the sampling vectors $\{\bm{a}_i\}_{i=1}^m$ satisfy assumption~\ref{assu:1}, and the measurements $\bm{b}$ are corrupted by noise $\bm{\epsilon}^{\text{noise}}$. If $\norm{\bm{z}_{\text{DR}}^1-\bm{z}_*^1}$ is small enough but greater than the noise level, i.e., $\norm{\bm{\epsilon}^{\text{noise}}}$, furthermore, the noise level $\norm{\bm{\epsilon}^{\text{noise}}}$ is small enough such that $\bm{z}_{\text{DR}}^k$ is still in $\mathcal{V}^k$ and $\norm{\alpha_k\bm{z}_*^k-\bm{z}^{\natural}}\leq \epsilon_2$, we have
  \begin{equation}
    \label{eq:loc-dr-noise}
    \dist(\bm{x}^{k},\bm{x}^{\natural})\leq (\beta+\epsilon_1)^{k-1}\norm{\bm{z}_{\text{DR}}^1-\bm{z}_*^1}+\frac{1}{1-(\beta+\epsilon_1)}\norm{\bm{\epsilon}^{\text{noise}}}.
  \end{equation}
\end{theorem}
\begin{proof}
We first compute
\begin{align*}
\bm{w}^{k+1}&:=\begin{pmatrix}
  \bm{\Omega}_x^*(\bm{x}_{\text{DR}}^{k+1}-\alpha_k\bm{x}^{\natural})\\ \bm{\Omega}_y^*(\bm{y}_{\text{DR}}^{k+1}-\alpha_k\bm{y}_*^k)
\end{pmatrix}\\
 &=J_k(\bm{v}^k)+(2\bm{B}^*\bm{B}-\bm{I})
  \begin{pmatrix}
    \bm{0}\\\bm{\Omega}_y^*\left(\bm{\epsilon}^{\text{noise}}\circ \frac{\bm{y}_{\text{DR}}^k}{\lvert \bm{y}_{\text{DR}}^k\rvert}\right)
  \end{pmatrix}
+\text{h.o.t.}
\end{align*}
Using the same argument as for the noiseless case, we have
\begin{equation}
\begin{aligned}
  \norm{\bm{v}^{k+1}}&\leq\norm{\bm{w}^{k+1}}\\
  &\leq (\beta+\epsilon_1)\norm{\bm{v}^k} + \norm{\bm{\epsilon}^{\text{noise}}}.
\end{aligned}
\end{equation}
Iterating the above inequality backward, we get
\begin{equation*}
  \norm{\bm{v}^k}\leq (\beta+\epsilon_1)^{k-1}\norm{\bm{v}^1} + \norm{\bm{\epsilon}^{\text{noise}}}\sum_{j=0}^{k-2}(\beta+\epsilon_1)^j\leq (\beta+\epsilon_1)^{k-1}\norm{\bm{v}^1} + \frac{1}{1-(\beta+\epsilon_1)}\norm{\bm{\epsilon}^{\text{noise}}}.
\end{equation*}
By $\dist(\bm{x}^k,\bm{x}^{\natural})\leq \norm{\bm{v}^k}$, we obtain the estimation of reconstruction in noisy case. 
\end{proof}

Applying GPS for noisy phase retrieval, its performance heavily depends on the noisy level. The requirement on noisy level is very strict. Thus for noisy case, RGPS is more appropriate and the performance is much better. 
\subsection{Local Convergence of RGPS}
\label{subsec:RGPS}
Considering the robust GPS (RGPS) for noiseless setting, the operator matrix $\bm{M}$ in~\eqref{eq:iter} is replaced by the following matrix
\begin{equation*}
  \tilde{\bm{M}}=t\bm{I}+(1-t)\bm{M}, 0< t< t_{\max}<1.
\end{equation*}

When considering the local convergence of RGPS via equivalent Douglas-Rachford iteration, we note that the fixed point set of the iteration scheme is no longer the set $\mathcal{F}$ of GPS. Thus local convergence of RGPS is different from the above analysis.

Define set $\mathcal{F}_0=\left\{
  \begin{pmatrix}
    \alpha\bm{x}^{\natural}\\
    \alpha\bm{y}^{\natural}
  \end{pmatrix}\Bigm |\lvert\alpha\rvert=1
\right\}$, we can check that elements of $\mathcal{F}_0$ belong to the fixed point set of RGPS~\footnote{We can show that the elements of $\mathcal{F}$ with $\bm{\theta}\neq 0$ is not a fixed point of RGPS. Besides $\mathcal{F}_0$, there may exist other nontrivial fixed points.}. In the following, we investigate the local convergence of RGPS around the point $\bm{z}^{\natural}=\begin{pmatrix}
  \bm{x}^{\natural}\\
  \bm{y}^{\natural}
\end{pmatrix}$. And we assume there is no other fixed points around the solution up to a global phase shift. At $k$-th step, following the same argument as before, given $\bm{z}_{\text{DR}}^k$, we denote $P_{\mathcal{F}_0}(\bm{z}_{\text{DR}}^k)=
\begin{pmatrix}
  \alpha_k\bm{x}^{\natural}\\
  \alpha_k\bm{y}^{\natural}
\end{pmatrix}
$ and $\bm{\Omega}=
\begin{pmatrix}
  \bm{\Omega}_x & \\
  & \bm{\Omega}_y
\end{pmatrix}=\diag\left(\frac{\alpha_k\bm{x}^{\natural}}{\lvert \bm{x}^{\natural}\rvert},\frac{\alpha_k\bm{y}^{\natural}}{\lvert \bm{y}^{\natural}\rvert}\right)$, we have the iteration scheme ($\bm{B}=\bm{C}\bm{\Omega}$)
\begin{equation*}
\begin{aligned}
     \begin{pmatrix}
    \bm{x}_{\text{DR}}^{k+1}-\alpha_k\bm{x}^{\natural}\\\bm{y}_{\text{DR}}^{k+1}-\alpha_k\bm{y}^{\natural}
  \end{pmatrix}&=\bm{\Omega}(1-t)(\bm{I}-\bm{B}^*\bm{B})\begin{pmatrix}
    \bm{\Omega}_x^*(\bm{x}_{\text{DR}}^{k}-\alpha_k\bm{x}^{\natural})\\\bm{\Omega}_y^*(\bm{y}_{\text{DR}}^{k}-\alpha_k\bm{y}^{\natural})
  \end{pmatrix}\\
 & {}+\bm{\Omega}((2(1-t))\bm{B}^*\bm{B}-(1-2t)\bm{I})
  \begin{pmatrix}
    \bm{\Omega}_x^*(\bm{x}_{\text{DR}}^k-\alpha_k\bm{x}^{\natural})\\ i\im\left(\bm{\Omega}_y^*(\bm{y}_{\text{DR}}^{k}-\alpha_k\bm{y}^{\natural})\right)
  \end{pmatrix}+\text{h.o.t.}\\
  &:=\bm{\Omega}J_0^t(\tilde{\bm{v}}^k)+\text{h.o.t.}
\end{aligned}
\end{equation*}
where 
\begin{align*}
  \tilde{\bm{v}}^k&=
  \begin{pmatrix}
    \bm{\Omega}_x^*(\bm{x}_{\text{DR}}^k-\alpha_k\bm{x}^{\natural})\\
    \bm{\Omega}_y^*(\bm{y}_{\text{DR}}^{k}-\alpha_k\bm{y}^{\natural})
  \end{pmatrix}=
  \begin{pmatrix}
    \frac{\overline{\alpha_k\bm{x}^{\natural}}}{\lvert \bm{x}^{\natural}\rvert}\circ (\bm{x}_{\text{DR}}^k-\alpha_k\bm{x}^{\natural})\\
 \frac{\overline{\alpha_k\bm{y}^{\natural}}}{\lvert \bm{y}^{\natural}\rvert} \circ (\bm{y}_{\text{DR}}^{k}-\alpha_k\bm{y}^{\natural})
\end{pmatrix}\\
&=\re(\tilde{\bm{v}}^k)+\im(\tilde{\bm{v}}^k)\\
&:=
\begin{pmatrix}
  \tilde{\bm{e}}\\
  \tilde{\bm{f}}
\end{pmatrix}+i
\begin{pmatrix}
  \tilde{\bm{g}}\\
  \tilde{\bm{h}}
\end{pmatrix}.
\end{align*}
We still compute
\begin{equation*}
\begin{aligned}
  J_0^t(\tilde{\bm{v}}^k)&=(1-t)(\bm{I}-\bm{B}^*\bm{B})\begin{pmatrix}
  \bm{0}\\\tilde{\bm{f}}
\end{pmatrix}+\left(t\bm{I}+(1-t)\bm{B}^*\bm{B}\right)\begin{pmatrix}
  \tilde{\bm{e}}+i\tilde{\bm{g}}\\i\tilde{\bm{h}}
\end{pmatrix}\\
& = (1-t)J_0(\tilde{\bm{v}}^k)+t\begin{pmatrix}
  \tilde{\bm{e}}+i\tilde{\bm{g}}\\i\tilde{\bm{h}}
\end{pmatrix}.
\end{aligned}
\end{equation*}
Therefore, we have
\begin{equation*}
  \norm{J_0^t(\tilde{\bm{v}}^k)}\leq (1-t)\norm{J_0(\tilde{\bm{v}}^k)}+t\norm{
    \begin{pmatrix}
      \tilde{\bm{e}}+i\tilde{\bm{g}}\\i\tilde{\bm{h}}
    \end{pmatrix}
}.
\end{equation*}
Using the same argument for Lemma~\ref{lem:part2}, there exists $0<\tilde{\beta}<1$, such that
\begin{align*}
  \norm{J_0(\tilde{\bm{v}}^k)}^2&\leq (\sigma_2^2+(1-\sigma_2^2)\delta_1^2)\norm{\tilde{\bm{e}}}^2+\sigma_2^2(\norm{\tilde{\bm{g}}}^2+\norm{\tilde{\bm{h}}}^2)+ \norm{\tilde{\bm{f}}}^2\\
                                &\leq \tilde{\beta}^2(\norm{\tilde{\bm{e}}}^2+\norm{\tilde{\bm{g}}}^2+\norm{\tilde{\bm{h}}}^2)+\norm{\tilde{\bm{f}}}^2.
\end{align*}
However, due to the possibility that $\tilde{\bm{f}}$ may be in the null space of $\tilde{\bm{B}}$, one can only have unity before $\norm{\tilde{\bm{f}}}^2$.

Define $r(t) = (1-t)\sqrt{\tilde{\beta}^2(\norm{\tilde{\bm{e}}}^2+\norm{\tilde{\bm{g}}}^2+\norm{\tilde{\bm{h}}}^2)+\norm{\tilde{\bm{f}}}^2}+t\sqrt{\norm{\tilde{\bm{e}}}^2+\norm{\tilde{\bm{g}}}^2+\norm{\tilde{\bm{h}}}^2}$, then 
\begin{equation*}
  \norm{J_0^t(\bm{v}^k)}\leq \max\left(r(0),r(t_{\max})\right).
\end{equation*}
Now we show that if $t$ satisfies some condition, $0< t< t_{\max}=\frac{2(1-\tilde{\beta}^2)}{2-\tilde{\beta}^2}$, where $\tilde{\beta}^2=\max\{\sigma_2^2+(1-\sigma_2^2)\delta_1^2,\sigma_2^2\}$, 
 there exists a constant $\gamma<1$ such that $\norm{J_0^t(\bm{v}^k)}<\gamma\norm{\tilde{\bm{v}}^k}$, which implies local convergence for RGPS.

\begin{lem}
  When $0< t< t_{\max}= \frac{2\alpha_c}{\alpha_c+1}$ and $\alpha_c=1-\tilde{\beta}^2<1$, then we have that
  \begin{equation*}
    (1-t)\sqrt{\norm{\tilde{\bm{v}}^k}^2-\alpha_c\norm{\tilde{\bm{p}}^k}^2}+t\sqrt{\norm{\tilde{\bm{p}}^k}^2}\leq \sqrt{(1-t)^2+\frac{t^2}{\alpha_c}}\norm{\tilde{\bm{v}}^k}:=\gamma\norm{\tilde{\bm{v}}^k}.
  \end{equation*}
\end{lem}

\begin{proof}
  In this proof, for fixed $t$, we denote $f(u,s)=(1-t)\sqrt{u-\alpha_cs}+t\sqrt{s}$ where $0<s\leq u$. By studying the partial derivatives $\partial_uf$ and $\partial_sf$, we know that $f(u,s)$ is monotonically increasing with respect to $u$ and $f(u,s)$ attains maximum at $s= \frac{t^2u}{(1-t)^2\alpha_c^2+\alpha_ct^2}$. At this point $f(u,s)\leq \sqrt{(1-t)^2+\frac{t^2}{\alpha_c}} \sqrt{u}$. If $0<t< t_{\max}=\frac{2\alpha_c}{\alpha_c+1}$, then this scaling factor is strictly less than unity. 
\end{proof}

By the same argument as before for GPS, we can obtain the following local convergence result for RGPS.

\begin{theorem}
  For RGPS, if the parameter $t$ satisfies that $0<t<t_{\max}=\frac{2(1-\tilde{\beta}^2)}{2-\tilde{\beta}^2}<1$, where $\tilde{\beta}$ depends on the singular values of $\bm{A}$, and $\dist(\bm{z}_{\text{DR}}^1,\bm{z}^{\natural})=\norm{\tilde{\bm{v}}^1}$ is small enough such that $\mathcal{F}_0$ is the only fixed point set around $\bm{z}^{\natural}$, there is a $0<\gamma<1$ such that
  \begin{equation*}
    \dist(\bm{x}^k,\bm{x}^{\natural})\leq (\gamma+o(1))^{k-1}\norm{\tilde{\bm{v}}^1}.
  \end{equation*}
\end{theorem}

The estimation of reconstruction of RGPS in noisy case can also be derived. Note that the requirement of small noise level is minimized. In this view, there is no doubt that it outperforms GPS for noisy measurements.

\begin{remark}
  We note that the dynamic range of $t$ is generally narrow. That is to say the iteration of RGPS is not far away from GPS, where $t=0$. As shown by our numerical experiments, although a slight perturbation of GPS, RGPS stabilizes the iterations, in particular around the solution.
\end{remark}

We conclude that all of GPS, RGPS, DR and RDR exhibit local convergence when they are applied to phase retrieval with $g(\bm{x})=0$. The local convergence of DR and RDR can be derived by mimicking the above analysis. In general, GPS (DR) seems to converge to the contraction basin near the solution easier and faster due to exact graph projection. However, near the solution, GPS (DR) tends to oscillate more than its robust version, RGPS (RDR). This is because the fixed point set for GPS (DR) $\mathcal{F}$ may not be simple, i.e., more than $\bm{z}^{\natural}$ (up to a global phase), especially when the number of measurement is large compared to the dimension of the signal. In this case, moreover, $\mathcal{F}$ connect all the way continuously to $\bm{z}^{\natural}$. Compared to DR, GPS tends to escape stagnation stage easier due to the graph projection in $(\bm{x},\bm{y})$ space.

\section{Technical Details for Proof}

We provide the main technical detail used in previous proofs in this section. The key is to show the eigen structure of the isometric matrix $\bm{B}$. The local convergence analysis for specific Douglas-Rachford algorithms for Fourier phase retrieval problem can be found in~\cite{Li162,Chen2015}. However, their convergence results require $m\geq 2n$ while our analysis is valid for arbitrary $m$ and $n$. Here we also prove convergence of the newly proposed robust version.

First, we study the singular values associated with the matrix $\bm{B}$. Note that matrix $\bm{B}^*$ is isometric, then $\norm{\bm{B}^*\bm{x}}=\norm{\bm{x}}$, i.e, $\bm{B}\bm{B}^*=\bm{I}$. Note that $\bm{B}\in\mathbb{C}^{n\times (n+m)}$ (here $m\geq n$, since we have no additional information on $\bm{x}$),\footnote{We can also show the local convergence for $m\leq n$ case using slightly modified analysis.} we write down the real form of complex matrix $\bm{B}$:
\begin{equation*}
  \mathcal{B}:=
  \begin{pmatrix}
    \re(\bm{B})\\\im(\bm{B})
  \end{pmatrix}\in\mathbb{R}^{2n\times (n+m)}.
\end{equation*}
Accordingly, we define the operator $G$, which maps a complex vector to its real and imaginary parts, i.e.,
\begin{equation*}
  G(\bm{z}):=
  \begin{pmatrix}
    \re(\bm{z})\\\im(\bm{z})
  \end{pmatrix}.
\end{equation*}

\begin{lem}
  By the definition of $G$, we have the following relation
\begin{equation}
\label{eq:real-comp}
  G(\bm{B}^*\bm{u})=
  \begin{pmatrix}
    \mathcal{B}^TG(\bm{u})\\\mathcal{B}^TG(-i\bm{u})
  \end{pmatrix}, \text{ where }\bm{u}\in\mathbb{C}^n.
\end{equation}
\end{lem}

\begin{proof}
  By straightforward calculation, we have 
  \begin{align*}
    G(\bm{B}^*\bm{u}) &= G\left((\re(\bm{B}^T)-i\im(\bm{B}^T))(\re(\bm{u})+i\im(\bm{u}))\right)\\
&=G\left(\left(\re(\bm{B}^T)\re(\bm{u})+\im(\bm{B}^T)\im(\bm{u})\right)+i\left(\re(\bm{B}^T)\im(\bm{u})-\im(\bm{B}^T)\re(\bm{u})\right)\right)\\
&=
    \begin{pmatrix}
      \re(\bm{B}^T)\re(\bm{u})+\im(\bm{B}^T)\im(\bm{u})\\
      \re(\bm{B}^T)\im(\bm{u})-\im(\bm{B}^T)\re(\bm{u})
    \end{pmatrix}.
  \end{align*}
\end{proof}

We assume that $\sigma_1\geq \sigma_2\geq\cdots\geq\sigma_{2n}\geq\sigma_{2n+1}=\cdots=\sigma_{n+m}=0$ ($m\geq n$) are the singular values of $\mathcal{B}$, and $\{\bm{\eta}_j\in\mathbb{R}^{n+m}\}_{j=1}^{n+m}$ and $\{\bm{\xi}_j\in\mathbb{R}^{2n}\}_{j=1}^{2n}$ are its associated right singular vectors and left singular vectors respectively. By the definition, for $j=1,\ldots,2n$, we have
\begin{align*}
  \mathcal{B}\bm{\eta}_j &= \sigma_j\bm{\xi}_j,\\
  \mathcal{B}^T\bm{\xi}_j& = \sigma_j\bm{\eta}_j.
\end{align*}
Accordingly, by denoting $\bm{\xi}_j =
  \begin{pmatrix}
    \bm{\xi}_j^R\\\bm{\xi}_j^I
  \end{pmatrix}$, we have
\begin{align*}
  \begin{pmatrix}
    \re(\bm{B})\bm{\eta}_j\\\im(\bm{B})\bm{\eta}_j
  \end{pmatrix}&=
                 \begin{pmatrix}
                   \sigma_j\bm{\xi}_j^R\\\sigma_j\bm{\xi}_j^I
                 \end{pmatrix},\\
  \begin{pmatrix}
    \re(\bm{B}^T) & \im(\bm{B}^T)
  \end{pmatrix}
                    \begin{pmatrix}
                      \bm{\xi}_j^R\\\bm{\xi}_j^I
                    \end{pmatrix}&=\sigma_j\bm{\eta}_j.
\end{align*}
Therefore, we have 
\begin{align*}
  \bm{B}\bm{\eta}_j & = \sigma_jG^{-1}(\bm{\xi}_j),\\
  \re(\bm{B}^T)\bm{\xi}_j^R+\im(\bm{B}^T)\bm{\xi}_j^I &= \sigma_j\bm{\eta}_j.
\end{align*}
The last equality can be written as
\begin{equation*}
  \re(\bm{B}^*G^{-1}(\bm{\xi}_j)) = \sigma_j\bm{\eta}_j.
\end{equation*}

By the isometry of $\bm{B}^*$, we can directly obtain the singular values $\sigma_1=1$ and $\sigma_{2n}=0$. Moreover, the following properties hold. 
\begin{theorem}[\!\!\cite{Chen2015}]
  For $i=1,\ldots,n$, the singular values $\{\sigma_i\}_{i=1}^{2n}$ and singular vectors $\{\bm{\xi}_i\}_{i=1}^{2n}$ have the following relations
  \begin{align*}
    1&=\sigma_i^2+\sigma_{2n+1-i}^2\\
    \bm{\xi}_{2n+1-i} &= G(-iG^{-1}(\bm{\xi}_i))\\
    \bm{\xi}_{i} & = G(iG^{-1}(\bm{\xi}_{2n+1-i})).
  \end{align*}
\end{theorem}

By the eigen structure $\{\sigma_1,\ldots,\sigma_{2n}\}$ of the matrix $\mathcal{B}$, we have the following relations.
\begin{lem}[\!\!\cite{Chen2015}]
  \label{lem:oper}
  For each $k=1,\ldots,n$, the following matrix-vector multiplication can be expressed as
  \begin{align*}
    \bm{B}^*\bm{B}\bm{\eta}_k &= \sigma_k\left(\sigma_k\bm{\eta}_k+i\sigma_{2n+1-k}\bm{\eta}_{2n+1-k}\right)\\
    \bm{B}^*\bm{B}\bm{\eta}_{2n+1-k} &= \sigma_{2n+1-k}\left(\sigma_{2n+1-k}\bm{\eta}_{2n+1-k}-i\sigma_{k}\bm{\eta}_{k}\right).
  \end{align*}
\end{lem}

In the next, we give the leading singular vector of $\mathcal{B}$ explicitly.

\begin{lem}
For the real matrix $\mathcal{B}\in\mathbb{R}^{2n\times(n+m)}$ generated by $\bm{B}=\bm{C}\bm{\Omega}_k$ with $\bm{\Omega}_k$ being $\diag
\begin{pmatrix}
  \frac{\alpha_k\bm{x}^{\natural}}{\lvert \bm{x}^{\natural}\rvert}\\
  \frac{\alpha_k\bm{A}^*\bm{x}^{\natural}}{\lvert \bm{A}^*\bm{x}^{\natural}\rvert}
\end{pmatrix}
$, the left singular vector corresponding to the leading singular value $\sigma_1=1$ is $\bm{\xi}_1=G(\bm{L}^*\alpha_k\bm{x}^{\natural})$ and the left singular vector corresponding to the least singular value $\sigma_{2n}=0$ is $\bm{\xi}_{2n}=G(-i\bm{L}^*\alpha_k\bm{x}^{\natural})$.
\end{lem}
\begin{proof}
First we have
  \begin{equation*}
    \bm{B}^*\bm{L}^*\alpha_k\bm{x}^{\natural} =
    \begin{pmatrix}
      \diag\left(\frac{\overline{\alpha_k\bm{x}^{\natural}}}{\lvert \bm{x}^{\natural}\rvert}\right) & \\ & \diag\left(\frac{\overline{\alpha_k\bm{A}^*\bm{x}^{\natural}}}{\lvert \bm{A}^*\bm{x}^{\natural}\rvert}\right)
    \end{pmatrix}
    \begin{pmatrix}
      (\bm{L}^{-1})^*\\\bm{A}^* (\bm{L}^{-1})^*
    \end{pmatrix}\bm{L}^*\alpha_k\bm{x}^{\natural}=
    \begin{pmatrix}
      \lvert \bm{x}^{\natural}\rvert\\\lvert \bm{A}^*\bm{x}^{\natural}\rvert
    \end{pmatrix},
  \end{equation*}
which says
\begin{equation*}
  \re(\bm{B}^*\bm{L}^*\alpha_k\bm{x}^{\natural})= \begin{pmatrix}
      \lvert \bm{x}^{\natural}\rvert\\\lvert \bm{A}^*\bm{x}^{\natural}\rvert
    \end{pmatrix}.
\end{equation*}
Similarly, it is easy to show that $B\begin{pmatrix}
      \lvert \bm{x}^{\natural}\rvert\\\lvert \bm{A}^*\bm{x}^{\natural}\rvert
    \end{pmatrix}=\bm{L}^*\alpha_k\bm{x}^{\natural}$.
Hence the leading singular value $\sigma_1=1$ and its associated left and right singular vectors are  $\bm{\xi}_1 = G(\bm{L}^*\alpha_k\bm{x}^{\natural})\in\mathbb{R}^{2n}$ and $\bm{\eta}_1=\begin{pmatrix}
      \lvert \bm{x}^{\natural}\rvert\\\lvert \bm{A}^*\bm{x}^{\natural}\rvert
    \end{pmatrix}\in\mathbb{R}^{n+m}$ respectively.  Furthermore, by the fact
  \begin{equation*}
    G(\bm{B}^*\bm{u})=
    \begin{pmatrix}
      \mathcal{B}^TG(\bm{u})\\
      \mathcal{B}^TG(-i\bm{u})
    \end{pmatrix}
\quad \mbox{and} \quad 
  \mathcal{B}^TG(-i\bm{L}^*\alpha_k\bm{x}^{\natural}) = 0,
\end{equation*}
we have $\sigma_{2n}=0$ and $\bm{\xi}_{2n}=G(-i\bm{L}^*\alpha_k\bm{x}^{\natural})$.
\end{proof}

\begin{lem}
Both $\bm{v}^k$ and $\tilde{\bm{v}}^k$, defined in Section \ref{subsec:GPS} and \ref{subsec:RGPS} respectively, provided both $\norm{\bm{v}^k}$ and $\norm{\tilde{\bm{v}}^k}$ are small enough, then they satisfy 
  $\im(\bm{v}^k)\perp\bm{\eta}_1,\im(\tilde{\bm{v}}^k)\perp\bm{\eta}_1$.
\end{lem}
\begin{proof}
According to the definition of $\tilde{\bm{v}}^{k}$, we have 
\begin{align*}
  \langle \tilde{\bm{v}}^k,i\bm{\eta}_1\rangle&=\left\langle \frac{\overline{\alpha_k\bm{x}^{\natural}}}{\lvert \bm{x}^{\natural}\rvert}(\bm{x}_{\text{DR}}^k-\alpha_k\bm{x}^{\natural}),i\lvert \bm{x}^{\natural}\rvert\right\rangle + \left\langle \frac{\overline{\alpha_k\bm{y}^{\natural}}}{\lvert \bm{y}^{\natural}\rvert}(\bm{y}_{\text{DR}}^k-\alpha_k\bm{y}^{\natural}),i\lvert \bm{y}^{\natural}\rvert\right\rangle\\
  &=\langle \bm{x}_{\text{DR}}^k-\alpha_k\bm{x}^{\natural},i\alpha_k\bm{x}^{\natural}\rangle + \langle \bm{y}_{\text{DR}}^k-\alpha_k\bm{y}^{\natural},i\alpha_k\bm{y}^{\natural}\rangle.
\end{align*}

Since $\alpha_k$ is determined by the projection, we have $\re\langle \tilde{\bm{v}}^k,i\bm{\eta}_1\rangle=0$, which says $\im(\tilde{\bm{v}}^k)\perp \bm{\eta}_1$. Likewise, 
\begin{align*}
  \langle \bm{v}^k,i\bm{\eta}_1\rangle&=\left\langle \frac{\overline{\alpha_k\bm{x}^{\natural}}}{\lvert \bm{x}^{\natural}\rvert}(\bm{x}_{\text{DR}}^k-\alpha_k\bm{x}^{\natural}),i\lvert \bm{x}^{\natural}\rvert\right\rangle + \left\langle \frac{\overline{\alpha_k\bm{y}^{\natural}}}{\lvert \bm{y}^{\natural}\rvert}(\bm{y}_{\text{DR}}^k-\alpha_k\bm{y}_*^k),i\lvert \bm{y}^{\natural}\rvert\right\rangle\\
  & = \left\langle \frac{\overline{\alpha_k\bm{x}^{\natural}}}{\lvert \bm{x}^{\natural}\rvert}(\bm{x}_{\text{DR}}^k-\alpha_k\bm{x}^{\natural}),i\lvert \bm{x}^{\natural}\rvert\right\rangle + \left\langle \frac{\overline{\alpha_k\bm{y}^{\natural}}}{\lvert \bm{y}^{\natural}\rvert}(\bm{y}_{\text{DR}}^k-\alpha_k\bm{y}_*^k),i\lvert \bm{y}_*^k\rvert-i\bm{\theta}_k\right\rangle\\
  &=\langle \bm{x}_{\text{DR}}^k-\alpha_k\bm{x}^{\natural},i\alpha_k\bm{x}^{\natural}\rangle + \langle \bm{y}_{\text{DR}}^k-\alpha_k\bm{y}_*^k,i\alpha_k\bm{y}_*^k\rangle-\langle \frac{\overline{\alpha_k\bm{y}_*^k}}{\lvert \bm{y}_*^k\rvert}(\bm{y}_{\text{DR}}^k-\alpha_k\bm{y}_*^k),i\bm{\theta}_k\rangle.
\end{align*}
By Lemma~\ref{lem:orth2} and the projection property, we also have $\re\langle \bm{v}^k,i\bm{\eta}_1\rangle=0$, i.e., $\im(\bm{v}^k)\perp \bm{\eta}_1$.
\end{proof}

\begin{proof}[Proof of Lemma 3.4]
  Since we have that
  \begin{equation*}
    \begin{pmatrix}
      \bm{g}\\\bm{h}
    \end{pmatrix}\perp \bm{\eta}_1,
  \end{equation*}
  we know that $\begin{pmatrix}
      i\bm{g}\\i\bm{h}
    \end{pmatrix}\in\text{span}\{i\bm{\eta}_2,\ldots,i\bm{\eta}_{n+m}\}$. For the vector $
    \begin{pmatrix}
      \bm{e}\\\bm{0}
    \end{pmatrix}
$, we can decompose it as $\bm{v}_1+\bm{v}_2$ such that $\bm{v}_1\in \text{span}\{\bm{\eta}_1\}$ and $\bm{v}_2\in\text{span}\{\bm{\eta}_2,\ldots,\bm{\eta}_{n+m}\}$. According to the operation $\bm{B}^*\bm{B}$ on the singular vectors, see Lemma~\ref{lem:oper}, we know that $\bm{B}^*\bm{B}\begin{pmatrix}
      \bm{e}\\\bm{0}
    \end{pmatrix}=\bm{v}_1+\bm{B}^*\bm{B}\bm{v}_2$ where $\bm{B}^*\bm{B}\bm{v}_2\in \text{span}\{\bm{\eta}_j,i\bm{\eta}_j|j=2,\ldots,2n-1\}$. Then
    \begin{align*}
      \norm{\bm{B}^*\bm{B}\begin{pmatrix}
  \bm{e}+i\bm{g}\\i\bm{h}
\end{pmatrix}}^2&=\norm{\bm{v}_1+\bm{B}^*\bm{B}\left(\bm{v}_2+\begin{pmatrix}
  i\bm{g}\\i\bm{h}
\end{pmatrix}\right)}^2\\
&=\norm{\bm{v}_1}^2+\norm{\bm{B}^*\bm{B}\left(\bm{v}_2+\begin{pmatrix}
  i\bm{g}\\i\bm{h}
\end{pmatrix}\right)}^2\\
&\leq \norm{\bm{v}_1}^2+\sigma_2^2(\norm{\bm{v}_2}^2+\norm{\bm{g}}^2+\norm{\bm{h}}^2),
    \end{align*}
where we use the fact $\bm{v}_2+\begin{pmatrix}
  i\bm{g}\\i\bm{h}
\end{pmatrix}\in \text{span}\{\bm{\eta}_j,i\bm{\eta}_j|j=2,\ldots,2n-1\}$.
Since 
\begin{align*}
  \norm{\bm{v}_1}^2+\norm{\bm{B}^*\bm{B}\bm{v}_2}^2&=\norm{\bm{v}_1+\bm{B}^*\bm{B}\bm{v}_2}^2\\
&=\norm{\bm{B}^*\bm{B}\begin{pmatrix}
  \bm{e}\\\bm{0}
\end{pmatrix}}^2\leq \delta_1^2\norm{\bm{e}}^2,
\end{align*}
we have
\begin{align*}
  \norm{\bm{v}_1}^2+\sigma_2^2(\norm{\bm{v}_2}^2+\norm{\bm{g}}^2+\norm{\bm{h}}^2)&= (1-\sigma_2^2)\norm{\bm{v}_1}^2+\sigma_2^2(\norm{\bm{v}_1}^2+\norm{\bm{v}_2}^2+\norm{\bm{g}}^2+\norm{\bm{h}}^2)\\
 &\leq (1-\sigma_2^2)\delta_1^2\norm{\bm{e}}^2+\sigma_2^2(\norm{\bm{e}}^2+\norm{\bm{g}}^2+\norm{\bm{h}}^2).
\end{align*}
This completes the proof.
\end{proof}

The following Lemma verifies that the second singular value $\sigma_2$ is strictly less than unity.

\begin{lem}
  Let $\bm{B}=\bm{C}\bm{\Omega}$, where $\bm{\Omega}=
  \begin{pmatrix}
   \frac{\alpha_k\bm{x}^{\natural}}{\lvert \bm{x}^{\natural}\rvert}  & \\
    & \frac{\alpha_k\bm{A}^*\bm{x}^{\natural}}{\lvert \bm{A}^*\bm{x}^{\natural}\rvert}
  \end{pmatrix}
$ where $\bm{C}=
\begin{pmatrix}
  \bm{L}^{-1} & \bm{L}^{-1}\bm{A}
\end{pmatrix}
\in\mathbb{C}^{n\times (m+n)}$. Then $\norm{\im(\bm{B}^*\bm{x})}=1$ holds for a unit vector $\bm{x}$ if and only if 
\begin{equation*}
  \begin{pmatrix}
    \frac{(\bm{L}^{-1})^*\bm{x}}{\lvert (\bm{L}^{-1})^*\bm{x}\rvert}\\
    \frac{\bm{A}^*(\bm{L}^{-1})^*\bm{x}}{\lvert\bm{A}^*(\bm{L}^{-1})^*\bm{x}\rvert}
  \end{pmatrix}=i\bm{\delta}\circ
  \begin{pmatrix}
    \frac{\bm{x}^{\natural}}{\lvert\bm{x}^{\natural}\rvert}\\
    \frac{\bm{A}^*\bm{x}^{\natural}}{\lvert\bm{A}^*\bm{x}^{\natural}\rvert}
  \end{pmatrix}
\end{equation*}
where the components of $\bm{\delta}$ are either 1 or -1. 
\end{lem}
\begin{proof}
  We have
  \begin{equation*}
  \norm{  \im(\bm{B}^*\bm{x})}=\norm{\im\left(
      \begin{matrix}
        \frac{\overline{\alpha_k\bm{x}^{\natural}}}{\lvert \bm{x}^{\natural}\rvert}\circ (\bm{L}^{-1})^*\bm{x}\\
        \frac{\overline{\alpha_k\bm{A}^*\bm{x}^{\natural}}}{\lvert \bm{A}^*\bm{x}^{\natural}\rvert}\circ \bm{A}^*(\bm{L}^{-1})^*\bm{x}
      \end{matrix}\right)}=\norm{\im\left( \begin{matrix}
        \frac{\overline{\alpha_k\bm{x}^{\natural}}}{\lvert \bm{x}^{\natural}\rvert}\circ \frac{(\bm{L}^{-1})^*\bm{x}}{\lvert (\bm{L}^{-1})^*\bm{x}\rvert}\circ \lvert (\bm{L}^{-1})^*\bm{x}\rvert\\
        \frac{\overline{\alpha_k\bm{A}^*\bm{x}^{\natural}}}{\lvert \bm{A}^*\bm{x}^{\natural}}\rvert \circ \frac{\bm{A}^*(\bm{L}^{-1})^*\bm{x}}{\lvert\bm{A}^*(\bm{L}^{-1})^*\bm{x}\rvert}\circ \lvert \bm{A}^*(\bm{L}^{-1})^*\bm{x}\rvert
      \end{matrix}\right)}\leq \norm{\bm{C}^*\bm{x}}= \norm{\bm{x}},
  \end{equation*}
as $\bm{C}^*$ is isometric. The equality holds if and only if
\begin{equation*}
  \begin{pmatrix}
    \frac{\overline{\alpha_k\bm{x}^{\natural}}}{\lvert \bm{x}^{\natural}\rvert}\circ \frac{(\bm{L}^{-1})^*\bm{x}}{\lvert (\bm{L}^{-1})^*\bm{x}\rvert}\\
        \frac{\overline{\alpha_k\bm{A}^*\bm{x}^{\natural}}}{\lvert \bm{A}^*\bm{x}^{\natural}}\rvert \circ \frac{\bm{A}^*(\bm{L}^{-1})^*\bm{x}}{\lvert\bm{A}^*(\bm{L}^{-1})^*\bm{x}\rvert}
  \end{pmatrix}= i\bm{\delta},
\end{equation*}
where components of $i\bm{\delta}$ are either $1$ or $-1$. So we have
\begin{equation*}
  \begin{pmatrix}
    \frac{(\bm{L}^{-1})^*\bm{x}}{\lvert (\bm{L}^{-1})^*\bm{x}\rvert}\\
    \frac{\bm{A}^*(\bm{L}^{-1})^*\bm{x}}{\lvert\bm{A}^*(\bm{L}^{-1})^*\bm{x}\rvert}
  \end{pmatrix}=i\bm{\delta}\circ
  \begin{pmatrix}
    \frac{\alpha_k\bm{x}^{\natural}}{\lvert\bm{x}^{\natural}\rvert}\\
    \frac{\alpha_k\bm{A}^*\bm{x}^{\natural}}{\lvert\bm{A}^*\bm{x}^{\natural}\rvert}
  \end{pmatrix}.
\end{equation*}
\end{proof}

From the uniqueness of the magnitude retrieval for sampling vectors $\{\bm{a}_i\}$ (Assumption~\ref{assu:1}), we have $\bm{x}=\pm i\frac{\bm{L}^*\alpha_k\bm{x}^{\natural}}{\lvert \bm{L}^*\bm{x}^{\natural}\rvert}$. By the result $\sigma_2= \max\{\norm{\im(\bm{B}^*\bm{u})}:\bm{u}\in\mathbb{C}^n,\bm{u}\perp iG^{-1}(\bm{\xi}_1),\norm{\bm{u}}=1\}$~\cite{Chen2015}, thus $\sigma_2<1$. 

In a nutshell, we have shown the local convergence of GPS and its robust version RGPS for phase retrieval problem~\eqref{eq:prob}. Though we have not given the attraction radius around the solution, it seems that GPS/RGPS shows global convergence starting from a random initialization when the ratio $m/n$ is large enough for Gaussian phase retrieval in all our numerical experiments. At the process of preparing this manuscript, a close work on the global convergence on alternating minimization for Gaussian phase retrieval is uploaded to arXiv~\cite{Zhang2018}. The requirement is $m/\log^3m\geq Mn^{3/2}\log^{1/2}n$ as $n,m\to\infty$. Its argument may be helpful for the proof of global convergence of GPS/RGPS.

\section{Numerical Experiments}

In this section, we use various tests to demonstrate the performance of GPS/RGPS applied to phase retrieval problem~\eqref{eq:prob} with and without prior information. To measure the reconstruction quality, we define the relative error between the reconstruction $\bm{x}$ and the optimal solution $\bm{x}^{\natural}$ as follows
\begin{equation}
  \label{eq:dis}
  rel. err(\bm{x},\bm{x}^{\natural})=\frac{\dist(\bm{x},\bm{x}^{\natural})}{\norm{\bm{x}^{\natural}}}=\min_{\lvert\alpha\rvert=1}\frac{\norm{\bm{x}-\alpha\bm{x}^{\natural}}}{\norm{\bm{x}^{\natural}}}.
\end{equation}
It is easy to see that the optimal $\alpha$ is given by $\frac{\bm{x}^*\bm{x}^{\natural}}{\lvert \bm{x}^*\bm{x}^{\natural}\rvert}$.

We first explore the numerical phase transition of GPS/RGPS applied to Gaussian phase retrieval without prior information. Then we add sparsity prior based on $\ell_0$ and $\ell_1$ norms respectively. We also consider more challenging and practical non-Gaussian transmission examples and (oversampling) Fourier phase retrieval. Current nonconvex solvers degrade significantly for non-Gaussian measurements without the help of randomness. With the flexibility and easiness of imposing prior information, we include nonnegativity condition in GPS/RGPS for transmission datasets. We show that the reconstruction is better than existing solvers, which have no straightforward way to add nonnegativity information. For Fourier phase retrieval, we consider three additional prior information: nonnegativity, rectangular support and total variation (TV) regularization. Experiments demonstrate superior performance of GPS/RGPS.

\subsection{Synthetic Gaussian Phase Retrieval}

Recently Gaussian phase retrieval is the most popular model problem of phase retrieval in the literature. We first compare GPS and RGPS to existing solvers for Gaussian phase retrieval without prior information. For Gaussian phase retrieval, problem~\eqref{eq:prob} is called real and complex cases if $\bm{a}_i$'s and the unknown $\bm{x}$ belong to $\mathbb{R}^n$ and $\mathbb{C}^n$ respectively. For the two cases, the length of the signal $n$ is set to $400$. We generate real/complex Gaussian signals at random, and draw the sampling vectors $\bm{a}_i\sim\mathcal{N}(0,\bm{I})$ and
$\bm{a}_i\sim\mathcal{CN}(0,\bm{I})=\mathcal{N}(0,\bm{I}/2)+i\mathcal{N}(0,\bm{I}/2)$ for real and complex cases respectively. We compare GPS/RGPS and Douglas-Rachford (DR) and robust Douglas-Rachford (RDR) with other nonconvex optimization algorithms, such as RAF~\cite{Wang2017Solving}, TAF~\cite{Wang16}, TWF~\cite{Chen2015b} and WirtFlow~\cite{Candes2015}, all of which are gradient flow based and need to tune an optimal step size. We use the default configuration of the parameters described in the corresponding references. Note that here we only consider the nonconvex approaches for comparison, as the comparison of nonconvex solvers to SDR PhaseLift approach can be found in~\cite{Li2018}. All tests are conducted with Monte-Carlo simulations. We do not explore the optimal $t$ for RDR and RGPS and set $t=0.1$ throughout all of our experiments. We set the maximum number of iterations to $5000$ and the relative error tolerance to $1e-3$ for each method. When either one of the two criteria satisfies, the iteration is stopped.

\paragraph{Phase Transition}

For Gaussian phase retrieval, the benchmark test is to compare the phase transition of each algorithm. The \emph{phase transition} is the critical point of the ratio of the number of measurements $m$ to the length of the signal $n$. When $m/n$ is above the phase transition, the unique solution can be located by the algorithm for problem~\eqref{eq:prob}. When $m/n$ is below the phase transition, there exists at least an instance  for which the algorithm failed to find the true solution. To numerically test the phase transition, for a fixed signal with length $n$, we test an array of different measurement sizes $m=\{1n,1.1n,\ldots,5n\}$. For each pair of $n$ and $m$, we solve $30$ randomly generated problems and calculate the successful recovery rate. A recovery is successful if the relative error is below $1e-3$. 

Since all of the compared algorithms are nonconvex, their performance depends on the initialization crucially. First we start these methods from the point returned by the reweighted maximal correlation method~\cite{Wang2017Solving}. The initial guess is close to the optimal solution if the number of measurement $m=\mathcal{O}(n)$. We compute the recovery rate and plot the curves of recovery rate in Figure~\ref{fig:1a} and~\ref{fig:1b} for real  and complex cases respectively. We also explore the effect of the initialization on these nonconvex methods, i.e., we start all algorithms from a random point. Phase transition in this case is plotted in Figure~\ref{fig:1c} and~\ref{fig:1d}. It shows GPS/RGPS and DR/RDR are much less sensitive to initialization than other gradient flow based solvers. Compared to DR and RDR, GPS and RGPS show sharper phase transition, in particular for small number of measurement. This demonstrates that GPS/RGPS are more likely to escape the stagnation point than DR/RDR. GPS shows the sharpest phase transition among all methods, successfully solving problem~\eqref{eq:prob} with $m=1.7n$ and $m=2.7n$ measurements for real and complex cases respectively. This phase transition is also better than the incremental nonconvex approach for PhaseLift~\cite{Li2018}. 
  
\begin{figure}
     \begin{subfigure}[b]{.45\textwidth}
     \centering
     \includegraphics[width=\textwidth]{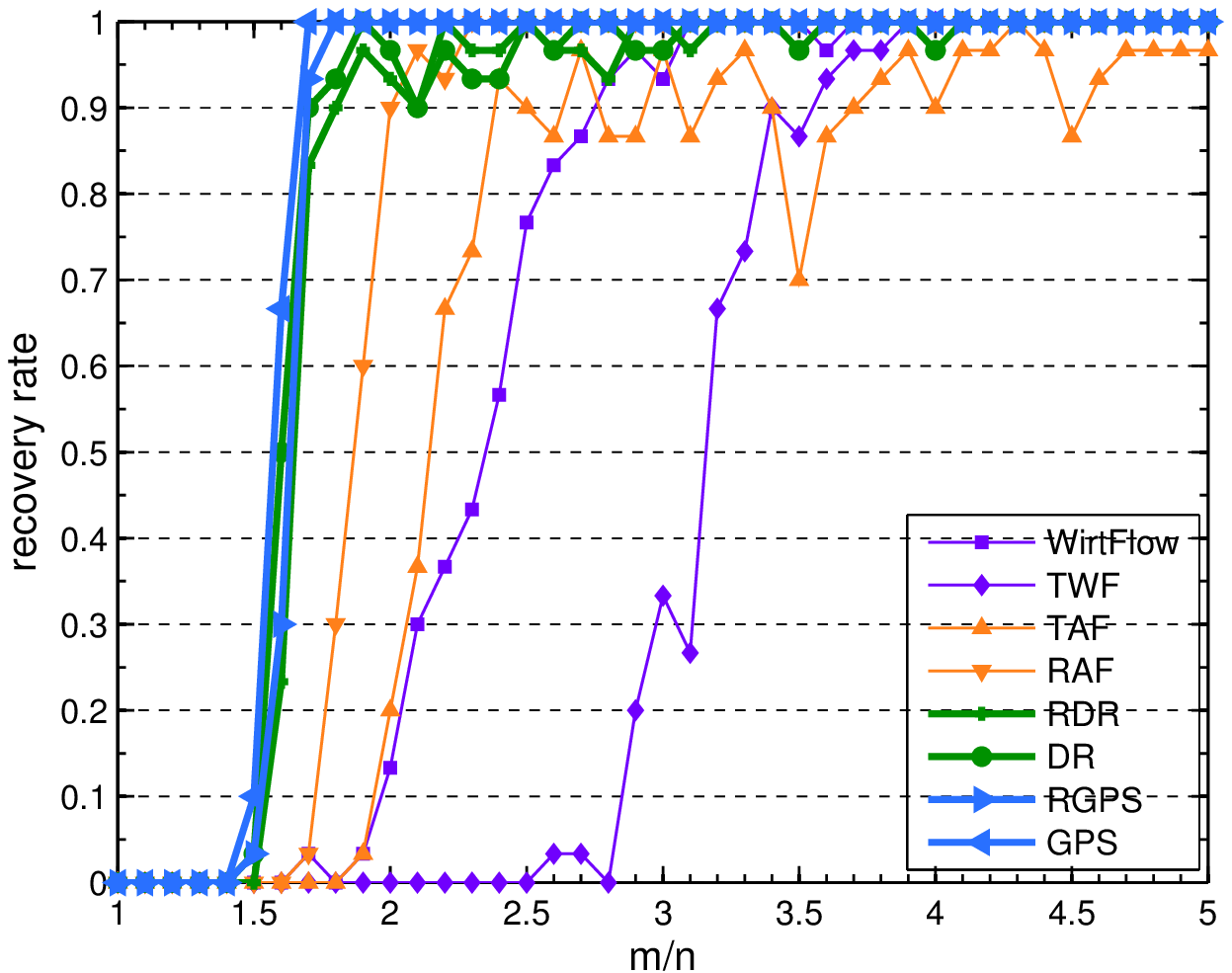}
     \caption{Standard real Gaussian model\label{fig:1a}}
   \end{subfigure}
 \begin{subfigure}[b]{.45\textwidth}
     \centering
     \includegraphics[width=\textwidth]{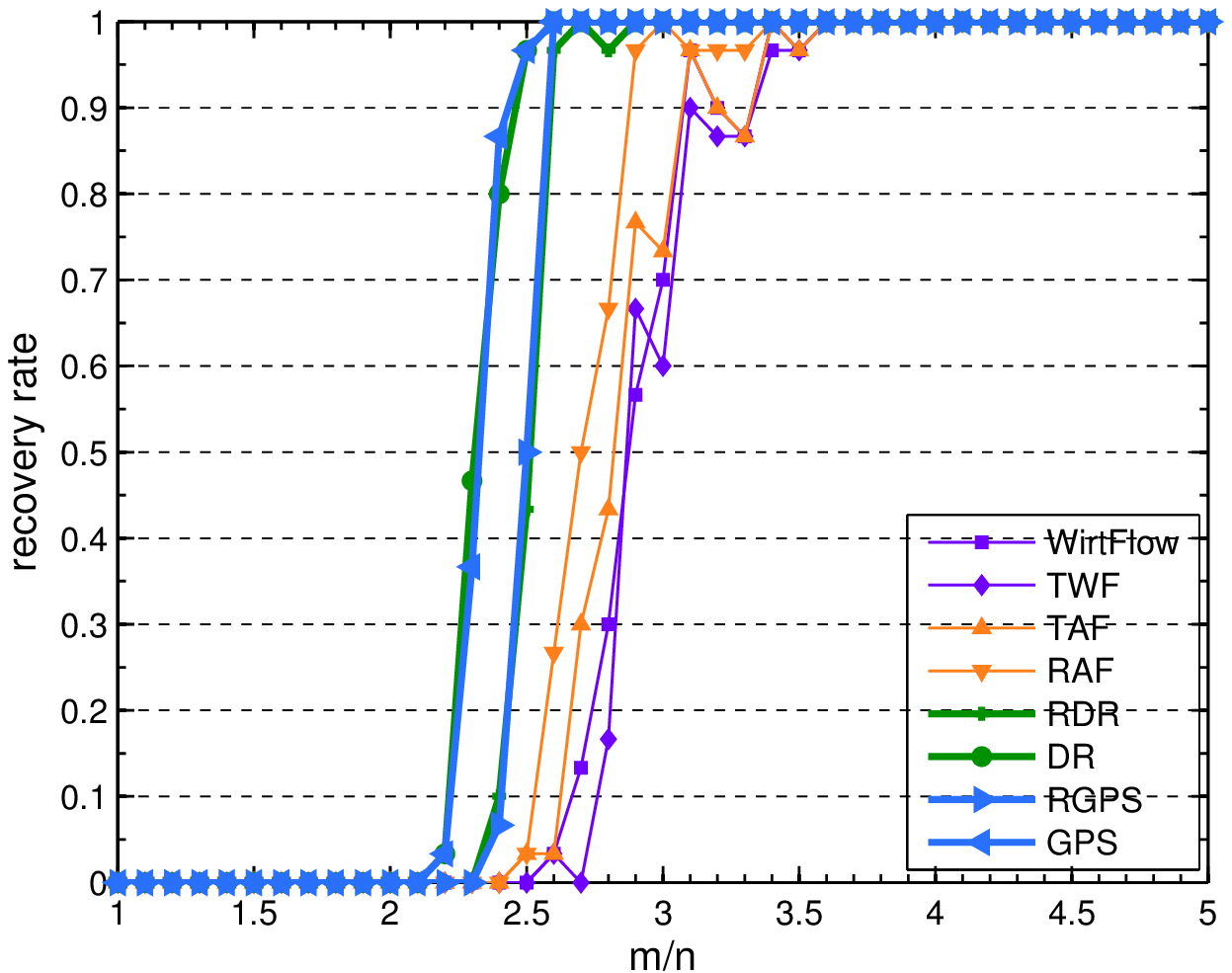}
     \caption{Standard complex Gaussian model\label{fig:1b}}
   \end{subfigure}\\
       \begin{subfigure}[b]{.45\textwidth}
     \centering
     \includegraphics[width=\textwidth]{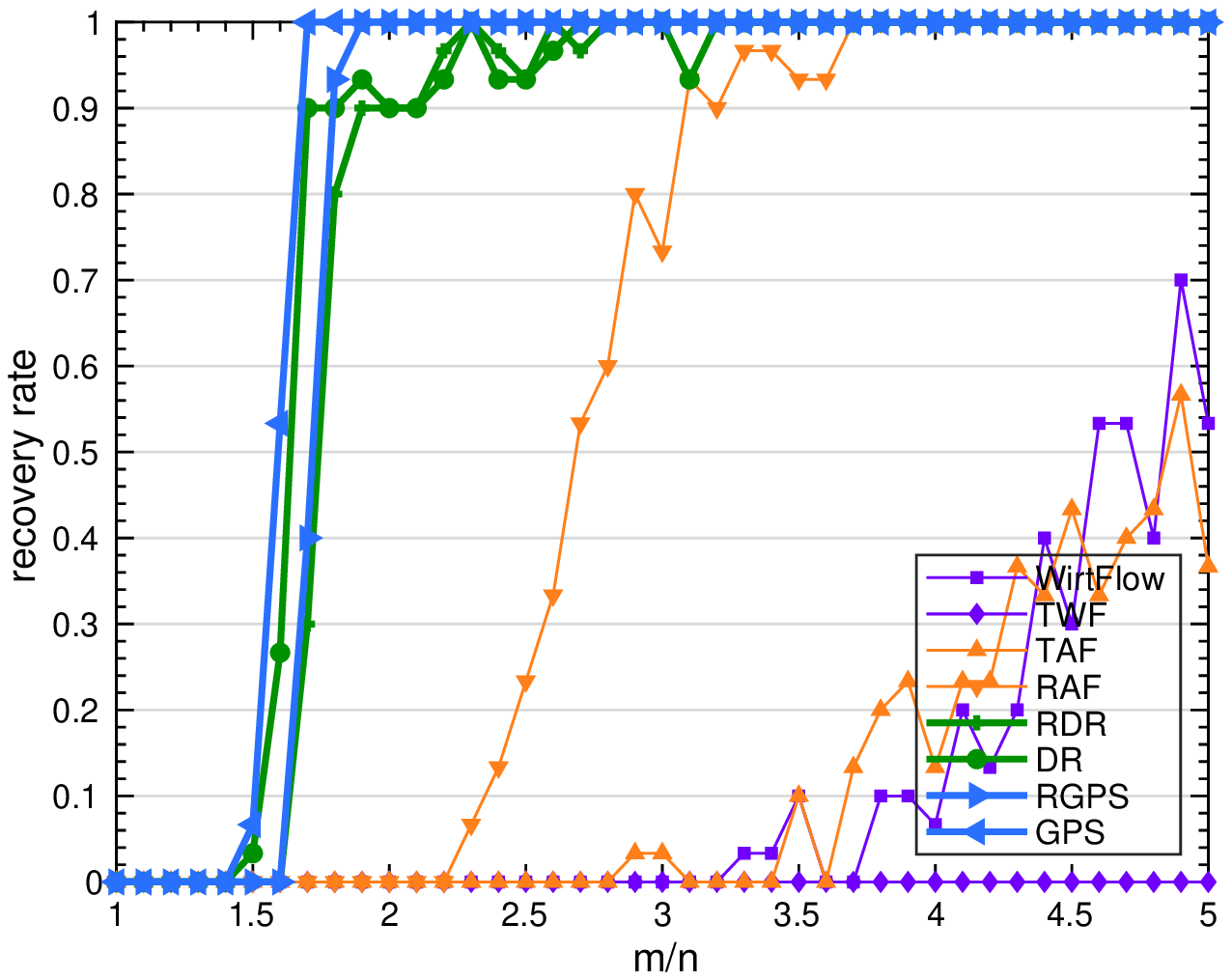}
     \caption{Standard real Gaussian model\label{fig:1c}}
   \end{subfigure}
 \begin{subfigure}[b]{.45\textwidth}
     \centering
     \includegraphics[width=\textwidth]{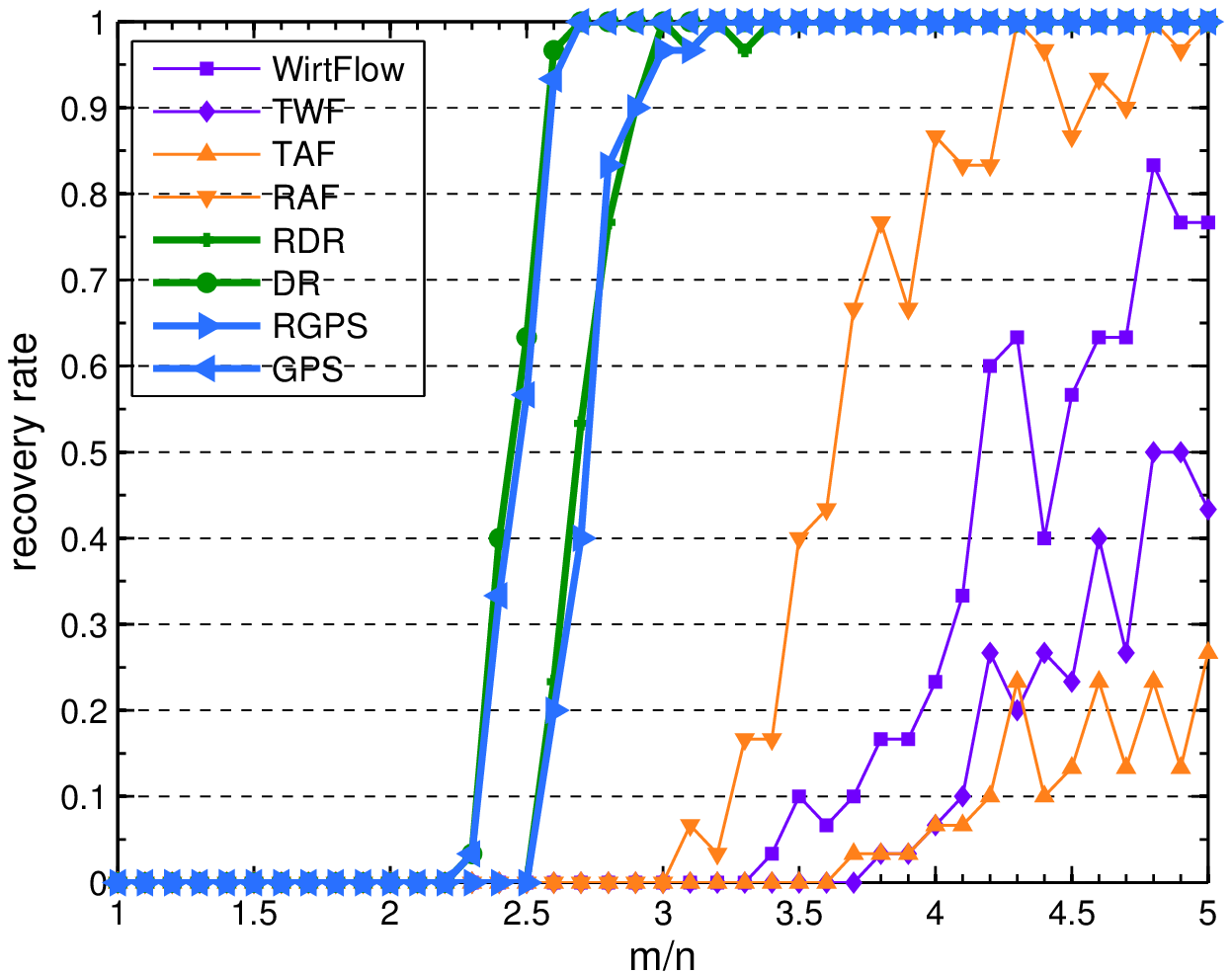}
     \caption{Standard complex Gaussian model\label{fig:1d}}
   \end{subfigure}
  \caption{Recovery rate for real and complex Gaussian phase retrieval. The top row shows results from the specific initialization~\cite{Wang2017Solving}. The bottom row shows results from random initialization.\label{fig:1}}
\end{figure}

For phase transition, compared with GPS, RGPS shows inferior performance and requires slightly more measurements for successful recovery, as shown in Figure~\ref{fig:1}. We list the average number of iteration of GPS/RGPS and DR/RDR in Table~\ref{tab:1}. For the four methods, their residual curves fall into two stages, locating the contraction basin and convergence to the solution. There are two different features of the residual curves of GPS/DR and RGPS/RDR. GPS/DR locate the contraction basin faster than their counterparts while RGPS/RDR converge faster than GPS/DR once in the contraction basin. Due to these different behaviors, the iteration number for robust versions behave differently as the number of measurement changes. When $m\geq 1.9n$ and $m\geq 3.1n$ for real and complex cases respectively, the average number of iteration of RGPS and RDR is smaller than that of their counterparts. Within $5000$ iterations, all of them find the solution, while GPS and DR oscillate around the solution, which leads to more iterations in total. However, when $m$ is smaller than the above critical point, the robust versions take more iterations and performs inferior to their counterparts in computation cost and phase transition. Within $5000$ iterations, GPS and DR locate the basin of solution and converge, while RGPS and RDR do not locate the basin of solution. Based on these behaviors, we can combine GPS and RGPS to benefit from both of their advantages. 

\begin{table}
\centering
\caption{Comparison of average number of iteration from $30$ repeat experiments.\label{tab:1}}
\ra{1.3}
\begin{tabular}{c@{}c@{}rrrrc@{}c@{}rrrr}
\toprule
&\phantom{ab}& 
\multicolumn{4}{c}{real} & &
\phantom{ab} & 
\multicolumn{4}{c}{comp}\\
\cmidrule{3-6} \cmidrule{9-12}
$m/n$&& GPS & RGPS & DR & RDR &$m/n$ && GPS & RGPS & DR & RDR\\
\midrule
1.5  && 4852 & 5000 & 4976 & 5000 &2.5&& 4704 & 5000 & 4837 & 5000 \\
1.6  && 3784 & 5000 & 4062 & 5000 &2.6&& 3986 & 5000 & 3971 & 5000 \\
1.7 && 1363 & 3842 & 972 & 4211 &2.7&& 2687 & 4341 & 2514 & 4699 \\
1.8 && 825 & 1178 & 504 & 1421 &2.8&& 1999 & 3899 & 2027 & 3157 \\
1.9 && 1131 & 444 & 1398 & 407 &2.9&& 1566 & 1847 & 1653 & 1853 \\
2.0 && 956 & 111 & 573 & 144 &3.0&& 1469 & 1584 & 1382 & 1543 \\
2.1 && 889 & 96 & 698 & 109 &3.1&& 1299 & 904 & 1207 & 911 \\
2.2 && 659 & 73 & 845 & 73 &3.2&& 1205 & 599 & 1200 & 691 \\
2.3 && 1000 & 68 & 937 & 63 &3.3&& 1248 & 411 & 1080 & 605 \\
2.4 && 854 & 63 & 933 & 56 &3.4&& 1011 & 370 & 1051 & 498 \\
2.5 && 742 & 61 & 699 & 56 &3.5&& 1115 & 339 & 1219 & 396 \\
2.6 && 1182 & 61 & 1021 & 52 &3.6&& 1460 & 316 & 1224 & 298 \\
2.7 && 894 & 58 & 1081 & 49 &3.7&& 1507 & 274 & 1228 & 278 \\
2.8 && 1000 & 56 & 1100 & 50 &3.8&& 1234 & 258 & 1286 & 257 \\
2.9 && 458 & 56 & 839 & 49 &3.9&& 1067 & 248 & 1198 & 249 \\
3.0 && 706 & 54 & 439 & 45 &4.0&& 1258 & 240 & 1117 & 238 \\
\bottomrule
\end{tabular}
\end{table}

\paragraph{Noisy Case}

When measurements are corrupted by noise, RGPS and RDR should outperform GPS and DR. We consider Gaussian noise $\bm{\epsilon}^{\text{noise}}$. To specify the effect of noise level, we generate noisy real and complex measurements with different SNR levels at $\{10,15,\ldots,50\}$. We set the number of measurement $m$ to $2n$ and $3n$ for real and complex cases respectively and the maximum number of iterations to $200$. All test algorithms start from random initializations, since all initialization algorithms degrades significantly in the noisy cases for small ratio $m/n$. The behavior of relative error vs. iteration number of a typical run for noisy measurement is plotted in Figure~\ref{fig:2a} and~\ref{fig:2b}. The oscillatory behavior of GPS/DR is clearly shown while their robust versions are much more monotone. The relative recovery error (in dB) vs. noise level (SNR) is plotted in Figure~\ref{fig:2c} and~\ref{fig:2d}. Other gradient flow based non-convex methods fail to locate a solution of relative error blow unity starting from a random initialization. Further experiments suggest  that when the output from RGPS or RDR, are input to these gradient flow based non-convex methods, a better solution can be located. Thus, one can utilize a hybrid approach: first locate an approximate recovery using RGPS, then feed it to another state-of-the-art method to further refine the recovery. 

\begin{figure}
     \begin{subfigure}[b]{.45\textwidth}
     \centering
     \includegraphics[width=\textwidth]{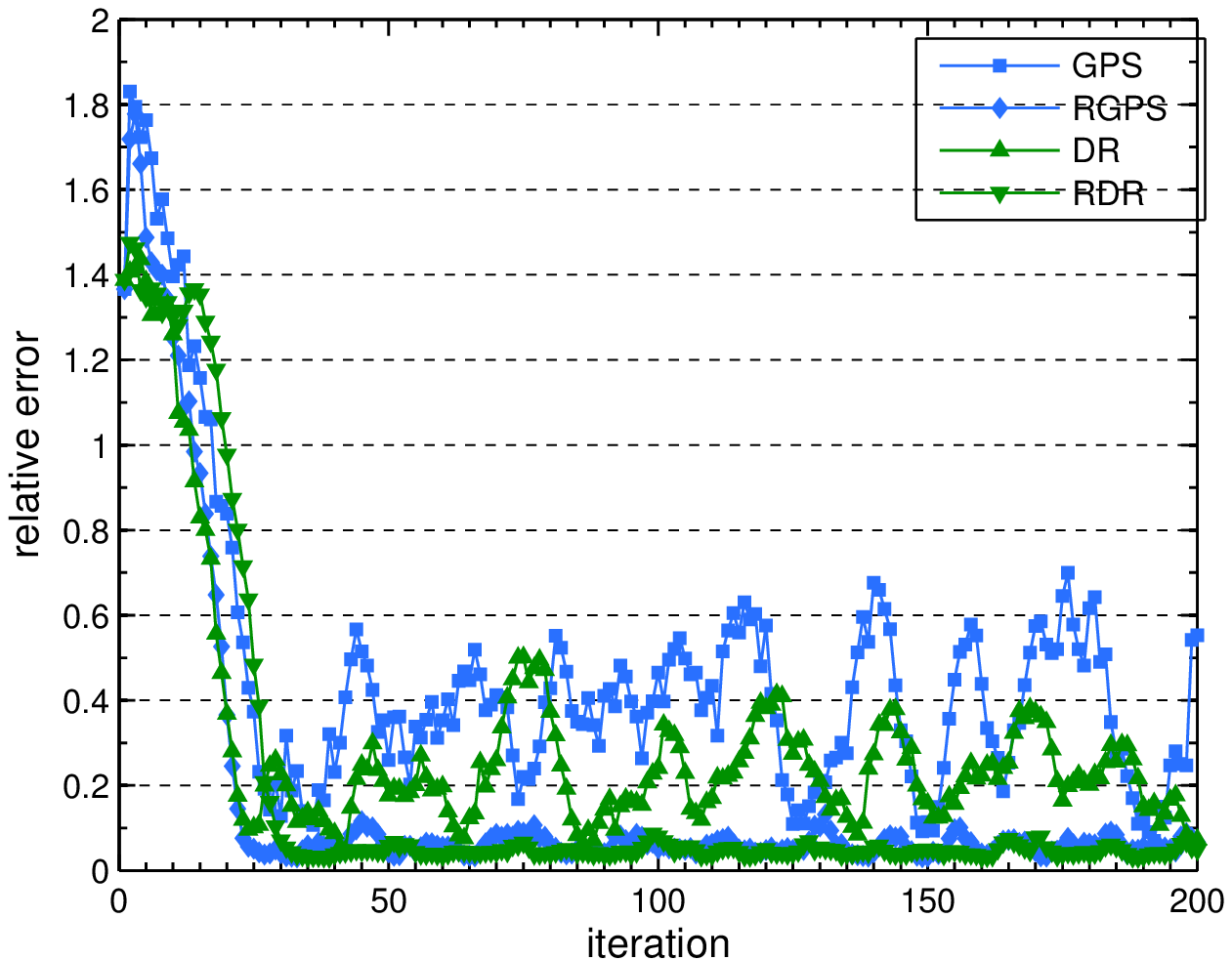}
     \caption{Standard real Gaussian model\label{fig:2a}}
   \end{subfigure}
 \begin{subfigure}[b]{.45\textwidth}
     \centering
     \includegraphics[width=\textwidth]{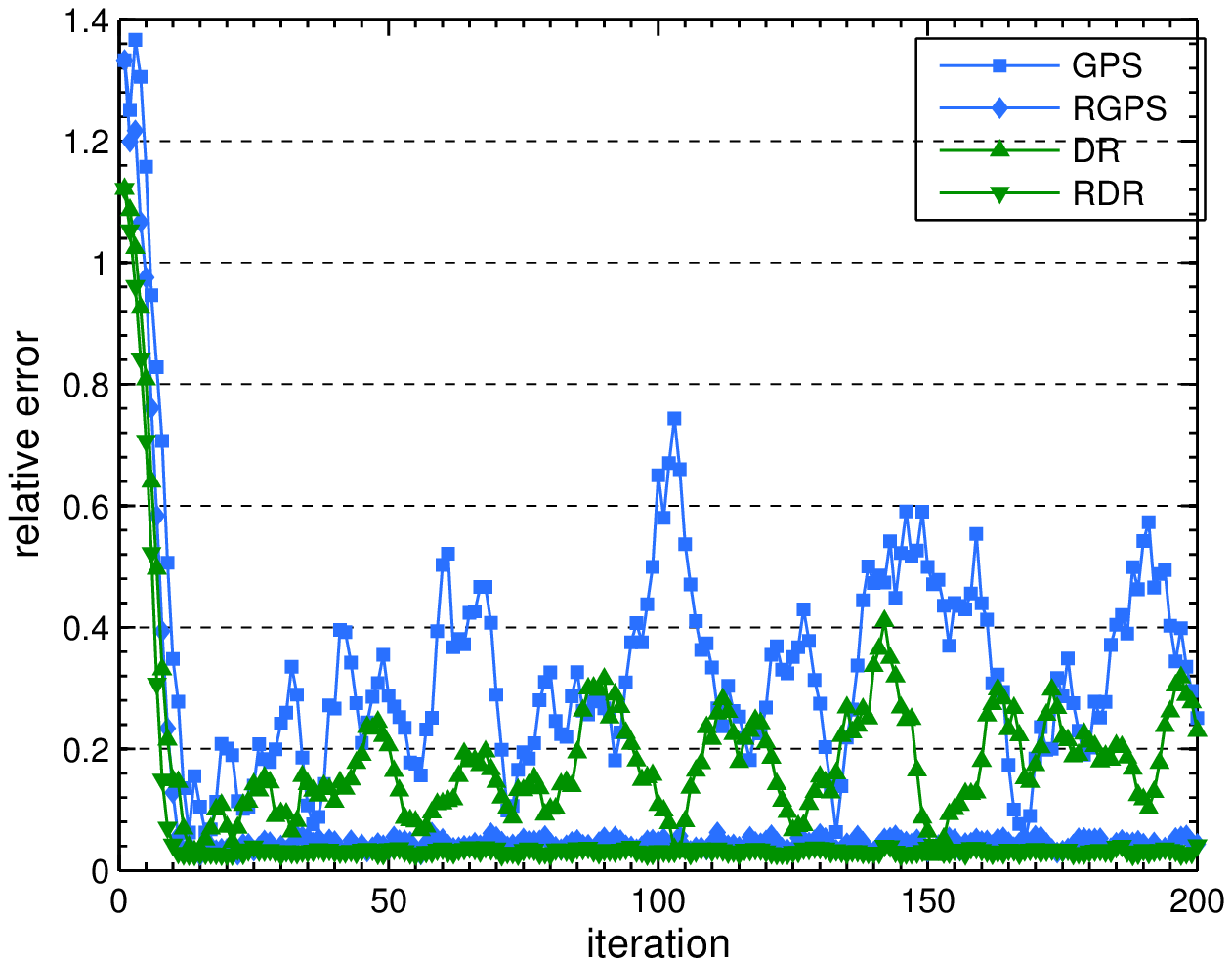}
     \caption{Standard complex Gaussian model\label{fig:2b}}
   \end{subfigure}\\
       \begin{subfigure}[b]{.45\textwidth}
     \centering
     \includegraphics[width=\textwidth]{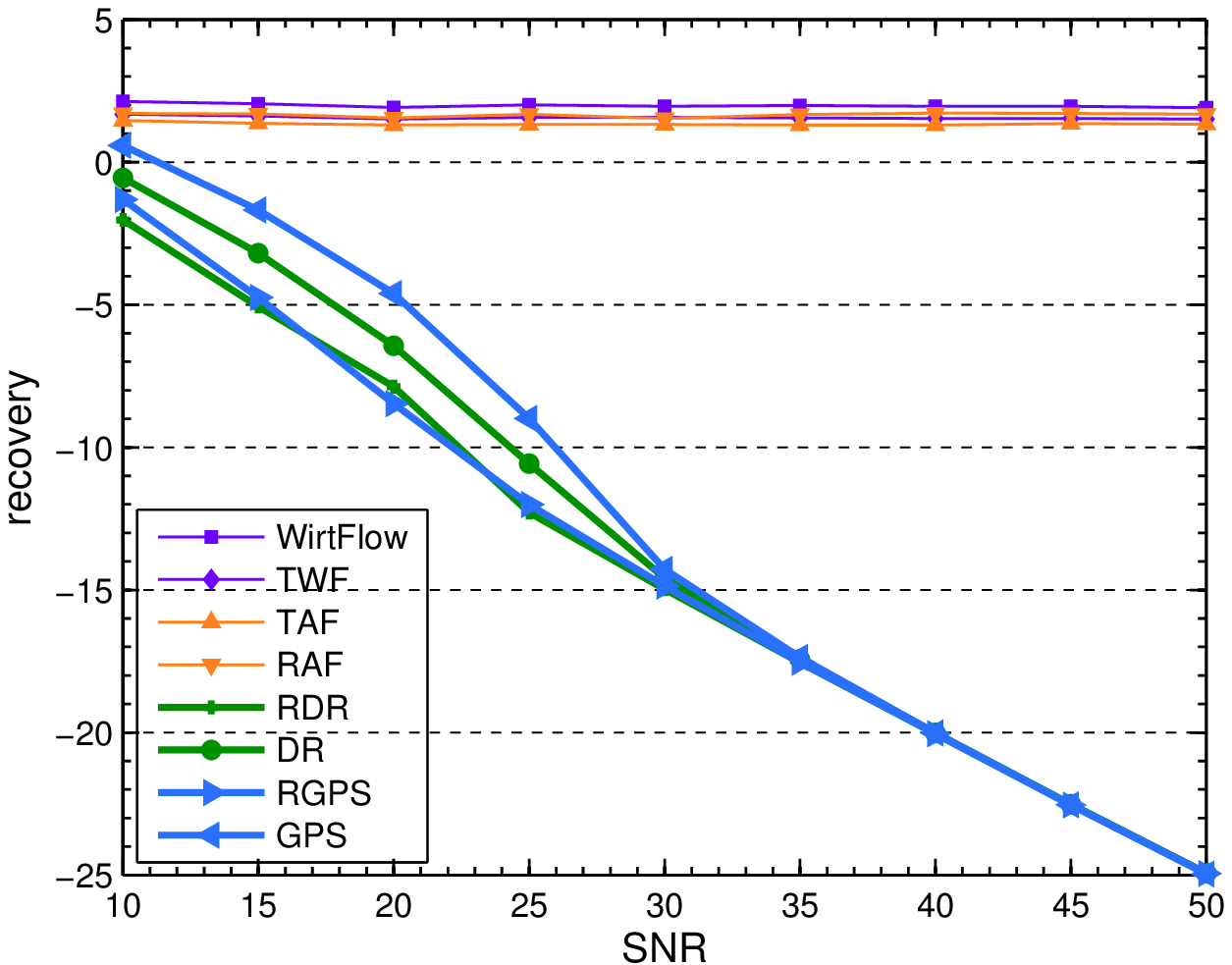}
     \caption{Standard real Gaussian model\label{fig:2c}}
   \end{subfigure}
 \begin{subfigure}[b]{.45\textwidth}
     \centering
     \includegraphics[width=\textwidth]{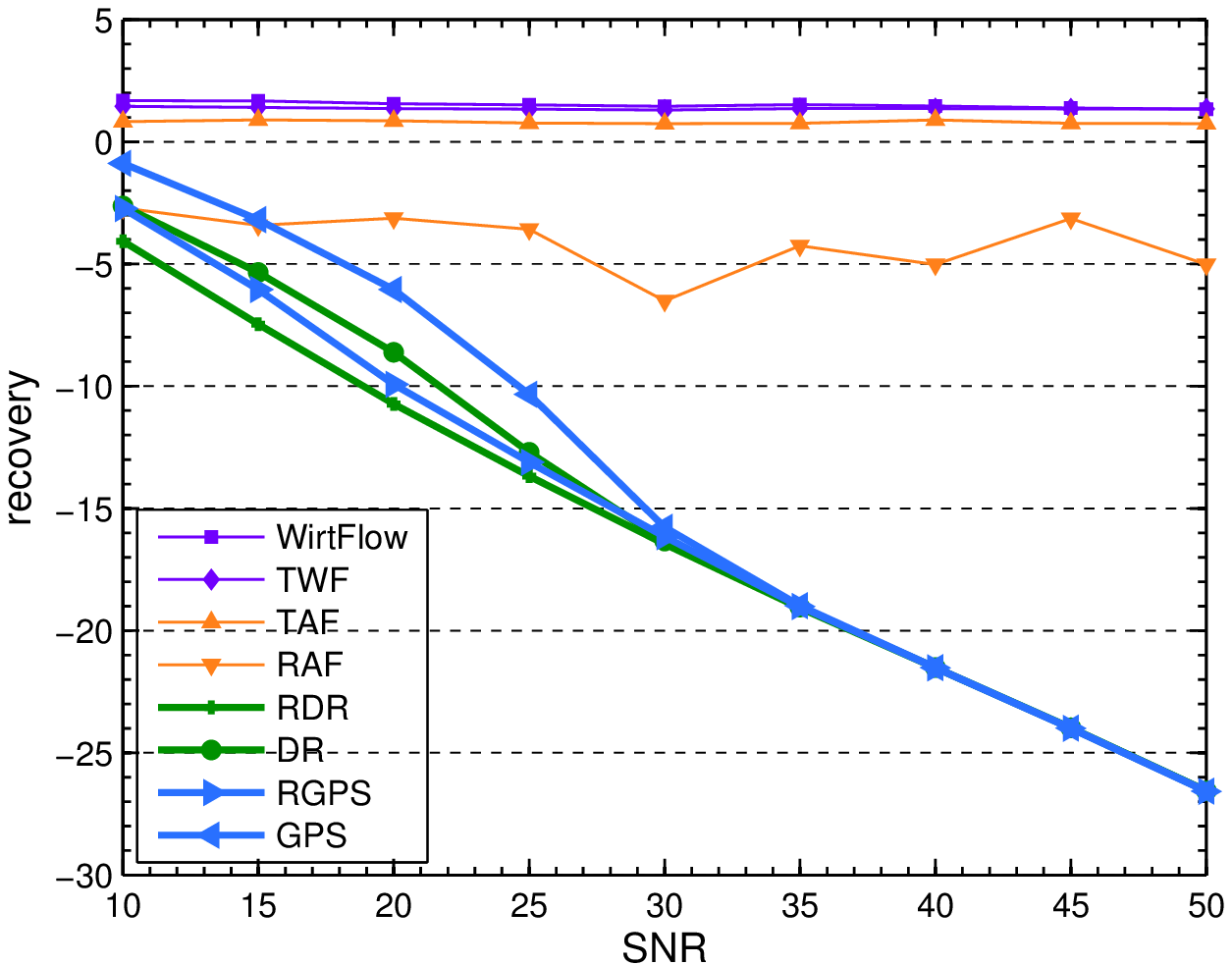}
     \caption{Standard complex Gaussian model\label{fig:2d}}
   \end{subfigure}
  \caption{Recovery for noisy measurements. Top row shows relative error vs. iteration number. Bottom row shows the relative recovery error vs. SNR.\label{fig:2}}
\end{figure}

\subsection{Sparse Phase Retrieval}

In this section, we consider the phase retrieval problem with additional prior information of sparsity and nonnegativity of the solution. Function $g$ can take $\ell_1$ and $\ell_0$ respectively to promote sparsity in convex and nonconvex forms. The length of the signal is still $400$. We set the known support to $J=[0,floor(n/2)]$ and the sparsity $s$  to $10,20,30$. The number of measurement $m$ is set to $n$.

With the sparsity-promoting term $g$, $\bm{x}$ can be updated by the following simple operations for the $\ell_1$ and $\ell_0$ functions respectively:
\begin{equation*}
  \bm{x}^{k+1}=
  \begin{cases}
    \max(\bm{x}^{k+\frac{1}{2}}+\bm{\lambda}^k-p,0) & g=\ell_1\\
    \text{hardthreshold}(\bm{x}^{k+\frac{1}{2}}+\bm{\lambda}^k) & g=\ell_0
  \end{cases}
\end{equation*}
where $p$ is the parameter for soft thresholding and the hard thresholding operator is just keeping the leading $s$ largest entries. For $\ell_1$ case, the recovery depends on the parameter $p$, which we set to $10,20,30$ for the sparsity level $s=10,20,30$ respectively. We report the successful recovery rate and the number of iterations in Table~\ref{tab:2}. The performance based on $\ell_1$ is much better than that based on $\ell_0$. 

\begin{table}
\centering
\caption{Recovery from $10$ repeated experiments with sparsity prior information.\label{tab:2}}
\ra{1.3}
\begin{tabular}{c@{}c@{}rrr@{}c@{}rrr}
\toprule
&\phantom{ab}& 
\multicolumn{3}{c}{recov. rate} & 
\phantom{ab} & 
\multicolumn{3}{c}{iteration}\\
\cmidrule{3-5} \cmidrule{7-9}
$s$&& $10$ & $20$ & $30$ && $10$ & $20$ & $30$\\
\midrule
\midrule
$\ell_1$ && 1 & 1 & 0.8 && 787 & 897 & 2223 \\
$\ell_0$ && 0.7 & 0.6 & 0.5 && 1994 & 2594 & 3169\\
\bottomrule
\end{tabular}
\end{table}

\subsection{Real Transmission Dataset}

Here we test RGPS for one type of practical measurement, where the measurement matrix $\bm{A}^*$ is a real transmission matrix provided by Phasepack~\cite{Chandra2017PhasePack}. It is used to benchmark the performance of various phase retrieval algorithms. For the transmission matrix dataset, the rows of the measurement matrix are calculated using a measurement process (also a phase retrieval problem), and some are more accurate than the others. Each measurement matrix comes with a per-row residual to measure the accuracy of that row. For our test, we use the measurement matrix \verb+A_prVAMP.mat+ of an image with a resolution ($16\times 16$). The measurement matrix can be cut off to a smaller size by only loading the more accurate rows. We set the cut-off residual bound to $0.04$, which allows the ratio $m/n$ to reach $6$. Both the maximum number of iteration of RGPS and RAF are set to 3000. 

The signal is of length $n=256$. We randomly collect the sampling matrix from $\bm{A}^*$ with the number of measurement $m=2n,3n,4n,5n,6n$. For each pair $(m,n)$, we repeatedly solve each problem $10$ times from different initializations. When applying RGPS, we impose both real-valuedness and nonnegativity prior, while we only impose real-valuedness for RAF, as nonnegativity can not be easily imposed in RAF. The average reconstruction error~\footnote{We calculate the error by the algorithm provided in the paper~\cite{Chandra2017PhasePack}.} is listed in Table~\ref{tab:3}. Although different initialization leads to different reconstruction, RGPS outperforms RAF in all cases, in particular when the ratio $m/n$ is below $6$. The best reconstructions of RGPS and RAF for each pair $(m,n)$ are plotted in Figure~\ref{fig:5}. The subcaption below each image denotes the relative error of reconstruction. 

\begin{table}
\centering
\caption{Relative reconstruction error averaged from $10$ repeated experiments from different random initializations.\label{tab:3}}
\ra{1.3}
\begin{tabular}{c@{}c@{}ccccc}
\toprule
$sampl. $&& $m=2n$ & $m=3n$ & $m=4n$ & $m=5n$ & $m=6n$ \\
\midrule
 RGPS && 0.6245 & 0.6670 & 0.6927 & 0.6666 & 0.6227 \\
 RAF && 0.8253 & 0.8224 & 0.8156 & 0.8195 & 0.6357\\
\bottomrule
\end{tabular}
\end{table}

\begin{figure}
  \centering
       \begin{subfigure}[b]{.18\textwidth}
     \centering
     \includegraphics[width=\textwidth]{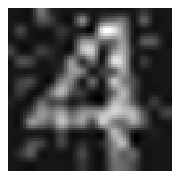}
     \caption{0.5748}
   \end{subfigure}
     \begin{subfigure}[b]{.18\textwidth}
     \centering
     \includegraphics[width=\textwidth]{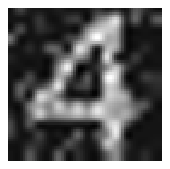}
     \caption{0.5738}
   \end{subfigure}
 \begin{subfigure}[b]{.18\textwidth}
     \centering
     \includegraphics[width=\textwidth]{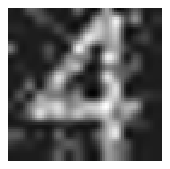}
     \caption{0.5681}
   \end{subfigure}
\begin{subfigure}[b]{.18\textwidth}
     \centering
     \includegraphics[width=\textwidth]{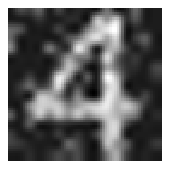}
     \caption{0.5586}
   \end{subfigure}
\begin{subfigure}[b]{.18\textwidth}
     \centering
     \includegraphics[width=\textwidth]{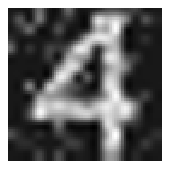}
     \caption{0.5214}
   \end{subfigure}\\
   \begin{subfigure}[b]{.18\textwidth}
     \centering
     \includegraphics[width=\textwidth]{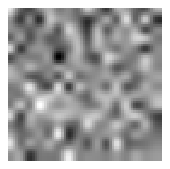}
     \caption{0.8193}
   \end{subfigure}
\begin{subfigure}[b]{.18\textwidth}
     \centering
     \includegraphics[width=\textwidth]{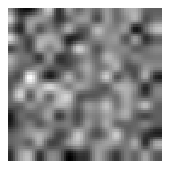}
     \caption{0.8147}
   \end{subfigure}
 \begin{subfigure}[b]{.18\textwidth}
     \centering
     \includegraphics[width=\textwidth]{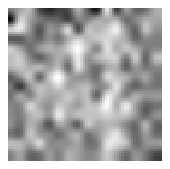}
     \caption{0.8055}
   \end{subfigure}
\begin{subfigure}[b]{.18\textwidth}
     \centering
     \includegraphics[width=\textwidth]{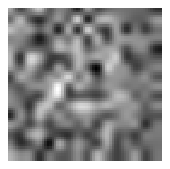}
     \caption{0.8041}
   \end{subfigure}
\begin{subfigure}[b]{.18\textwidth}
     \centering
     \includegraphics[width=\textwidth]{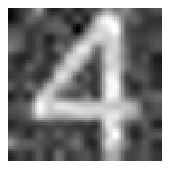}
     \caption{0.5228}
   \end{subfigure}
\caption{Reconstruction from RGPS (1st row) and RAF (2nd row)  using transmission datasets of different sizes  (from left to right, $m=2n,3n,4n,5n,6n$) with relative error in subcaptions. \label{fig:5}}
\end{figure}

\subsection{Incorporating TV Regularization}

Our framework can also include the TV regularization to improve the reconstruction by popular HIO for oversampling Fourier phase retrieval. 
We first solve the phase retrieval problem without TV regularization and obtain a solution as the initialization, and then refine it by adding a TV minimization term to improve the reconstruction quality. In this case, we apply RGPS to solve the following problem
\begin{equation}
\label{eq:tv}
\begin{aligned}
  \min_{\bm{x},\bm{y}_1,\bm{y}_2} \quad &f_1(\bm{y}_1) + f_2(\bm{y}_2) + g(\bm{x})\\
 \text{s.t.}\quad& \bm{A}^*\bm{x} = \bm{y}_1\\
 & \bm{D}\bm{x} = \bm{y}_2
\end{aligned}
\end{equation}
where matrix $\bm{D}$ corresponds to the total variation linear operator which comprises of horizontal and vertical differences of $\bm{x}$. The oversampling measurement means a known rectangular support $S$ of the solution which provides another constraint. Two sets of synthetic data are generated by the cameraman image and caffeine molecule~\cite{Waldspurger2015}, which are both of size $128\times 128$. We first pad the image of size $128\times 128$ to size of $256\times 256$ with zeros and the region of size $128\times 128$ is used as the known support set $S$. The prior information is $\bm{x}\in\mathbb{R}_+^{n_1\times n_2}\cap S$ and TV of the image should be small. We express ~\eqref{eq:tv} in standard form, where $f(\bm{y})=f_1(\bm{y}_1)+f_2(\bm{y}_2)=\mathbb{I}_{\lvert\bm{y}_1\rvert=\bm{b}}(\bm{y}_1)+\norm{\bm{y}_2}_1,g(\bm{x})=\mathbb{I}_{\mathbb{R}_+^{n_1\times n_2}\cap S}(\bm{x})$ and the graph set is $C=\{(\bm{x},\bm{y}_1,\bm{y}_2)|\bm{A}^*\bm{x} = \bm{y}_1,\bm{D}\bm{x} = \bm{y}_2\}$. 

For Fourier phase retrieval, $\bm{A}^*$ is isometric. By exploiting the structure of TV operator, the graph projection step $\Pi_{C}(\bm{c},\bm{d}_1,\bm{d}_2)$ of GPS is computed by solving
\begin{equation*}
  (2\bm{I}+\bm{D}^*\bm{D})\bm{x} = \bm{c} + \bm{A}\bm{d}_1+\bm{D}^*\bm{d}_2.
\end{equation*}
The conjugate gradient method is used to solve the above linear system for its matrix-free property. The robust GPS can be implemented accordingly for noisy measurements. Without TV regularization, RGPS produces about the same reconstruction quality as HIO. However, HIO is more efficient. Again we use a hybrid approach: first run HIO ($\beta=1$) by setting the number of iteration to $1000$ and then feed the output to RGPS and run $30$ iterations to solve~\eqref{eq:tv}. To investigate the effect of noise, we consider four noise levels with SNR being, $\infty$ (noiseless case), $30$, $40$ and $50$. For each case, we run HIO+RGPS $10$ times from different initializations. The average relative error is listed in Table~\ref{tab:6}. And the best outputs of HIO and HIO+RGPS for cameraman and molecule at different noise levels are depicted in Figure~\ref{fig:6}, where the subcaption gives the relative error. It is obvious that after the refinement of TV-minimization by RGPS, the quality of reconstruction is better. Note that the relative error is calculated after possible shift and mirror-reflection. There may be some misalignment of the molecule, so the relative error may increase as SNR increases. Note that the comparison between HIO and HIO+RGPS makes sense for all cases, as they share the same alignment.

\begin{table}
\centering
\caption{Average relative error of $10$ experiments starting from random initializations for Fourier phase retrieval with/without TV minimization at different noise level.\label{tab:6}}
\ra{1.3}
\begin{tabular}{c@{}c@{}rrrr@{}c@{}rrrr}
\toprule
&\phantom{ab}& 
\multicolumn{4}{c}{cameraman} & 
\phantom{ab} & 
\multicolumn{4}{c}{molecule}\\
\cmidrule{3-6} \cmidrule{8-11}
SNR && $\infty$ & 30 & 40 & 50 && $\infty$ & 30 & 40 & 50\\
\midrule
HIO && 0.0728 & 0.1721 & 0.1012 & 0.0798 && 0.3101 & 0.3096 & 0.2802 & 0.3354 \\
HIO+RGPS && 0.0289 & 0.0752 & 0.0408 & 0.0343 && 0.2149 & 0.1998 & 0.1729 & 0.2353 \\
\bottomrule
\end{tabular}
\end{table}

\begin{figure}
\centering
\begin{subfigure}[b]{.2\textwidth}
     \centering
     \includegraphics[width=\textwidth]{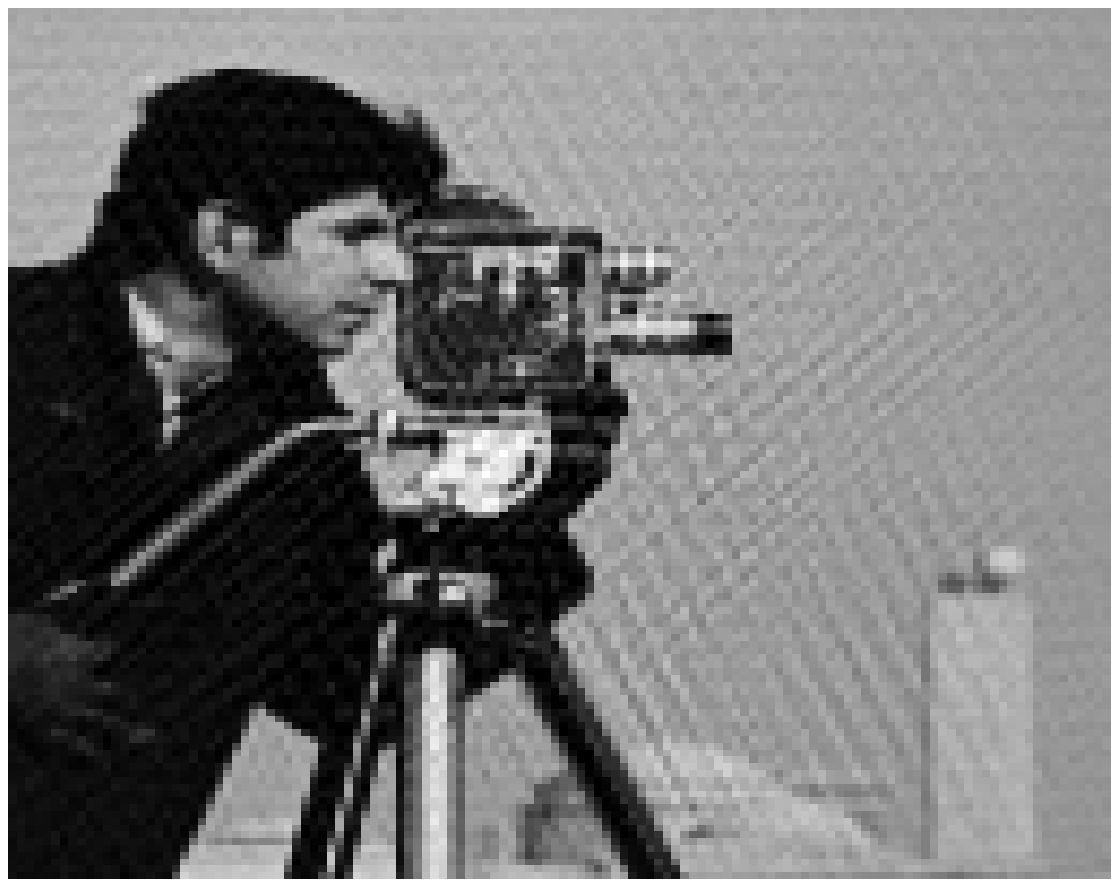}
     \caption{0.0611}
   \end{subfigure}
 \begin{subfigure}[b]{.2\textwidth}
     \centering
     \includegraphics[width=\textwidth]{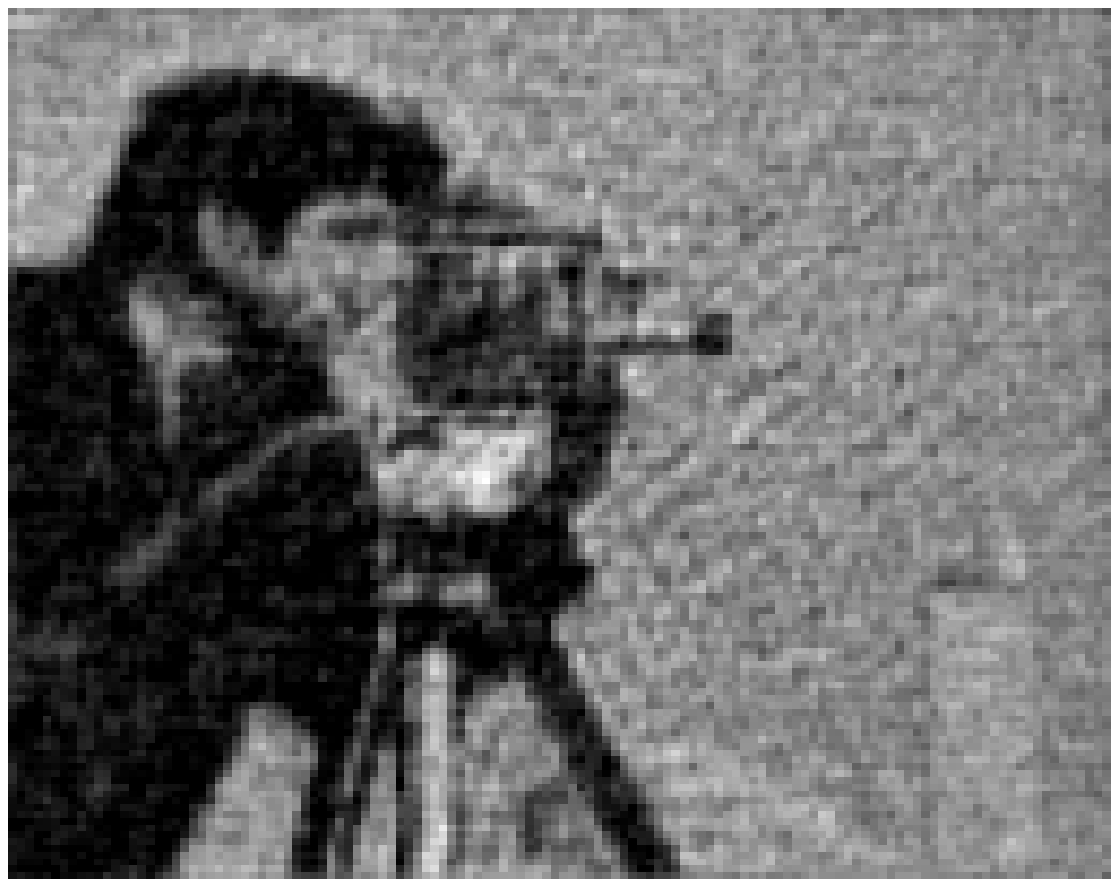}
     \caption{0.1633}
   \end{subfigure}
\begin{subfigure}[b]{.2\textwidth}
     \centering
     \includegraphics[width=\textwidth]{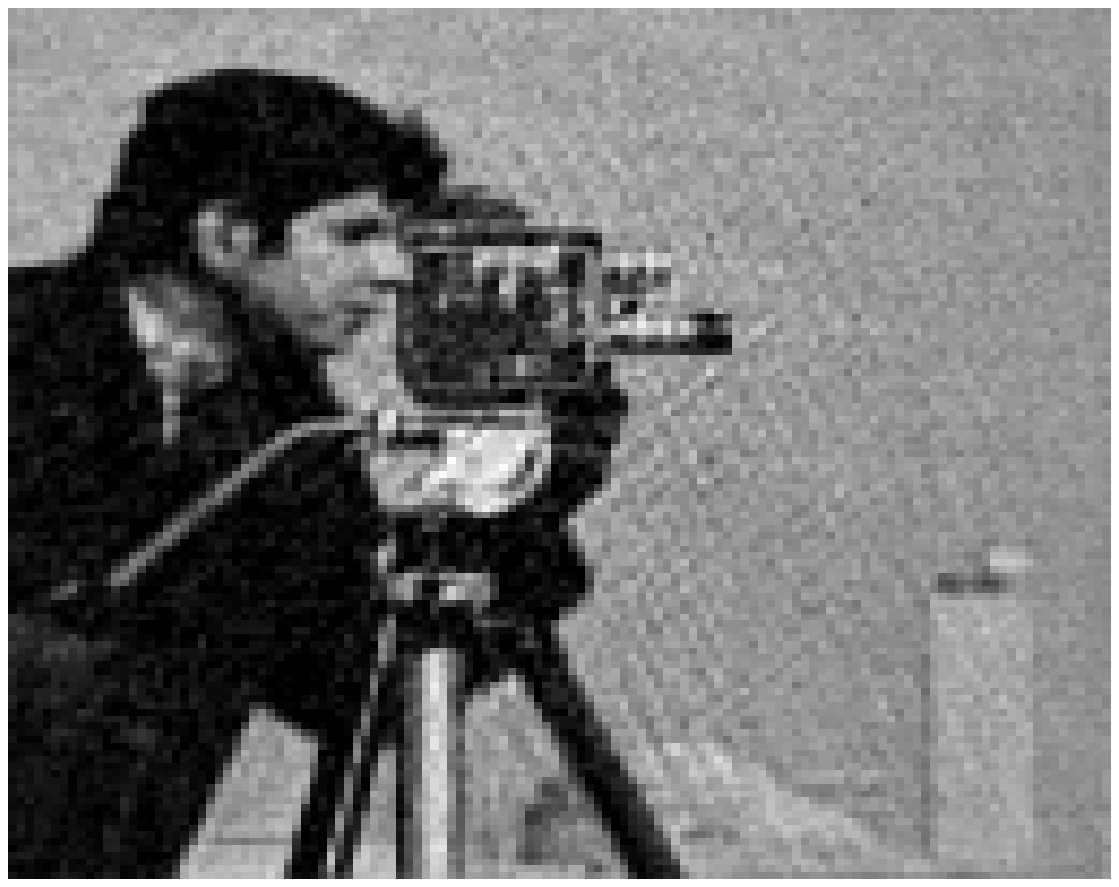}
     \caption{0.0948}
   \end{subfigure}
\begin{subfigure}[b]{.2\textwidth}
     \centering
     \includegraphics[width=\textwidth]{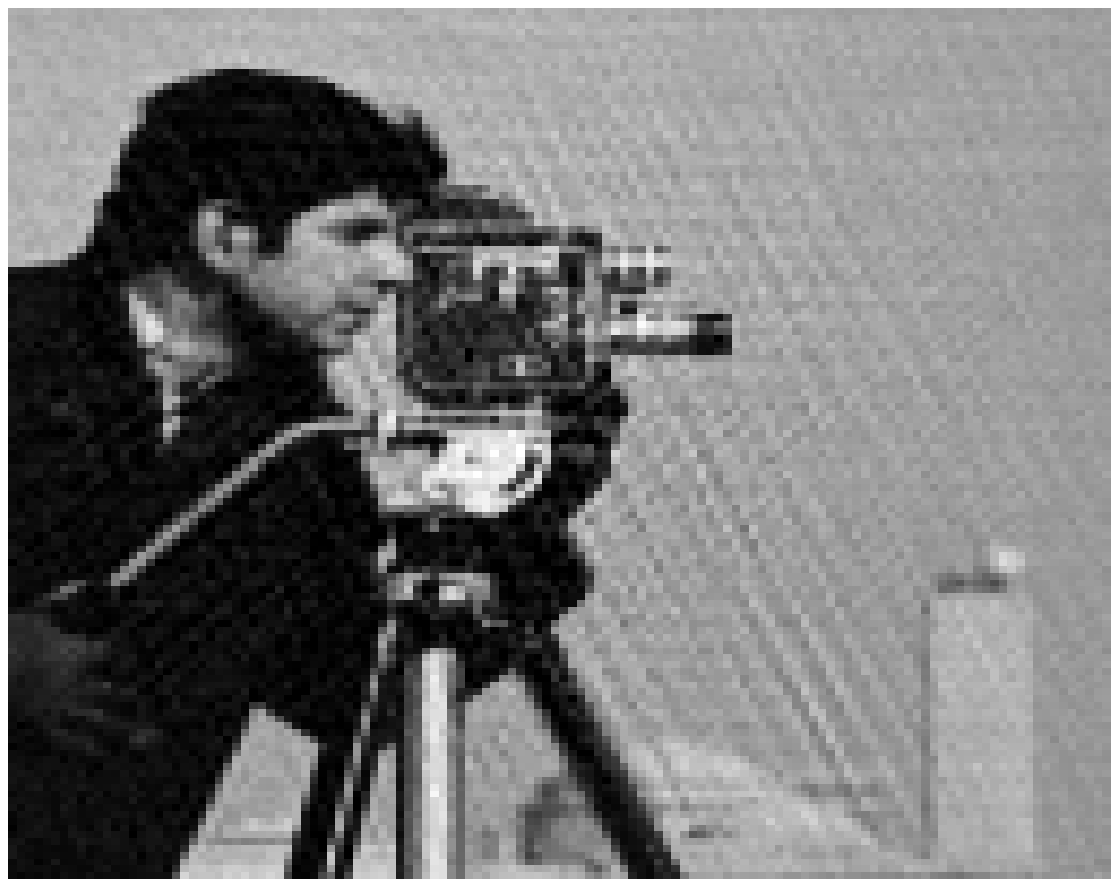}
     \caption{0.0686}
   \end{subfigure}\\
     \begin{subfigure}[b]{.2\textwidth}
     \centering
     \includegraphics[width=\textwidth]{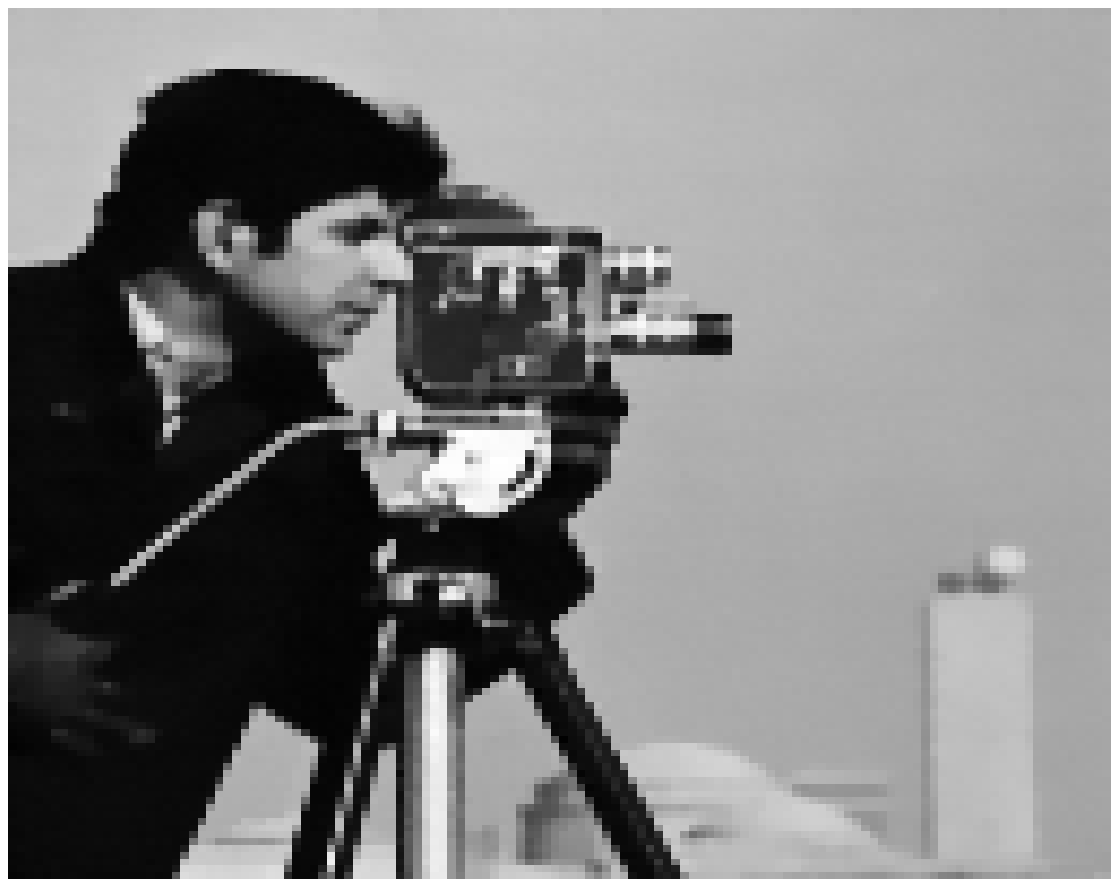}
     \caption{0.0265}
   \end{subfigure}
 \begin{subfigure}[b]{.2\textwidth}
     \centering
     \includegraphics[width=\textwidth]{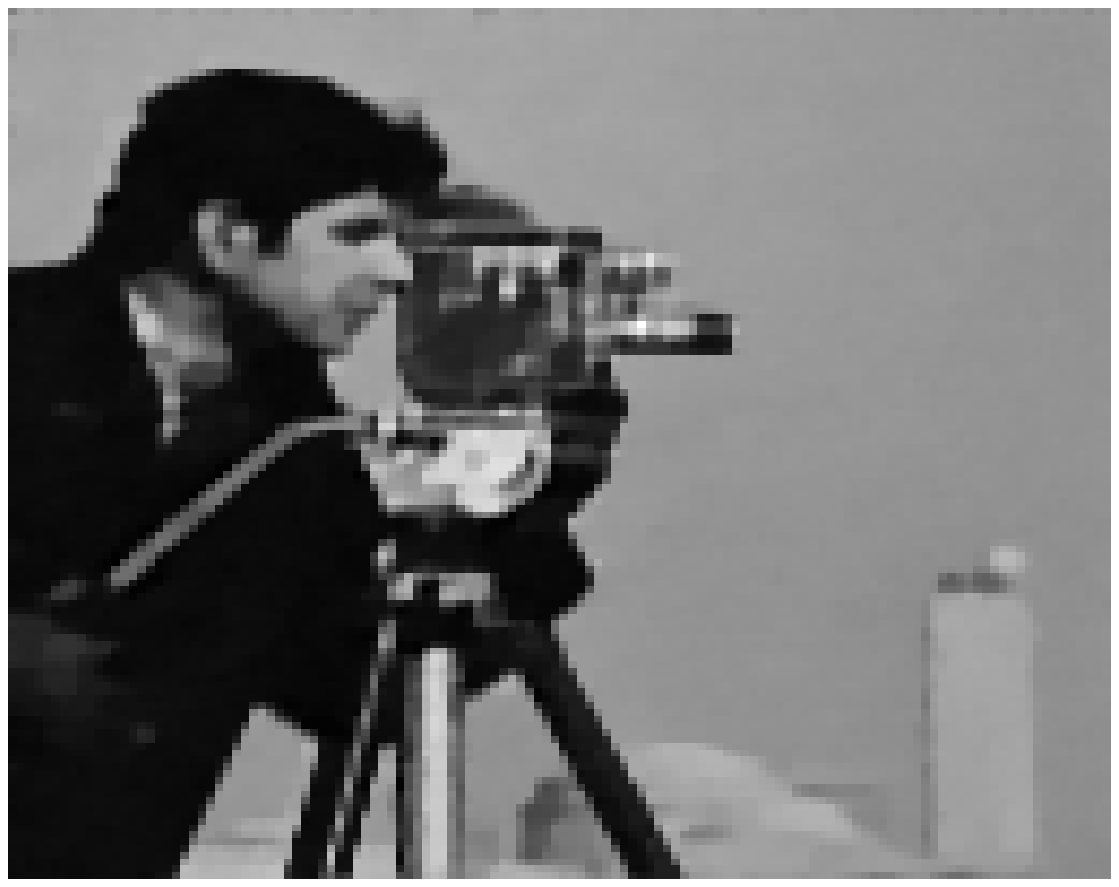}
     \caption{0.0665}
   \end{subfigure}
\begin{subfigure}[b]{.2\textwidth}
     \centering
     \includegraphics[width=\textwidth]{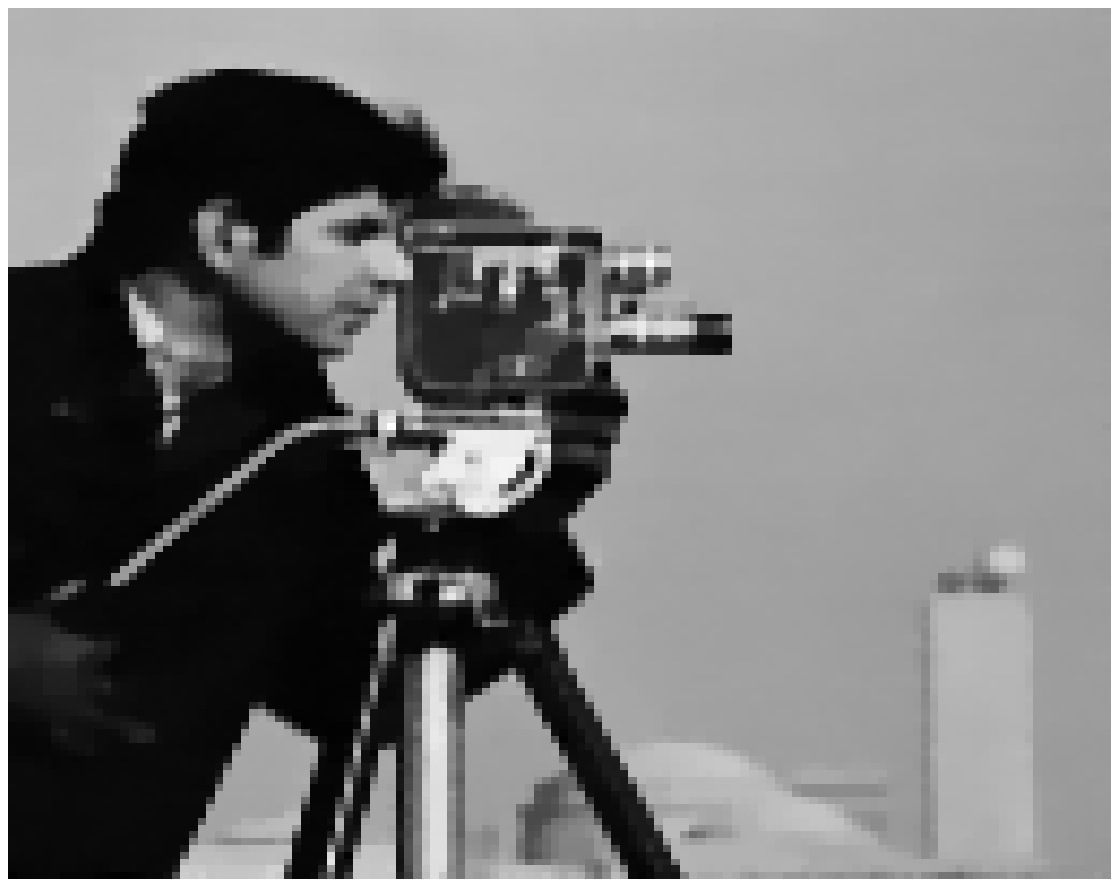}
     \caption{0.0381}
   \end{subfigure}
\begin{subfigure}[b]{.2\textwidth}
     \centering
     \includegraphics[width=\textwidth]{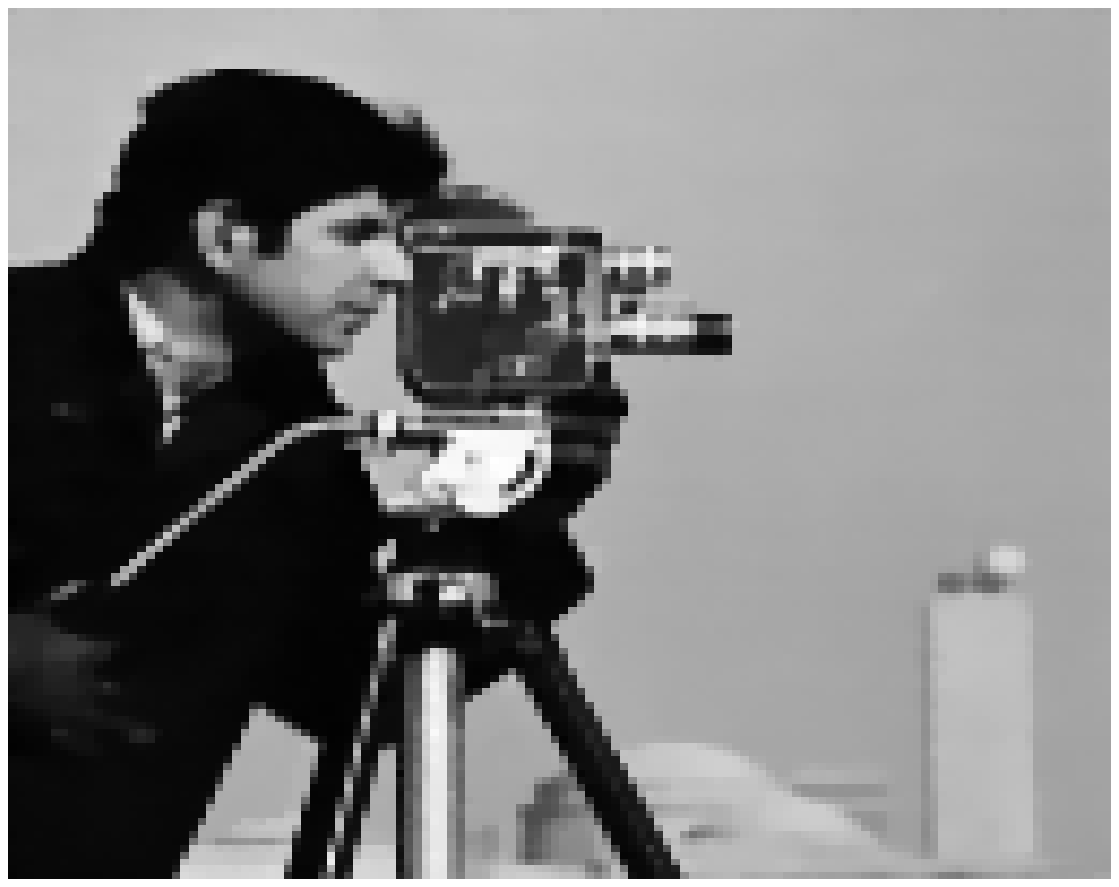}
     \caption{0.0309}
   \end{subfigure}\\
 \begin{subfigure}[b]{.2\textwidth}
     \includegraphics[width=\textwidth]{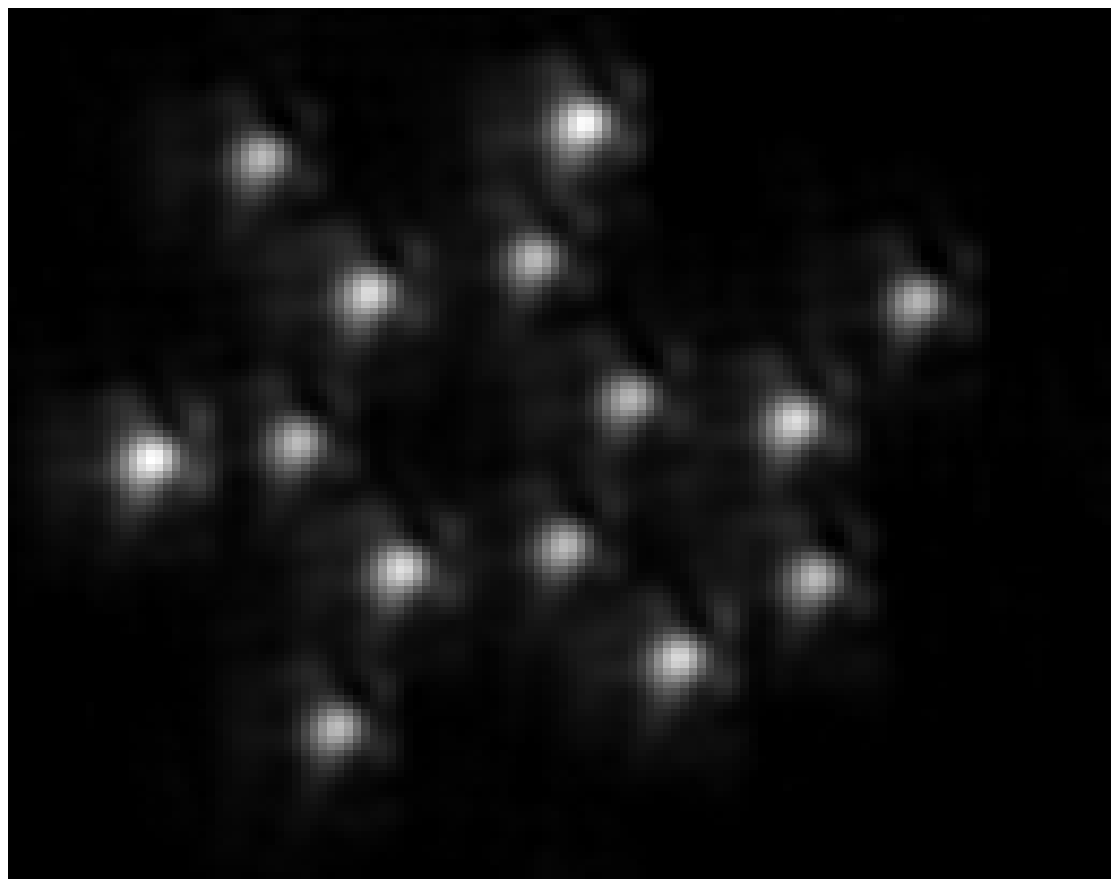}
     \caption{0.2237}
   \end{subfigure}
 \begin{subfigure}[b]{.2\textwidth}
     \centering
     \includegraphics[width=\textwidth]{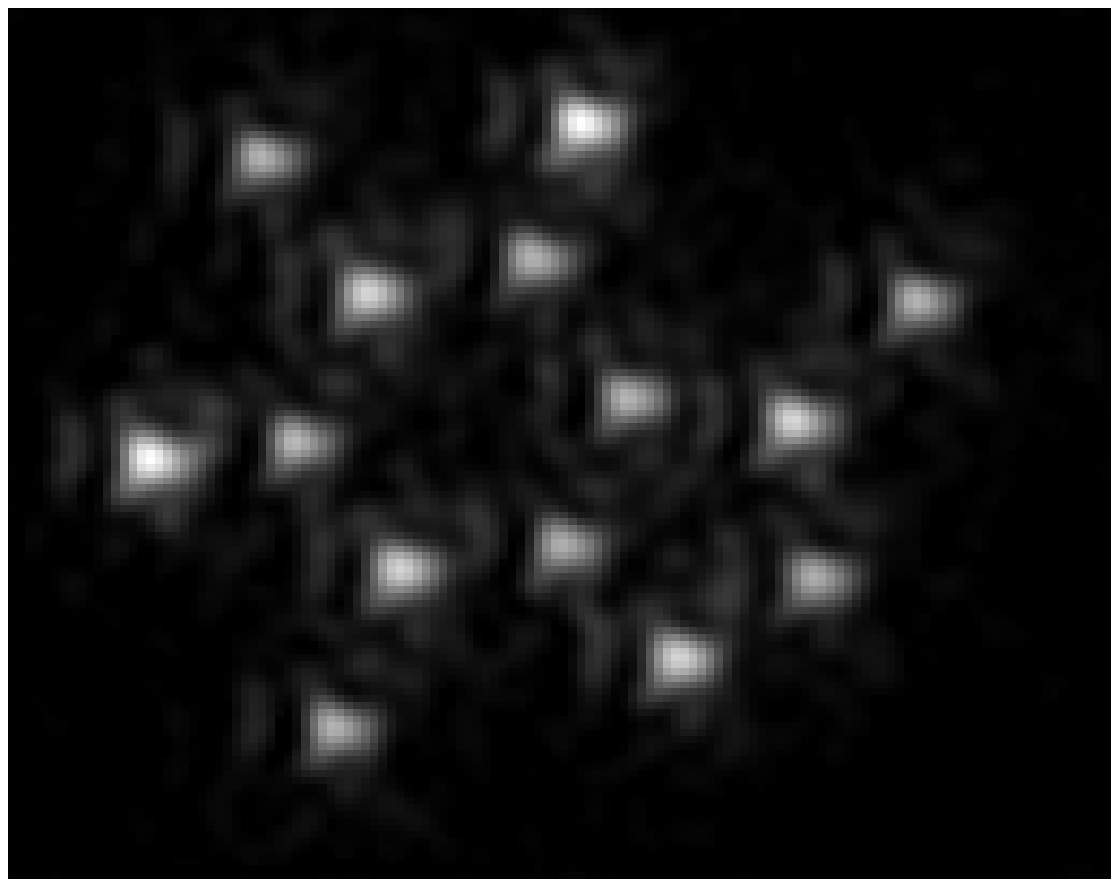}
     \caption{0.2624}
   \end{subfigure}
\begin{subfigure}[b]{.2\textwidth}
     \centering
     \includegraphics[width=\textwidth]{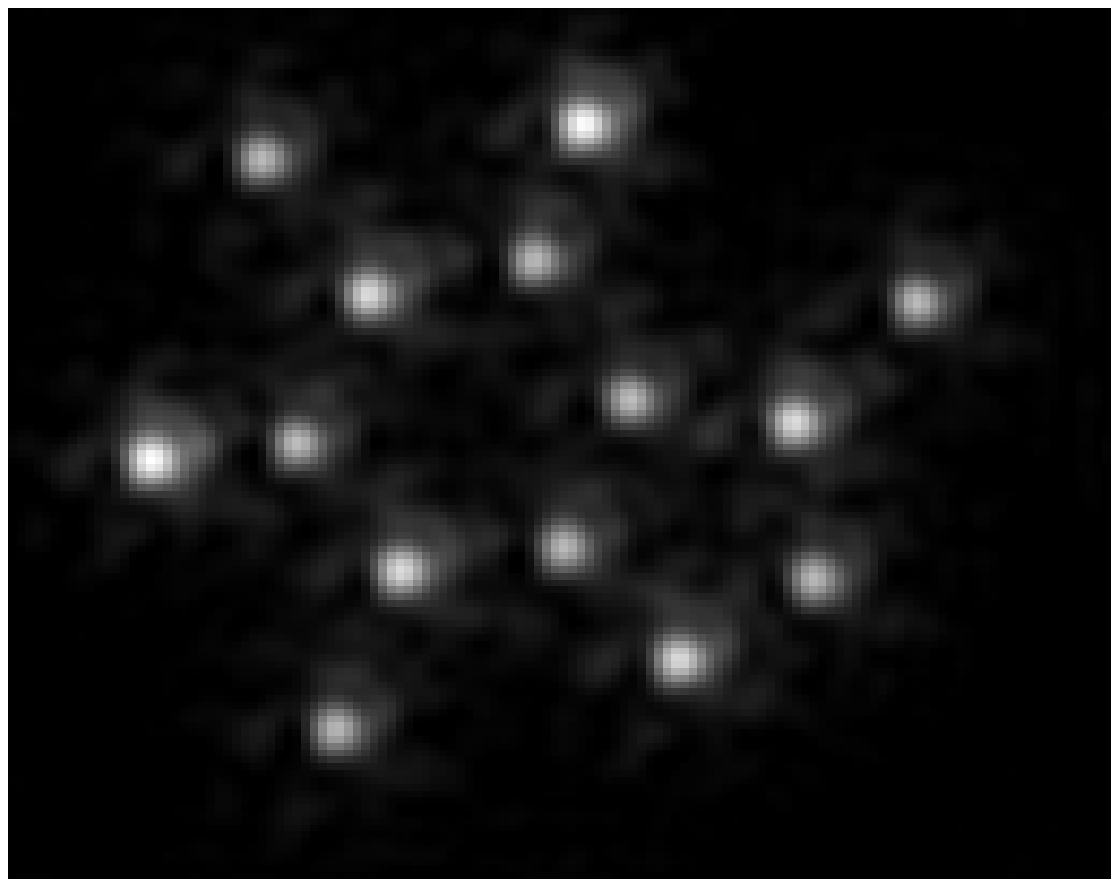}
     \caption{0.2223}
   \end{subfigure}
\begin{subfigure}[b]{.2\textwidth}
     \centering
     \includegraphics[width=\textwidth]{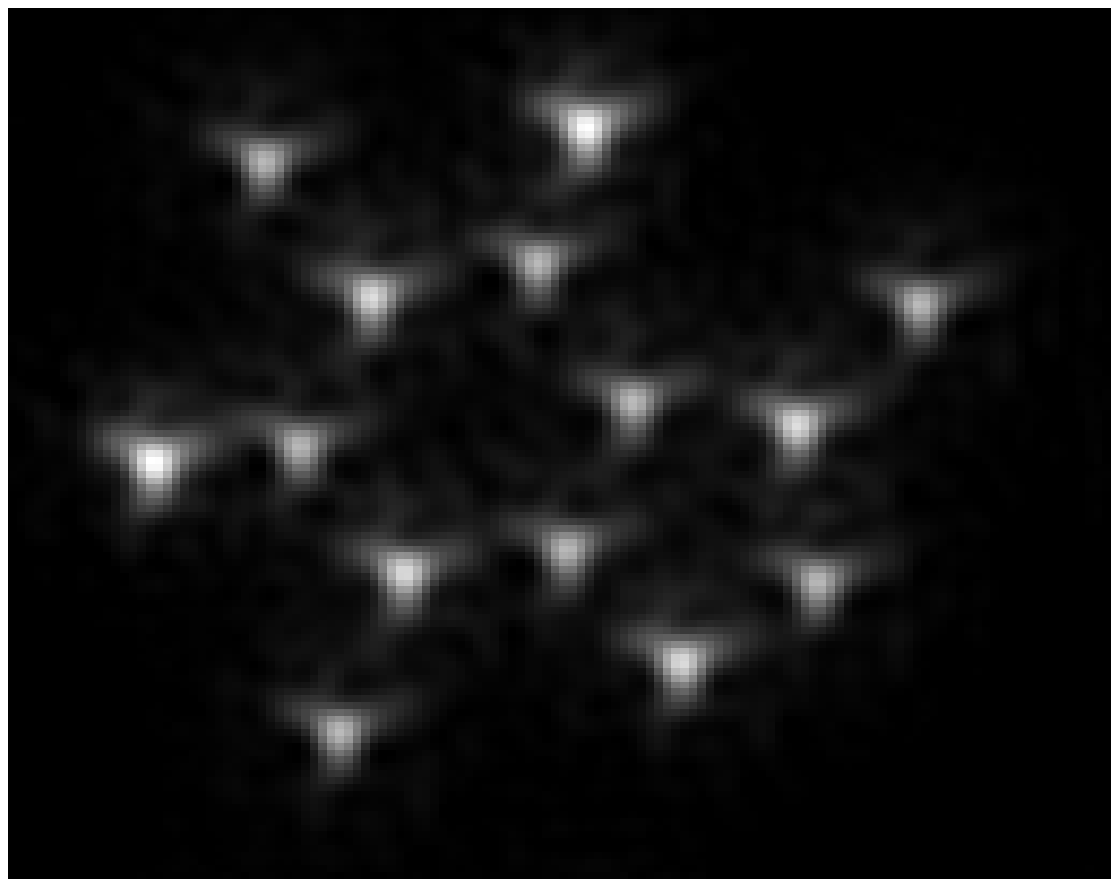}
     \caption{0.2588}
   \end{subfigure}\\
     \begin{subfigure}[b]{.2\textwidth}
     \centering
     \includegraphics[width=\textwidth]{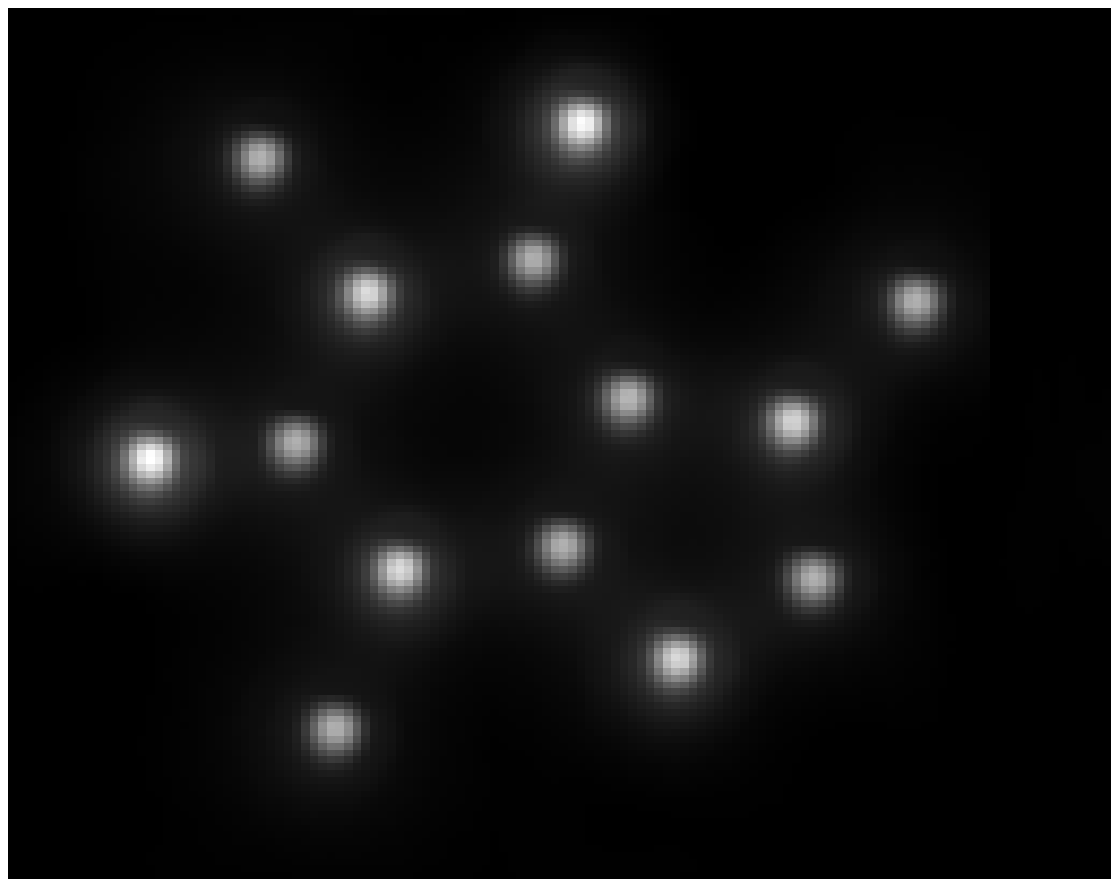}
     \caption{0.1380}
   \end{subfigure}
 \begin{subfigure}[b]{.2\textwidth}
     \centering
     \includegraphics[width=\textwidth]{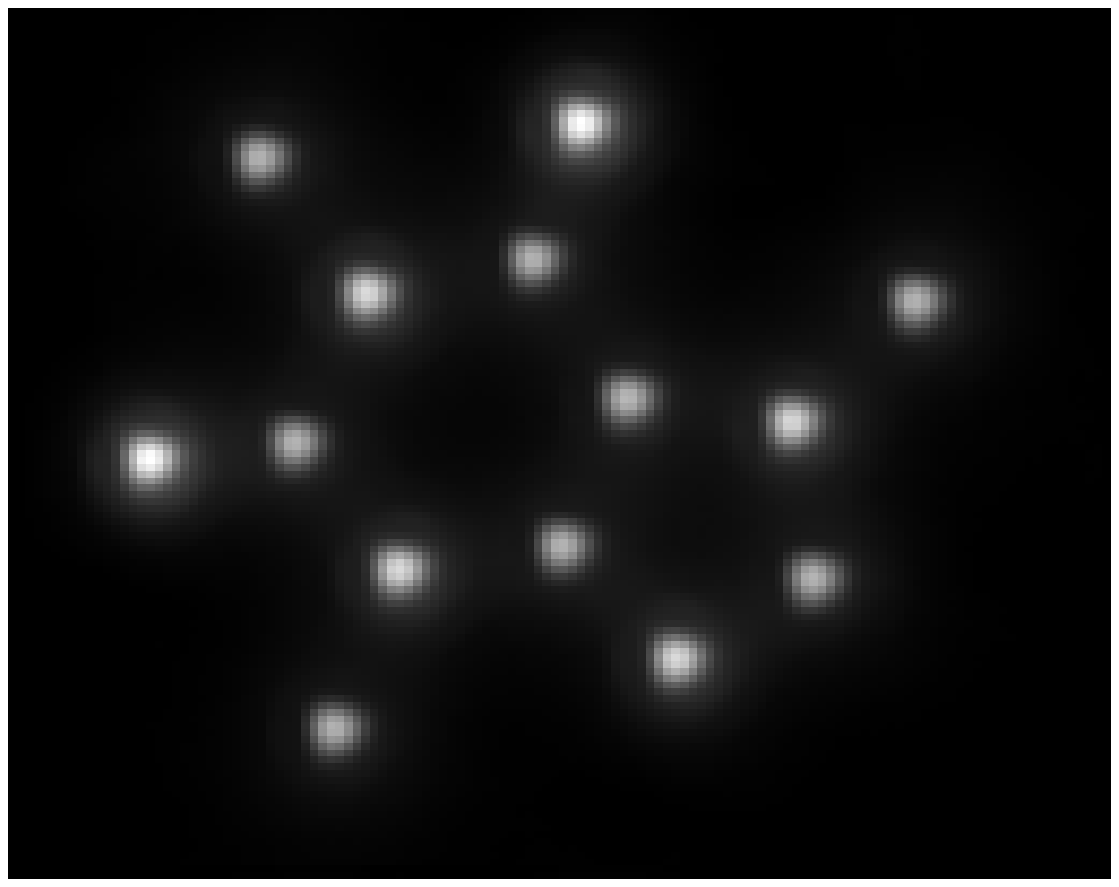}
     \caption{0.1051}
   \end{subfigure}
\begin{subfigure}[b]{.2\textwidth}
     \centering
     \includegraphics[width=\textwidth]{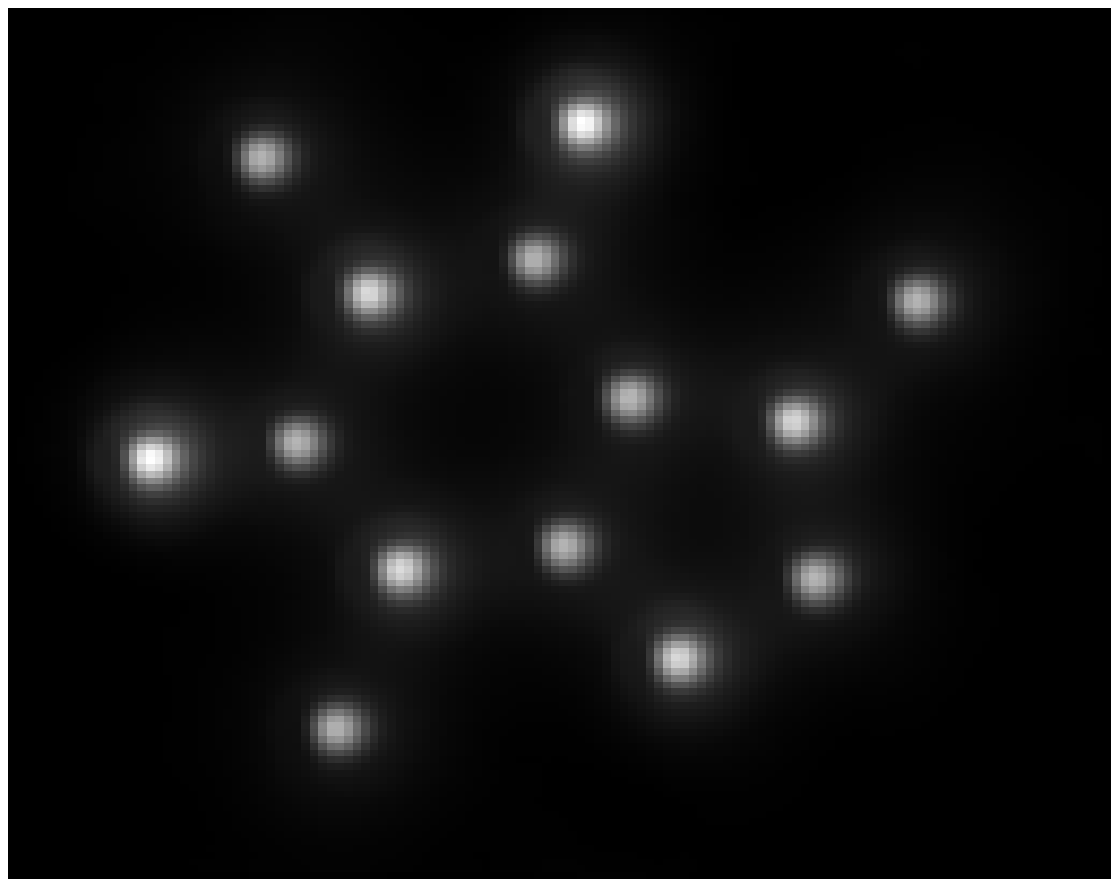}
     \caption{0.0925}
   \end{subfigure}
\begin{subfigure}[b]{.2\textwidth}
     \centering
     \includegraphics[width=\textwidth]{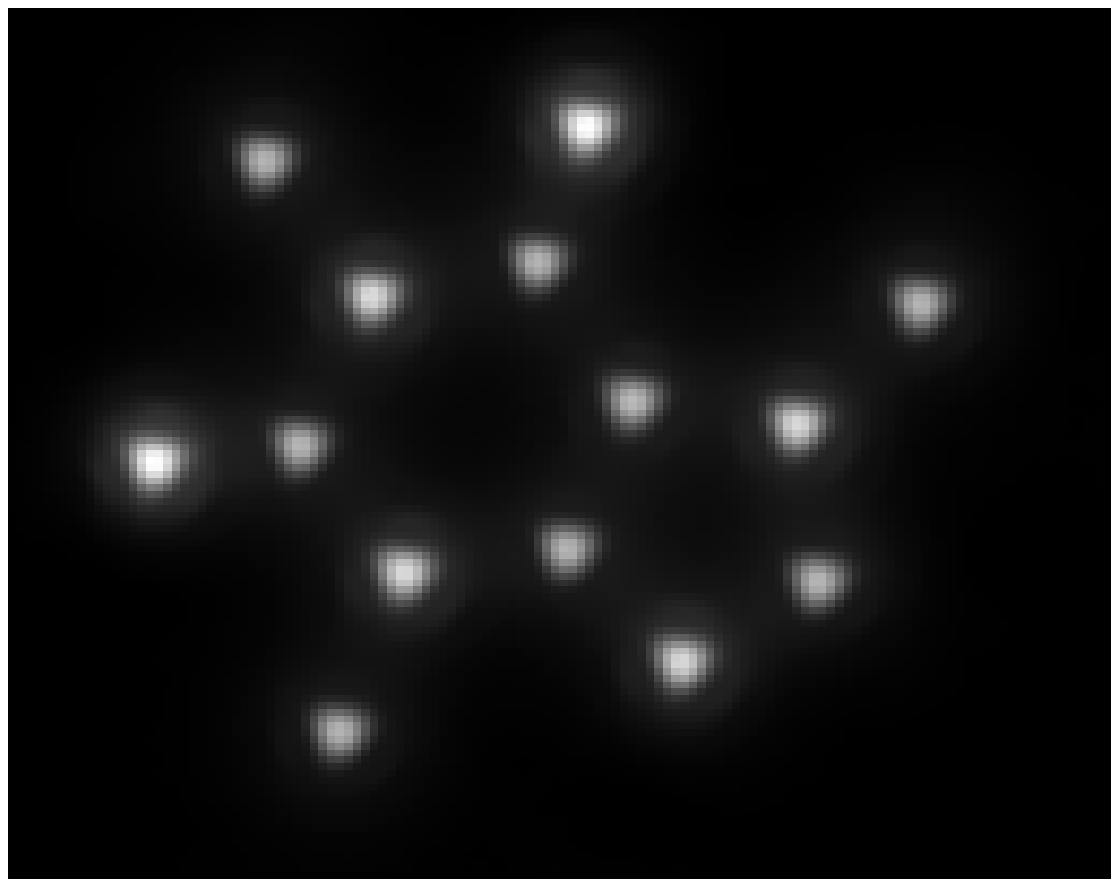}
     \caption{0.1382}
   \end{subfigure}
\caption{Reconstruction of Fourier phase retrieval by HIO (1st and 3rd row) and HIO+RGPS (2nd and 4th row). From left to right, the noise level SNRs are $\infty$ (noiseless), $30,40,50$ respectively.\label{fig:6}}
\end{figure}

\section{Conclusion}

We proposed a unified framework for phase retrieval with prior information via graph projection splitting (GPS) and robust GPS (RGPS). Current solvers only work for either isometric Fourier measurements with special prior information or general measurements without prior information, while our framework allows general measurement matrix and prior information simultaneously. GPS is motivated by the splitting formulation and variable-stacking. By introducing the splitting and graph projection, GPS can flexibly incorporate additional prior information about the solution and each resulting subproblem can be solved easily. RGPS and robust Douglas-Rachford (RDR) for phase retrieval without prior information for noisy measurements are also proposed. Advantages of GPS and RGPS over existing gradient flow-based methods include graph projection step and no line search. We show local convergence of GPS and RGPS for noiseless case without prior. For noisy case, we characterize the reconstruction error around the solution. 

For Gaussian phase retrieval without prior information, compared to other existing methods, GPS shows the sharpest phase transition and RGPS shows more stable reconstruction in various numerical experiments. RGPS outperforms GPS when the number of measurement is large enough. The performance of GPS and RGPS seem less dependent on the initialization than other gradient flow-based nonconvex solvers. RGPS also outperforms RAF for transmission measurement data especially when the number of measurements is small. It can also refine the reconstruction of HIO when TV regularization is added. The inclusion of TV regularization into oversampling Fourier phase retrieval is new and improves the reconstruction quality.

\section*{Acknowledgments}

JL was supported by China Postdoctoral Science Foundation grant No. 2017M620589 and National Natural Science Foundation of China grant No. 11801025. Hongkai Zhao would like to thank the summer visitor program at CSRC.


\begin{thebibliography}{10}

\bibitem{Candes2013}
Emmanuel~J Cand{\`e}s, Yonina~C Eldar, Thomas Strohmer, and Vladislav
  Voroninski.
\newblock Phase retrieval via matrix completion.
\newblock {\em {SIAM} Review}, 57(2):225--251, jan 2015.

\bibitem{Candes2014}
Emmanuel~J Cand{\`e}s, Xiaodong Li, and Mahdi Soltanolkotabi.
\newblock Phase retrieval from coded diffraction patterns.
\newblock {\em Applied and Computational Harmonic Analysis}, 39(2):277--299,
  oct 2015.

\bibitem{Candes2015}
Emmanuel~J Cand{\`e}s, Xiaodong Li, and Mahdi Soltanolkotabi.
\newblock Phase retrieval via wirtinger flow: Theory and algorithms.
\newblock {\em {IEEE} Transactions on Information Theory}, 61(4):1985--2007,
  apr 2015.

\bibitem{Candes2012}
Emmanuel~J Cand{\`e}s, Thomas Strohmer, and Vladislav Voroninski.
\newblock {PhaseLift}: Exact and stable signal recovery from magnitude
  measurements via convex programming.
\newblock {\em Communications on Pure and Applied Mathematics},
  66(8):1241--1274, nov 2012.

\bibitem{Chandra2017PhasePack}
Rohan Chandra, Ziyuan Zhong, Justin Hontz, Val McCulloch, Christoph Studer, and
  Tom Goldstein.
\newblock {PhasePack}: A phase retrieval library.
\newblock In {\em 2017 51st Asilomar Conference on Signals, Systems, and
  Computers}. {IEEE}, oct 2017.

\bibitem{Chen2015}
Pengwen Chen and Albert Fannjiang.
\newblock Fourier phase retrieval with a single mask by
  douglas{\textendash}rachford algorithms.
\newblock {\em Applied and Computational Harmonic Analysis}, 44(3):665--699,
  may 2018.

\bibitem{Chen2015b}
Yuxin Chen and Emmanuel~J. Cand{\`{e}}s.
\newblock Solving random quadratic systems of equations is nearly as easy as
  solving linear systems.
\newblock {\em Communications on Pure and Applied Mathematics}, 70(5):822--883,
  apr 2016.

\bibitem{Dhifallah2017Phase}
Oussama Dhifallah, Christos Thrampoulidis, and Yue~M. Lu.
\newblock Phase retrieval via linear programming: Fundamental limits and
  algorithmic improvements.
\newblock In {\em 2017 55th Annual Allerton Conference on Communication,
  Control, and Computing (Allerton)}. {IEEE}, oct 2017.

\bibitem{Fienup1982}
James~R Fienup.
\newblock Phase retrieval algorithms: A comparison.
\newblock {\em Applied Optics}, 21(15):2758--2769, 1982.

\bibitem{Hayes1982}
Monson~H Hayes.
\newblock The reconstruction of a multidimensional sequence from the phase or
  magnitude of its fourier transform.
\newblock {\em IEEE Transactions on Acoustics, Speech, and Signal Processing},
  30(2):140--154, apr 1982.

\bibitem{Kreutz-Delgado2009}
Ken Kreutz-Delgado.
\newblock The complex gradient operator and the cr-calculus.
\newblock Jun 2009.

\bibitem{Kuznetsova1988}
Tatiana~I Kuznetsova.
\newblock On the phase retrieval problem in optics.
\newblock {\em Soviet Physics Uspekhi}, 31(4):364, apr 1988.

\bibitem{Li2018}
Ji~Li, Jian-Feng Cai, and Hongkai Zhao.
\newblock Scalable incremental nonconvex optimization approach for phase
  retrieval, 2018.

\bibitem{Li162}
Ji~Li and Tie Zhou.
\newblock On relaxed averaged alternating reflections ({RAAR}) algorithm for
  phase retrieval with structured illumination.
\newblock {\em Inverse Problems}, 33(2):025012, jan 2017.

\bibitem{Luke2004}
D~Russell Luke.
\newblock Relaxed averaged alternating reflections for diffraction imaging.
\newblock {\em Inverse Problems}, 21(1):37--50, nov 2004.

\bibitem{Luke2003}
D~Russell Luke, Heinz~H Bauschke, and Patrick~L Combettes.
\newblock Hybrid projection{\textendash}reflection method for phase retrieval.
\newblock {\em Journal of the Optical Society of America A}, 20(6):1025--1034,
  2003.

\bibitem{Millane1990}
Rick~P Millane.
\newblock Phase retrieval in crystallography and optics.
\newblock {\em Journal of the Optical Society of America A}, 7(3):394--411,
  1990.

\bibitem{Misell1973}
D~L Misell.
\newblock A method for the solution of the phase problem in electron
  microscopy.
\newblock {\em Journal of Physics D: Applied Physics}, 6(1):L6, jan 1973.

\bibitem{Netrapalli2013}
Praneeth Netrapalli, Prateek Jain, and Sujay Sanghavi.
\newblock Phase retrieval using alternating minimization.
\newblock {\em {IEEE} Transactions on Signal Processing}, 63(18):4814--4826,
  sep 2015.

\bibitem{Shechtman2015}
Yoav Shechtman, Yonina~C Eldar, Oren Cohen, Henry~Nicholas Chapman, Jianwei
  Miao, and Mordechai Segev.
\newblock Phase retrieval with application to optical imaging: A contemporary
  overview.
\newblock {\em {IEEE} Signal Processing Magazine}, 32(3):87--109, may 2015.

\bibitem{Waldspurger2015}
Ir{\`e}ne Waldspurger, Alexandre {d’A}spremont, and St{\'e}phane Mallat.
\newblock Phase recovery, maxcut and complex semidefinite programming.
\newblock {\em Mathematical Programming}, 149(1--2):47--81, dec 2015.

\bibitem{Wang16}
Gang Wang, Georgios~B. Giannakis, and Yonina~C. Eldar.
\newblock Solving systems of random quadratic equations via truncated amplitude
  flow.
\newblock {\em {IEEE} Transactions on Information Theory}, 64(2):773--794, feb
  2018.

\bibitem{Wang2017Solving}
Gang Wang, Georgios~B. Giannakis, Yousef Saad, and Jie Chen.
\newblock Solving almost all systems of random quadratic equations.
\newblock 2017.

\bibitem{Wen2012Alternating}
Zaiwen Wen, Chao Yang, Xin Liu, and Stefano Marchesini.
\newblock Alternating direction methods for classical and ptychographic phase
  retrieval.
\newblock {\em Inverse Problems}, 28(11):115010, oct 2012.

\bibitem{Zhang2018}
Teng Zhang.
\newblock Phase retrieval by alternating minimization with random
  initialization, 2018.

\end{thebibliography}

\end{document}